\documentclass[11pt,oneside]{amsart}
\usepackage{fouriernc} 
\usepackage{mathtools}

\usepackage{STY-Definitions}    
\usepackage{STY-PageSetup}      
\usepackage{STY-Environments}   
\usepackage{comment}
\usepackage{soul} 

\usepackage{arydshln}
\usepackage{amsaddr}

\usepackage{comment}

\newcommand{\Sp}{Sp}

\usepackage{STY-text}
\usepackage{STY-circuit_diagrams}
\usepackage{subcaption}
\usepackage{mathtools}

\usepackage{ wasysym }

 \usepackage{relsize}

\newcommand{\trictop}{\text{\lightning}}
\newcommand{\trietop}{{\text{\reflectbox{\lightning}}}}
\newcommand{\tridetop}{\text{ \rotatebox[origin=c]{180}{\lightning}}}
\newcommand{\tridctop}{\text{ \rotatebox[origin=c]{180}{\reflectbox{\lightning}}}}

\newcommand{\tricto}[1]{\text{\lightning}}
\newcommand{\trieto}[1]{{\text{\reflectbox{\lightning}}}}
\newcommand{\trideto}[1]{\text{ \rotatebox[origin=c]{180}{\lightning}}}
\newcommand{\tridcto}[1]{\text{ \rotatebox[origin=c]{180}{\reflectbox{\lightning}}}}
\usepackage{stackengine}
\newcommand\mtiny[1]{\mbox{\tiny\ensuremath{#1}}}
\newcommand{\coo}[1]{\bottominset{\mtiny{#1}}{${ \mathlarger{\mathlarger{ \mathlarger{\sqsubset}}}}$}{1pt}{1pt}}

\newcommand{\evo}[1]{\bottominset{\mtiny{#1}}{$\mathlarger{\mathlarger{ \mathlarger{\sqsupset}}}$}{1pt}{-1pt}}

\newcommand{\serredual}[3]{{^{\coo{#2}}{#1}_{\evo{#3}}}}

\renewcommand{\includegraphics}[2][]{\fbox{}}

\ifdraft{

\usepackage{STY-DraftSetup}

} 
{}

\hypersetup{
 pdfkeywords={latex starter,template},
 pdfauthor={Niles Johnson},
}

\usepackage{tikz-cd} 
\usetikzlibrary{decorations.pathmorphing}
\usetikzlibrary{external}

\title{Riemann-Roch theorems in monoidal 2-categories}
\author[J. A. Campbell]{Jonathan A. Campbell}
\email{jonalfcam@gmail.com}
\address{Center for Communications Research, La Jolla} 
\author[K. Ponto]{Kate Ponto}
\email{kate.ponto@uky.edu}
\address{Department of Mathematics, University of Kentucky}
\date{\today}

\begin{document}

\begin{abstract}
Smooth and proper dg-algebras have an Euler class valued in the Hochschild homology of the algebra.  This Euler class is worthy of this name since it satisfies many familiar properties including compatibility with the familiar pairing on the Hochschild homology of the algebra and that of its opposite.  This compatibility is the Riemann-Roch theorems of \cite{shklyarov,petit}.

In this paper we prove a broad generalization of these Riemann-Roch theorems.  We generalize  from the bicategory of dg-algebras and their bimodules to symmetric monoidal bicategories and from Euler class to traces of non identity maps. 
Our generalization also implies spectral Riemann-Roch theorems.

We regard this result as an instantiation of a 2-dimensional generalized cobordism hypothesis.  This perspective draws the result close to many others that generalize results about Euler characteristics and classes
to bicategorical traces.
\end{abstract}
\maketitle
\setcounter{tocdepth}{1}
\tableofcontents

\section{Introduction}

The trace of a matrix is the first example of an additive invariant that is usually encountered in mathematics. Less familiar than additivity, though no less fundamental, is  the observation that trace is multiplicative: \[\tr(A \otimes B) = \tr(A) \tr(B)\] for square matrices $A, B$ of the same size. Euler characteristics and Euler classes are generalizations of the notion of trace to more general categories.  In this paper we prove a vast generalization to Riemann-Roch theorems that describe the compatibility between Euler classes and pairings, themselves generalizations of the multiplicativity of trace.

 For a case of our theorem that indicates the true generality but minimizes unfamiliar terminology, we consider  dg-algebras.  
If $A$ is a smooth and proper dg-$k$-algebra there are homomorphisms 
\[k\xto{}\HH(A)\otimes \HH(A^\op)\quad\HH(A)\otimes \HH(A^\op)\xto{}k\]  
that display $\hh(A^\op)$ as the dual (\S\ref{sec:duality_bicat}) of $\hh(A)$. 
Let $\evo{A}$ and $\coo{A}$ denote $A$ as an  $(A^\op\otimes A,k)$- and $(k,A\otimes A^\op)$-bimodule, respectively.  
Tensoring with the modules $\coo{A}$ and $\evo{B}$ defines an isomorphism between the categories of $(A,B)$-bimodules and $(B^\op,A^\op)$-bimodules.  For an $(A,B)$-bimodule $M$, let $^\sqsubset M_\sqsupset$ be the corresponding $(B^\op,A^\op)$-bimodule and use the same notation for bimodule homomorphisms. 

\begin{thm}[\cref{thm:pairing_var_1}\ref{item:pairing_var_1}, compare to \cref{thm:shklyarov}]\label{intro:shklyarov}
 Suppose $A$ and $B$ are smooth and proper dg-$k$-algebras.
 If $M$ is an $(A,B)$-bimodule that is finitely generated and projective as an $A$-module, $N$ is a $(B,A)$-bimodule that is finitely generated and projective as a $B$-module, and $f\colon M\to M$ and $g\colon N\to N$ are homomorphisms, then 
 the trace of
 \[
 \begin{tikzcd}
\HH(A; M \otimes_B N) \ar{rr}{\HH(A;f \otimes_B g)} && \HH(A; M \otimes_B N) 
 \end{tikzcd}
 \]
 coincides with the map 
\begin{equation}\label{eq:shklyarov}
k\xto{}\HH(A)\otimes \HH(A^\op)\xto{\tr(f)\otimes \tr(^\sqsubset g_\sqsupset)}\HH(B)\otimes \HH(B^\op)\xto{}k.
\end{equation}
\end{thm}

The main results of \cite{shklyarov,petit} are special case of \cref{intro:shklyarov} where $A=k$ and $f$ and $g$ are identity maps.

\begin{rmk}As noted above, this is merely a special case of our main result.  In \cref{sec:serre_duality} we show this theorem holds for symmetric monoidal bicategories.    
Because spectral categories play the role of algebras in a particular monoidal bicategory
(see \cite{cp2}), \cref{intro:shklyarov} also holds for spectra and topological Hochschild homology ($\thh$). 
\end{rmk}

Though this statement, its generalizations \cref{thm:petit,thm:pairing_var_1,thm:shklyarov}, and predecessors \cite{shklyarov,petit} appear complicated, we show in \cref{sec:path} that
when phrased in the correct generality,  they become much more transparent.
This understanding is motivated by graphical reasoning that interprets \cref{intro:shklyarov,thm:petit,thm:pairing_var_1,thm:shklyarov}   
as results about toroidal traces, in the sense of \cite{cp2}. In the
pictures in \cref{fig:toroidal_no_framing}, the colored regions correspond to the algebras,
and the bimodules are to be read as living at the boundaries of the regions.
The torus in \cref{fig:toroidal_no_framing_1} is the trace of $\HH(f\otimes_B g)$
and the torus in \cref{fig:toroidal_no_framing_2} is the composite in \eqref{eq:shklyarov}.

\begin{figure}
\begin{subfigure}[b]{0.20\textwidth}
\resizebox{\textwidth}{!}{\begin{tikzpicture}

\newcommand\vshift{2}
\newcommand\hshift{2}

\coordinate (A) at (0,5*\vshift);
\coordinate (B) at (-2*\hshift,3*\vshift);
\coordinate (C) at (2*\hshift,3*\vshift);
\coordinate (D) at (2*\hshift,-1*\vshift);
\coordinate (E) at (0*\hshift,-3*\vshift);
\coordinate (F) at (-2*\hshift,-1*\vshift);
\coordinate (G) at (0*\hshift,4*\vshift);
\coordinate (H) at (-1*\hshift,3*\vshift);
\coordinate (I) at (3*\hshift,-1*\vshift);
\coordinate (J) at (0*\hshift,-4*\vshift);
\coordinate (K) at (-3*\hshift,-1*\vshift);
\coordinate (L) at (1*\hshift,3*\vshift);
\coordinate (M) at (0*\hshift,6*\vshift);
\coordinate (N) at (-3*\hshift,3*\vshift);

\coordinate (O) at (1*\hshift,-1*\vshift);
\coordinate (P) at (0*\hshift,-2*\vshift);

\coordinate (Q) at (-1*\hshift,-1*\vshift);

\coordinate (R) at (3*\hshift,3*\vshift);
\coordinate (S) at (0*\hshift,2*\vshift);
\coordinate (T) at (1.3*\hshift,1*\vshift);
\coordinate (U) at (-1.3*\hshift,1*\vshift);
\coordinate (V) at (0*\hshift,0*\vshift);
\coordinate (W) at (0*\hshift,1*\vshift);

\draw[fill=\et,color=\et] 
	(M)to  [looseness=.8,out=180, in=90] 
	(N)to  [looseness=.8,out=-90, in=135]
	(U)to  [looseness=.8,out=-45, in=135]
	(V)to  [looseness=.8,out=-45, in=90]
	(O)to  [looseness=.8,out=-90, in=0]
	(P)--
	(E)to  [looseness=.8,out=0, in=-90]
	(D)to  [looseness=.8,out=90, in=-45]
	(W)to  [looseness=.8,out=135, in=-90]
	(B)to  [looseness=.8,out=90, in=180]
	(A)--(M);

\draw[fill=\ep,color=\ep] 
	(A)to  [looseness=.8,out=180, in=90] 
	(B)to  [looseness=.8,out=-90, in=135]
	(W)to  [looseness=.8,out=-45, in=90]
	(D)to  [looseness=.8,out=-90, in=0]
	(E)--
	(J)to  [looseness=.8,out=0, in=-90]
	(I)to  [looseness=.8,out=90, in=-45]
	(T)to  [looseness=.8,out=135, in=-45]
	(S)to  [looseness=.8,out=135, in=-90]
	(H)to  [looseness=.8,out=90, in=180]
	(G)--(A);

\draw[fill=\et,color=\et] 
	(M)to  [looseness=.8,out=0, in=90]
	(R)to  [looseness=.8,out=-90, in=45] 
	(T)to  [looseness=.8,out=-135, in=45] 
	(V)to  [looseness=.8,out=-135, in=90] 
	(Q)to  [looseness=.8,out=-90, in=180]
	(P)--
	(E)to  [looseness=.8,out=180, in=-90]
	(F)to  [looseness=.8,out=90, in=-135]
	(W)to  [looseness=.8,out=45, in=90] 
	(C)to  [looseness=.8,out=90, in=0]
	(A)--(M);

\draw[fill=\ep,color=\ep]
	(A)to  [looseness=.8,out=0, in=90]
	(C)to  [looseness=.8,out=-90, in=45] 
	(W)to  [looseness=.8,out=-135, in=90] 
	(F)to  [looseness=.8,out=-90, in=180]
	(E)--
	(J)to  [looseness=.8,out=180, in=-90]
	(K)to  [looseness=.8,out=90, in=-135]
	(U)to  [looseness=.8,out=45, in=-135]
	(S)to  [looseness=.8,out=45, in=-90]
	(L)to  [looseness=.8,out=90, in=0]
	(G)--(A);

\draw
	(J)to  [looseness=.8,out=180, in=-90]
	(K)to  [looseness=.8,out=90, in=-135]
	(U)to  [looseness=.8,out=45, in=-135]
	(S)to  [looseness=.8,out=45, in=-90]
	(L)to  [looseness=.8,out=90, in=0]
	(G);

\draw
	(M)to  [looseness=.8,out=0, in=90]
	(R)to  [looseness=.8,out=-90, in=45] 
	(T)to  [looseness=.8,out=-135, in=45] 
	(V)to  [looseness=.8,out=-135, in=90] 
	(Q)to  [looseness=.8,out=-90, in=180]
	(P);

\draw 
	(V)to  [looseness=.8,out=-45, in=90]
	(O)to  [looseness=.8,out=-90, in=0]
	(P);

\draw
	(M)to  [looseness=.8,out=180, in=90] 
	(N)to  [looseness=.8,out=-90, in=135]
	(U);

\draw
	(J)to  [looseness=.8,out=0, in=-90]
	(I)to  [looseness=.8,out=90, in=-45]
	(T);
\draw
	(S)to  [looseness=.8,out=135, in=-90]
	(H)to  [looseness=.8,out=90, in=180]
	(G);

\draw[dotted,line width=0.5mm] (N)to[looseness = .5, out =-90, in =-90] (H);
\draw[dotted,line width=0.15mm](H)to[looseness = .5, out =90, in =90] (N);

\draw[dotted,line width=0.5mm] (R)to[looseness = .5, out =-90, in =-90] (L);
\draw[dotted,line width=0.15mm] (L)to[looseness = .5, out =90, in =90] (R);

\draw[dotted,line width=0.15mm]  (K)to[looseness = .5, out =90, in =90] (Q);
\draw[dotted,line width=0.5mm] (Q)to[looseness = .5, out =-90, in =-90] (K);
\draw[dotted,line width=0.15mm] (O)to[looseness = .5, out =90, in =90] (I);
\draw[dotted,line width=0.5mm](I)to[looseness = .5, out =-90, in =-90] (O);

\end{tikzpicture}}
\caption{$\tr(\HH(f\otimes_B g))$}\label{fig:toroidal_no_framing_1}
\end{subfigure}
\hspace{.1\textwidth}
\begin{subfigure}[b]{0.20\textwidth}
\resizebox{\textwidth}{!}{\begin{tikzpicture}

\newcommand\vshift{2}
\newcommand\hshift{2}

\coordinate (A) at (-3*\hshift,3*\vshift);
\coordinate (B) at (-2*\hshift,1*\vshift);
\coordinate (C) at (-1*\hshift,2.25*\vshift);
\coordinate (D) at (-4*\hshift,1.75*\vshift);
\coordinate (E) at (0*\hshift,7.5*\vshift);
\coordinate (F) at (0*\hshift,-6.5*\vshift);
\coordinate (G) at (2*\hshift,3*\vshift);
\coordinate (H) at (3*\hshift,1*\vshift);
\coordinate (I) at (1*\hshift,1.75*\vshift);
\coordinate (J) at (4*\hshift,2.25*\vshift);
\coordinate (K) at (0*\hshift,5*\vshift);
\coordinate (L) at (0*\hshift,-4*\vshift);
\coordinate (M) at (-4*\hshift,4*\vshift);
\coordinate (N) at (-4*\hshift,-3*\vshift);
\coordinate (O) at (-1*\hshift,4*\vshift);
\coordinate (P) at (-1*\hshift,-3*\vshift);
\coordinate (Q) at (4*\hshift,4*\vshift);
\coordinate (R) at (4*\hshift,-3*\vshift);
\coordinate (S) at (1*\hshift,4*\vshift);
\coordinate (T) at (1*\hshift,-3*\vshift);
\coordinate (U) at (-2*\hshift,-.75*\vshift);
\coordinate (V) at (2*\hshift,-.75*\vshift);
\coordinate (W) at (0*\hshift,.5*\vshift);
\coordinate (X) at (0*\hshift,-2*\vshift);

\draw[fill=\et] 
	(E)to  [looseness=1,out=180, in=90] 
	(M)to  [looseness=.8,out=-90, in=90] 
	(D)to  [looseness=.8,out=0, in=180] 
	(A)to  [looseness=.8,out=0, in=180] 
	(B)to  [looseness=.8,out=0, in=180] 
	(C)to  [looseness=.8,out=90, in=-90] 
	(O)to  [looseness=.8,out=90, in=180] 
	(K)to  [looseness=.8,out=0, in=90] 
	(S)to  [looseness=.8,out=-90, in=90] 
	(I)to  [looseness=.8,out=0, in=180] 
	(G)to  [looseness=.8,out=0, in=180] 
	(H)to  [looseness=.8,out=0, in=180] 
	(J)to  [looseness=.8,out=90, in=-90]
	(Q)to  [looseness=1,out=90, in=0] 
	(E);

\draw[fill=\ep,color=\ep] 
	(D)to  [looseness=.8,out=0, in=180] 
	(A)to  [looseness=.8,out=0, in=180] 
	(B)to  [looseness=.8,out=0, in=180] 
	(C)to  [looseness=.8,out=-90, in=135] 
	(W)to  [looseness=.8,out=-45, in=135] 
	(V)to  [looseness=.8,out=-45, in=90] 
	(R)to  [looseness=1,out=-90, in=0]  
	(F)to  [looseness=1,out=180, in=-90]
	(N)to  [looseness=.8,out=90, in=-135] 
	(U)to  [looseness=.8,out=45, in=-135] 
	(W)to  [looseness=.8,out=45, in=-90] 
	(I)to  [looseness=.8,out=0, in=180] 
	(G)to  [looseness=.8,out=0, in=180] 
	(H)to  [looseness=.8,out=0, in=180] 
	(J)to  [looseness=.8,out=-90, in=45]
	(V)to  [looseness=.8,out=-135, in=45]
	(X)to  [looseness=.8,out=-135, in=90]  
	(P)to  [looseness=.8,out=-90, in=180] 
	(L)to  [looseness=.8,out=0, in=-90]
	(T)to  [looseness=.8,out=90, in=-45]
	(X)to  [looseness=.8,out=135, in=-45]  
	(U)to  [looseness=.8,out=135, in=-90]  
(D)
;

\draw
	(C)to  [looseness=.8,out=-90, in=135] 
	(W)to  [looseness=.8,out=-135, in=45] 
	(U)to  [looseness=.8,out=135, in=-90]  
(D)
;

\draw
	(J)to  [looseness=.8,out=-90, in=45]
	(V)to  [looseness=.8,out=-135, in=45]
	(X)to  [looseness=.8,out=-135, in=90]  
	(P)to  [looseness=.8,out=-90, in=180] 
	(L)to  [looseness=.8,out=0, in=-90]
	(T)to  [looseness=.8,out=90, in=-45]
	(X)
;

\draw
	(W)to  [looseness=.8,out=45, in=-90] 
	(I)
;

\draw
	(V)to  [looseness=.8,out=-45, in=90] 
	(R)to  [looseness=1,out=-90, in=0]  
	(F)to  [looseness=1,out=180, in=-90]
	(N)to  [looseness=.8,out=90, in=-135] 
	(U)
;

\draw[dotted,line width=0.5mm] (M)to[looseness = .5, out =-90, in =-90] (O);
\draw[dotted,line width=0.15mm](O)to[looseness = .5, out =90, in =90] (M);
\draw[dotted,line width=0.5mm] (S)to[looseness = .5, out =-90, in =-90] (Q);
\draw[dotted,line width=0.15mm](Q)to[looseness = .5, out =90, in =90] (S);
\draw[dotted,line width=0.5mm] (N)to[looseness = .5, out =-90, in =-90] (P);
\draw[dotted,line width=0.15mm](P)to[looseness = .5, out =90, in =90] (N);
\draw[dotted,line width=0.5mm] (T)to[looseness = .5, out =-90, in =-90] (R);
\draw[dotted,line width=0.15mm](R)to[looseness = .5, out =90, in =90] (T);

\end{tikzpicture}}
\caption{\eqref{eq:shklyarov}}\label{fig:toroidal_no_framing_2}
\end{subfigure}
\caption{Toroidal traces without framing.  Compare to \cref{fig:toroidal_framing}.}\label{fig:toroidal_no_framing}
\end{figure}
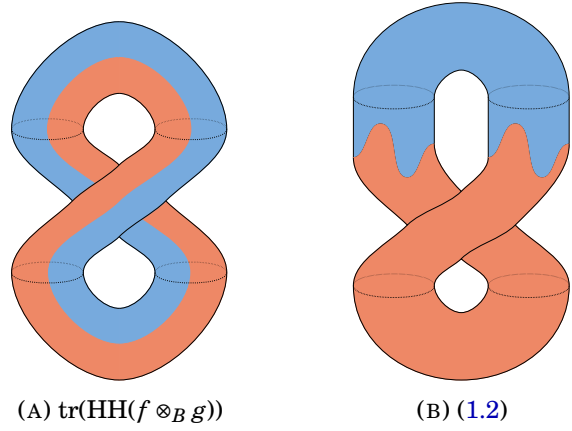

Indeed, one can see the same thing happening on both tori --- the only
reason that the theorem is not \textit{immediately} obvious
pictorially is the embedding of the torus in 3-space destroys the
obvious symmetry between its two constituent copies of $S^1$.

\begin{rmk}
If one were willing to prove full coherence theorems for this kind of
calculus and/or prove a kind of 2-dimensional cobordism hypothesis
\cite{schommer-pries, lurie_cobordism} that took bimodules into
account, this kind of graphical reasoning would actually
\textit{prove} the theorem. However, this kind of framed, bimodular
cobordism hypothesis is overkill for the task at hand, so we content
ourselves with translating the pictures into category theory, and
verifying the category theory directly. 
\end{rmk}

\subsection{Organization}

We begin in \cref{sec:duality_bicat} with a review of duality and bicategories and the formalism of shadows.  In \cref{sec:path} we outline the path to the main theorem --- the proof is quite formal and superficially does not much resemble the statements in the introduction. Nevertheless, it is this relatively innocuous reformulation which will imply the results in the introduction. 

\cref{sec:monoidal_bicat} is a review of the details of monoidal bicategories that we will need. 

\cref{sec:1_dualizable,sec:2_dualizable} are discussions of various flavors of duality that exist in monoidal bicategories, with the main content in \cref{sec:2_dualizable}. 

In \cref{sec:serre_duality} we prove the main result in categorical generality and in \cref{sec:motivation} we  relate the results in the literature to our statement of the theorem. 

\subsection{Acknowledgments} The second author was partially supported by NSF grants DMS-1810779 and DMS-2052905.  Many thanks to John Lind for his careful reading and helpful suggestions and to Nick Gurski for sharing his insights on coherence.  We also thank the referee for detailed comments that strengthened this paper.

\section{Duality in bicategories}\label{sec:duality_bicat}
Let $V$ be a finite dimensional vector space over a field $k$. A dual vector space $V^\ast$ can be defined 
without reference to the internal hom: a dual is specified by a unit $\epsilon: V^\ast \otimes V \to k$, a counit, $\eta: k \to V \otimes V^\ast$, and the fact that the following two composites are the identity map
\[
\begin{tikzcd}
  V \ar{r}{\eta \otimes \id} & V \otimes V^\ast \otimes V \ar{r}{\id \otimes \epsilon} & V
\end{tikzcd}
\qquad 
\begin{tikzcd}
  V^\ast \ar{r}{\id \otimes \eta} & V^\ast \otimes V \otimes V^\ast \ar{r}{\epsilon \otimes \id} & V^\ast 
\end{tikzcd}
\]
Concretely, the maps $\eta, \epsilon$ can be defined in terms of a basis $\{e_i\}^n_{i=1}$ by $\eta(1) = \sum e_i \otimes e^\ast_i$, $\epsilon(e^\ast_j \otimes e_i) = \delta_{ij}$ and extending by linearity.

A dual immediately gives very explicit ways of extracting invariants of vector spaces. For instance, the composite
\[
\begin{tikzcd}
k \ar{r}{\eta} & V\otimes V^\ast \ar{r}{\cong} & V^\ast \otimes V \ar{r}{\epsilon} & k 
\end{tikzcd}
\]
is multiplication by $\dim(V)$. For an endomorphism $f: V\to V$ the composite 
\[
\begin{tikzcd}
k \ar{r}{\eta} & V\otimes V^\ast \ar{r}{f \otimes \id} & V \otimes V^\ast \ar{r}{\cong} & V^\ast \otimes V \ar{r}{\epsilon} & k 
\end{tikzcd}
\]
is multiplication by $\tr(f)$. This also has the benefit that dimension and trace can be seen as \textit{morphisms}, and thus purely internal to the category.

Since the notion of dual, dimension, and trace can be expressed within a category with tensor product and symmetry, we can define all of these notions in a symmetric monoidal category $(\mc{C}, \otimes, \mbf{1})$ using the above definitions. The ``dimension'' of a dualizable object $C \in \mc{C}$ becomes the element $\chi(C) \in \hom_{\mc{C}}(\mbf{1}, \mbf{1})$ given by 
\[
\begin{tikzcd}
  \mbf{1} \ar{r} & C \otimes C^\ast \ar{r}{\cong} & C^\ast \otimes C \ar{r} & \mbf{1} 
\end{tikzcd}
\]
and similarly for the trace of an endomorphism $f: C \to C$.

While extracting invariants of objects required a symmetry isomorphism, the definition of duality did not. Thus, the definition of duality applies equally well in  a monoidal category. Monoidal categories can be viewed as a degenerate version of a bicategory\footnote{The weakness of a bicategory is necessary for the examples of interest here.} --- a monoidal category is a bicategory with one object, the objects of the monoidal category correspond to morphisms in a bicategory, and the tensor product structure corresponds to composition. 
With this context,  we think of a bicategory as a  monoidal category with many objects and the motivating example is the bicategory of rings, bimodules and bimodule maps.  We find  \cite{leinster} to be a good source for definitions of bicategories.

\begin{notn}  We denote the bicategorical composition in a bicategory $\mc{B}$ by $\odot$.  If $A$ is an object of $\mc{B}$ we denote the identity 1-cell for $A$ by $U_A$.  
\end{notn}

\begin{example}\label{ex:all_the_ex}\hfill 
\begin{enumerate}
\item\label{it:morita} In the category of bimodules, the unit 1-cell associated to a  ring $A$ is $A$ regarded as an $A-A$ bimodule and 
$\odot$ is the tensor product.
\item\label{it:morita_R}  For a commutative ring $R$, there is a  bicategory $\mathrm{Alg}(R)$ whose
objects are R-algebras, 
1-morphisms are algebra bimodules, and 
2-morphism are bimodule homomorphisms. 
 The bicategory composition $\odot$ is tensor product over the shared algebra.  

\item\label{it:prof} There is a  bicategory of  categories and profunctors (bimodules) enriched over a cocomplete closed symmetric monoidal category $\mathcal{V}$.  The bicategory composition $\odot$ is a coend. 

\item\label{it:dg_morita} There is a   bicategory of dg-algebras, and the derived category of bimodules and homomorphisms.   The bicategorical product is the tensor product.
\item\label{it:dg_prof} There is a   bicategory of dg-categories, where 0-cells are dg-categories and 1- and 2-cells are the derived category of bimodules.  We will abuse notation and also denote this category $\Mod(\Cat_{\dg})$.
\item\label{it:spec_prof} There is a   bicategory of spectral categories where 1- and 2-cells are the homotopy category of bimodules. We will  denote this category $\Mod(\Cat_{\Sp})$.
\end{enumerate}

See \cite[Thm. 8.39]{cp2} for an extended and explicit description of \ref{it:dg_morita} to \ref{it:spec_prof}.

Example \ref{it:morita} is a special case of example \ref{it:morita_R}  where $R=\mathbb{Z}$.  Example \ref{it:morita_R}  is a subbicategory of  example \ref{it:prof}  where we only consider the categories with a single object.  (The same for examples \ref{it:dg_morita}  and \ref{it:dg_prof}.) Finally there is  a spectral analog of example \ref{it:dg_morita}   where the objects are replaced by ring spectra. 
\end{example}

In this paper we will focus almost exclusively on bicategories that generalize the bicategory of bimodules.  This reflects the particular objective of this paper and not the underlying theory.  See \cite{ps:mult} for more topologically focused discussions. 

The following definition first appeared in \cite{may_sigurdson}.

\begin{defn}\label{defn:dualizable}
  Let $M$ be a 1-cell in a bicategory $\mc{B}(C, D)$. We say $M$ is \textbf{right dualizable} if there is a 1-cell $N$ together with 2-cells
  \[
  \eta\colon U_C \to M \odot N \qquad \epsilon\colon N \odot M \to U_D 
  \]
  such that the triangle identities \eqref{eq:triangle_id} hold. 
  \begin{equation}\label{eq:triangle_id}
  \begin{tikzcd}
    M \ar{r}{\cong} & U_C \odot M \ar{r}{\eta \odot \id}  & M \odot N \odot M\ar{r}{\id \odot \epsilon} & M \odot U_D \ar{r}{\cong} & M
\\[-15pt] 
    N\ar{r}{\cong} & N\odot  U_C \ar{r}{\id \odot \eta}  & N\odot M \odot N \ar{r}{\epsilon \odot \id} & U_D \odot  N\ar{r}{\cong} & N
  \end{tikzcd}
  \end{equation}
We say $N$ is \textbf{right dual} to $M$, $(M,N)$ is a \textbf{dual pair}, that $N$ is \textbf{left dualizable}, and that $M$ is its \textbf{left dual}.
\end{defn}

\begin{example}The examples in \cref{ex:all_the_ex} all have many dualizable 1-cells and they are all variations on projectivity. Those in examples \ref{it:prof}, \ref{it:dg_prof}, and \ref{it:spec_prof}, 
go under the name of \textbf{smooth, proper} categories and are 
 many object generalizations of finitely generated projective modules. For a beautiful introduction to this topic, see \cite{toen}. 
See also \cite{blumberg_gepner_tabuada} for a careful discussion of dualizability in categories. 
\end{example}

The following lemma is easy, but critical. 

\begin{lem}[Compare to \cref{lem:one_dual_composites}, {\cite[16.5.1]{may_sigurdson}}]\label{lem:duals_compose}
If $M_1 \in \mc{B}(A, B)$ and $M_2 \in \mc{B}(B, C)$ are right dualizable, then so is $M_1 \odot M_2$. 
\end{lem}

\begin{defn}\label{defn_shadow}\cite{p:thesis}
A \textbf{shadow} for a bicategory $\mc{B}$ consists of functors
 \[\sh{-}\colon \mc{B}(R, R)\to \mathbf{T}\]
for each object $R$ of $\mc{B}$ and some fixed category $\mathbf{T}$, equipped with a natural isomorphism
\[\theta\colon  \sh{ M \odot  N}\to \sh{N \odot M }\]
for $M \in \mc{B}( R, S)$  and $N \in \mc{B}( S,  R)$ such that the following diagrams commute whenever they make sense:
  \[\xymatrix@C=12pt{
    \sh{(M\odot N)\odot P} \ar[r]^\theta \ar[d]_{\sh{a}} &
    \sh{P \odot (M\odot N)} \ar[r]^{\sh{a}} &
    \sh{(P\odot M) \odot N}
    \\
    \sh{M\odot (N\odot P)} \ar[r]^\theta 
    & \sh{(N\odot P)\odot M} \ar[r]^{\sh{a}} 
   & \sh{N\odot (P\odot M)}\ar[u]_\theta 
  }\hfill \xymatrix@C=12pt{
    \sh{M\odot U_C} \ar[r]^\theta \ar[dr]_{\sh{r}} &
    \sh{U_C\odot M} \ar[d]^{\sh{l}} \ar[r]^\theta &
    \sh{M\odot U_C} \ar[dl]^{\sh{r}}
    \\
    &\sh{M}}\]
\end{defn}

\begin{example}\label{ex:shadow}
The bicategories in \cref{ex:all_the_ex} all have shadows and they can be interpreted as the bicategorical product with the diagonal module associated to a 0-cell.  In the bicategory of rings and modules this is Hochschild homology and in the spectral generalizations it is topological Hochschild homology.
\end{example}

Recent work of Hess and Rasekh \cite{hess_rasekh} has shown there is an equivalence between functors out of THH of a bicategory and shadows on that bicategory so the examples in \cref{ex:shadow} capture those we need to consider.

\begin{defn}\label{defn:twisted_trace}\cite{p:thesis}
  Let $\phi\colon P \odot M \to M \odot Q$ be a 2-cell where $M$ is right dualizable. The {\bf twisted trace} of $\phi$ is  the composite
  \[
  \sh{P} \cong \sh{P \odot U_A} \to \sh{P \odot M \odot N} \to \sh{M \odot Q \odot N} \cong \sh{N \odot M \odot Q} \to \sh{U_B \odot Q} \cong \sh{Q} 
  \]
\end{defn}
If $\phi$ is the identity map $M\to M$ we call the trace of $\phi$ the {\bf Euler characteristic} of $M$.  This is consistent with the Euler characteristic of spaces regarded as elements of the stable homotopy category.

\begin{example}\label{ex:all_the_ex_trace_1}
  Let $M$ be a right dualizable $(A, B)$-bimodule in the bicategory of rings and bimodules. Then the Euler characteristic is a map
  \[
  \begin{tikzcd}
    \chi(A)\colon \sh{U_A} = \hh(A) \ar{r} & \hh(B)  = \sh{U_B}
  \end{tikzcd}
  \]

\label{ex:all_the_ex_trace_2}
As a more explicit example of this,  
 let $V$ be a finite dimensional complex representation of a finite group $G$. Then $V$ is a $(\mathbb{C}[G], \mathbb{C})$-bimodule which is right dualizable. The induced map
  \[
  \begin{tikzcd}
    \hh_0 (\mathbb{C}[G]) \ar{r} & \hh_0 (\mathbb{C}) 
  \end{tikzcd}
  \]
  is the character of the representation. 
\end{example}

We end this section with another easy, but critical result that builds on \cref{lem:duals_compose}.

\begin{thm}[Compare to \cref{lem:trace_otimes_compatible}, {\cite[7.5]{ps:bicat}}]\label{thm:composite}
  Let $M_1 \in \mc{B}(A, B)$, $M_2 \in \mc{B}(B, C)$ be right dualizable and $Q_1 \in \mc{B}(A,A)$, $Q_2 \in \mc{B}(B, B)$ and $Q_3 \in \mc{B}(C, C)$. Let $f_1\colon Q_1\odot M_1\to M_1\odot Q_2$ and $f_2\colon Q_2\odot M_2\to M_2\odot Q_3$ be 2-cells. Then the trace of
  \[
  Q_1 \odot M_1 \odot M_2 \xrightarrow{f_1 \odot \id_{M_2}} M_1 \odot Q_2 \odot M_2 \xrightarrow{\id_{M_1} \odot f_2} M_1 \odot M_2 \odot Q_3
  \]
  is
  \[
  \sh{Q_1} \xrightarrow{\tr(f_1)} \sh{Q_2} \xrightarrow{\tr(f_2)} \sh{Q_3}. 
  \]
\end{thm}

\section{Path to the main theorem}\label{sec:path}
In this section we use the traces  of \cref{sec:duality_bicat} to state, but not yet prove, a bicategorical generalization of \cref{intro:shklyarov}.  This generalization requires hypotheses beyond a shadow on a bicategory and in this section we describe the additional conditions.  
When stating a theorem, there is often a choice between hypotheses that are easy to verify or state and those that  are most compatible with the proof technique.  The main theorem of this paper appears to be a case where there is a significant discrepancy between these two goals.  We resolve that tension by focusing on precisely what our proof requires in this section.  In \cref{sec:1_dualizable,sec:2_dualizable,sec:serre_duality} we verify that symmetric  monoidal bicategories with many dualizable 0-cells (\cref{sec:1_dualizable,sec:2_dualizable}) and, in particular, subbicategories of \cref{ex:all_the_ex}  satisfy the conditions described in this section.

\cref{ex:all_the_ex_trace_1,ex:shadow} begin the process of generalizing \cref{intro:shklyarov} by recognizing the Hochschild homology groups as shadows and the Euler or Hochschild class as a trace.  The next step is to appropriately interpret the pairing map.    A non degenerate pairing of $R$-modules for a commutative ring $R$ is a map 
\[M\otimes_R N\to L\] 
so that the adjoint $M\to \Hom(N, L)$ 
is an isomorphism.  
For pairings that take values in the ground ring $R$, this is a pair of modules $(M,N)$ and an isomorphism $M\to \Hom(N,R)$.

We will replace this assumption by the assumption that $(M,N)$ is a dual pair in a  monoidal category  $(\sC,\otimes, I)$.  
If $\sC$ is closed and $(M,N)$ is a dual pair,  there is an isomorphism
\[N\to \Hom(M,I).\] 
In the category of modules over a ring $R$, dualizability is not the same as a nondegenerate pairing, but working within the hypotheses of  \cref{intro:shklyarov} we don't lose any generality.  With this change, we gain an easy generalization of pairing to other  monidal categories and the coevaluations will be essential to our generalization of \cref{intro:shklyarov}.

Then the first additional condition on a bicategory with shadow is the following:
\begin{goal}[\cref{sec:1_dualizable}]\label{goal:smc}
The shadow on $\sB$ takes values in a  monoidal category.
\end{goal}

The Euler classes and Hochschild classes of \cref{intro:shklyarov} are traces in bicategories and so we think of them as maps $\sh{U_A}\to \sh{U_B}$ for 0-cells $A$ and $B$ of a bicategory $\sB$.  (In the motivating examples $U_A$ and $\sh{U_A}$ are the ground field.) 
Since $\sh{U_A}$ and $\sh{U_B}$ have to be compatible with \cref{goal:smc} we have the following condition:
\begin{goal}[\cref{sec:2_dualizable}]\label{goal:dualizable}
For each 0-cell $A$ of $\sB$, $\sh{U_A}$ is dualizable.  
\end{goal}

In \cref{sec:2_dualizable} we will produce dualizability for shadows of more 1-cells, but this is an artifact of the process and not essential for the most immediate generalizations of the result in the introduction.  It would be interesting to know if there are further generalizations that make use of the additional dualizability.  

The final goal almost slides under the radar since it is such a fundamental and familiar fact about modules.  It is the observation that 
an $(A,B)$-bimodule can be understood as an $(B^{\op},A^{\op})$-bimodule. 
In the generality here it is the following statement.

\begin{goal}[\cref{sec:serre_duality}]\label{goal:dualizable_functor}
  There is a  function $\zdual{-}$ from 0-cells of $\sB$ to itself that so that the dual of $\sh{U_A}$ is $\sh{U_{\zdual{A}}}$ and this function extends to a functor 
\[
^\sqsubset -_\sqsupset\colon \sB(A,B)\to \sB(\zdual{B}, \zdual{A})\]
that takes right dualizable 1-cells to right dualizable 1-cells. 
\end{goal}

\cref{goal:smc,goal:dualizable,goal:dualizable_functor} don't together imply \cref{thm:goal}, but they make it possible to understand the statement.

\begin{goal}[\cref{sec:euler_pairing,sec:serre_duality}]\label{thm:goal}
Suppose $M\in \sB(A,B)$ and  $N\in \sB(B,A)$ are right dualizable 1-cells. For 2-cells  $f\colon M\to M$ and $g\colon N\to N$  the following diagram commutes.
\[\xymatrix@C=50pt{\sh{U_A}\otimes \sh{U_{\zdual{A}}}\ar[r]^-{\tr(f)\otimes \tr(^\sqsubset g_\sqsupset
)}&\sh{U_B}\otimes \sh{U_{\zdual{B}}}\ar[d]
\\
I\ar[r]^-{\tr(f\odot g)}\ar[u]&I}
\]
The left vertical map is the coevaluation for the dual pair $(\sh{U_A},\sh{U_{\zdual{A}}})$ and the right vertical map is the evaluation for the dual pair $(\sh{U_B},\sh{U_{\zdual{B}}})$. 
\end{goal}

\begin{example}
\cref{intro:shklyarov} is a special case of this result under some simplifying assumptions. 
First suppose that   $\sh{U_A}\cong \sh{U_{\zdual{A}}}\cong I$ and the coevaluation for $I$ is the unit isomorphism.  Then the diagram in \cref{thm:goal} becomes the following diagram.
\[\xymatrix@C=50pt{I\otimes I\ar[r]^-{\tr(f)\otimes \tr(^\sqsubset g_\sqsupset
)}&\sh{U_B}\otimes \sh{U_{\zdual{B}}}\ar[d]
\\
I\ar[r]^-{\tr(f\odot g)}\ar[u]&I}
\]
If the target of the shadow is the category of $k$-vector spaces for a field $k$ then $I=k$.   Taking Hochschild homology to be the shadow, the traces pick out Hochschild homology classes.  Then the right vertical map is a nondegenerate pairing in the more conventional sense.  The maps from $I=k$ pick out elements in their targets.
\end{example}

\section{Monoidal bicategories}\label{sec:monoidal_bicat}
Our main theorem requires working in the generality of symmetric monoidal bicategories. In this section we review the necessary definitions and establish our graphical notation for manipulations. While the formal definition of a monoidal bicategory is unwieldy, the basic intuition, first examples and graphical descriptions are illuminating.  We consider those first before recalling the formal definitions.

Intuitively, a monoidal bicategory is a bicategory $\sB$ equipped with a functor 
\[\otimes \colon \sB\times \sB\to \sB\]
that is appropriately compatible with the bicategorical composition $\odot$ and is unital and associative.  A symmetric monoidal bicategory gains 1-cells $A\otimes B\to B\otimes A$ that satisfy the triangle identities for an adjoint pair.  See \cref{fig:dualizable_1_cell_symmetry}.

For motivating examples of monoidal bicategories we return to the motivating example of the previous section.

\begin{example}The examples in \cref{ex:all_the_ex} are (symmetric) monoidal bicategories.  See \cite{shulman_symmetric_2,shulman_symmetric,cp2}.
\begin{itemize}
\item For a commutative ring $R$, the  monoidal product $\otimes$  in the bicategory $\mathrm{Alg}(R)$ is tensor product over $R$.

\item The monoidal product in the  bicategory of  categories and profunctors (bimodules) enriched over a cocomplete closed symmetric monoidal category $\mathcal{V}$
is induced by the monoidal product on $\mathcal{V}$.
\end{itemize}
See \cite[Thm. 8.39]{cp2}  for more details.
\end{example}

\subsection{Circuit diagrams}\label{sec:circuit_diagram}

Symmetric monoidal bicategories have an impressive amount of structure and keeping track of all it is burdensome.  Fortunately, Gurski and Osorno \cite{gurski_osorno} proved a coherence theorem for symmetric monoidal bicategories 
that massively simplifies bookkeeping challenges.

\begin{thm} \cite[Theorem 1.23]{gurski_osorno}  In the free symmetric monoidal bicategory on a single object, every diagram of 2-cells commutes. Equivalently, between every pair of parallel 1-cells there is either a unique
invertible 2-cell or no 2-cells at all. Moreover, parallel 1-cells are isomorphic if and only if they
have the same underlying permutation.
\end{thm}

This alleviates the need to record associativity and unit maps, but leaves us with the challenge of keeping track of the bicategorical compositions $\odot$ and the monoidal products $\otimes$ of 1-cells.  We will resolve this by writing compositions graphically using the conventions of \cite{cp2}.  They are closely related to those in \cite{douglas_schommer_pries_snyder, schommer-pries,ps:indexed}.

We represent 0-cells by black horizontal lines with parallel colored shading.  
\begin{equation}
\begin{tikzpicture}
\ocalt{y}{}{Y}{4.5}{0}{0}{\en}
\ocalt{x}{}{X}{0}{0}{0}{\en}

  \begin{scope}[on background layer]
\alp{ ({x}t)--({y}t)}
\end{scope}
\end{tikzpicture}
\end{equation}
The default is that the shading is below the black lines, but in later sections we will flip the colored shading to the top to indicate the 1-dual of a 0-cell (see \cref{defn:1_dualizable}).   Vertically stacked shaded lines indicate the $\otimes$ product.

One cells are represented by colored boxes with any number of shaded edges entering on the left and right sides.  (Shaded edges do not  enter on the top and bottom.)  
\begin{equation}
\resizebox{.15\textwidth}{!}{
\begin{tikzpicture}
\ocalt{ta}{}{X}{0}{1}{-1}{\en}
\ocalt{t2}{}{Y}{4}{1}{-1}{\en}
\oc{t3}{N}{}{2}{1}{-1}{\et}
  \begin{scope}[on background layer]
\dlp{({t3}t)--({t2}t)}
\blp{({t3}m)--({t2}m)}
\alp{({t3}b)--({t2}b)}
\dblp{({ta}t)--({t3}t)}
\clp{({ta}m)--({t3}m)}
\end{scope}
\end{tikzpicture}}
\end{equation}
For example, a 1-cell from $A\otimes B$ to $C\otimes D\otimes E$ would have two edges entering on the left with the edge representing $A$ below that for $B$ and three edges on the right with that for $C$ at the bottom and $E$ at the top.    If there are no edges attached to a side of a 1-cell then the 0-cell on that side is the monoidal unit 0-cell. 

Vertical stacking of colored boxes represents $\otimes$ product and horizontal concatenation is the $\odot$ composition. 

Two-cells are represented by arrows between groups of edges and colored boxes.  In particular, we do not represent 2-cells with surfaces.

\begin{example}  As a limited first example, 
\cref{fig:dualizable_1_cell} represents the maps 
and compatibility for a dual pair of 1-cells in a bicategory.  
As all 1-cells are assigned to a single 0-cell on left and right the colored boxes each have 1-edge entering on the left and one edge on the right.  There is no vertical stacking since this is a definition in a bicategory rather than a monoidal bicategory.

\begin{figure}
    \begin{subfigure}[t]{0.2\textwidth}
        \centering
\resizebox{\textwidth}{!}{\begin{tikzpicture}

\node[draw] (A1) at (0, 2){\begin{tikzpicture}
\ocalt{y}{}{Y}{4.5}{0}{0}{\en}
\ocalt{x}{}{X}{0}{0}{0}{\en}

  \begin{scope}[on background layer]
\alp{ ({x}t)--({y}t)}
\end{scope}
\end{tikzpicture}
};
\node[draw] (A2) at (0, 0){\begin{tikzpicture}
\ocalt{y}{}{Y}{4.5}{0}{0}{\en}
\ocalt{x}{}{X}{0}{0}{0}{\en}
\oc{m}{M}{}{1.5}{0}{0}{\et}
\oc{n}{N}{}{3}{0}{0}{\ep}
  \begin{scope}[on background layer]
\blp{ ({m}b)--({n}t)}
\alp{ ({x}b)--({m}b)}
\alp{ ({n}t)--({y}t)}
\end{scope}

\end{tikzpicture}
};

\draw[ ->](A1)--(A2)node [midway , fill=white] {$\eta$};
\end{tikzpicture}}

        \caption{Coevaluation}
    \end{subfigure}%
\hspace{.5cm}
    \begin{subfigure}[t]{0.2\textwidth}
        \centering
\resizebox{\textwidth}{!}{
\begin{tikzpicture}

\node[draw] (A1) at (0, 2){\begin{tikzpicture}
\ocalt{y}{}{Y}{4.5}{0}{0}{\en}
\ocalt{x}{}{X}{0}{0}{0}{\en}
\oc{n}{N}{}{1.5}{0}{0}{\ep}
\oc{m}{M}{}{3}{0}{0}{\et}

  \begin{scope}[on background layer]
\alp{({m}b)--({n}t)}
\blp{ ({x}b)--({n}b)}
\blp{({m}t)--({y}t)}
\end{scope}
\end{tikzpicture}
};

\node[draw] (A2) at (0, 0){\begin{tikzpicture}
\ocalt{y}{}{Y}{4.5}{0}{0}{\en}
\ocalt{x}{}{X}{0}{0}{0}{\en}
  \begin{scope}[on background layer]
\blp{ ({x}t)--({y}t)}
\end{scope}
\end{tikzpicture}
};

\draw[ ->](A1)--(A2)node [midway , fill=white] {$\epsilon$};
\end{tikzpicture}}

        \caption{Evaluation}
    \end{subfigure}%
 \hspace{.5cm}
    \begin{subfigure}[t]{0.45\textwidth}
        \centering
\resizebox{\textwidth}{!}{
\begin{tikzpicture}

\node[draw] (A3) at (0, 4){\begin{tikzpicture}
\ocalt{y}{}{Y}{6}{0}{0}{\en}
\ocalt{x}{}{X}{0}{0}{0}{\en}
\oc{n}{N}{}{1.5}{0}{0}{\ep}

  \begin{scope}[on background layer]
\alp{({n}b)--({y}t)}
\blp{({x}b)--({n}b)}
\end{scope}
\end{tikzpicture}
};

\node[draw] (A1) at (0, 2){\begin{tikzpicture}
\ocalt{y}{}{Y}{6}{0}{0}{\en}
\ocalt{x}{}{X}{0}{0}{0}{\en}
\oc{n}{N}{}{1.5}{0}{0}{\ep}
\oc{m}{M}{}{3}{0}{0}{\et}
\oc{n2}{N}{}{4.5}{0}{0}{\ep}
  \begin{scope}[on background layer]
\alp{ ({m}b)--({n}t)}
\alp{ ({n2}b)--({y}t)}
\blp{ ({x}b)--({n}b)}
\blp{({m}t)--({n2}t)}
\end{scope}
\end{tikzpicture}
};

\node[draw] (A2) at (0, 0){\begin{tikzpicture}
\ocalt{y}{}{Y}{6}{0}{0}{\en}
\ocalt{x}{}{X}{0}{0}{0}{\en}
\oc{n2}{N}{}{4.5}{0}{0}{\ep}
  \begin{scope}[on background layer]
\alp{ ({n2}b)--({y}t)}
\blp{({x}t)--({n2}b)}
\end{scope}
\end{tikzpicture}
};
\draw[ ->](A3)--(A1)node [midway , fill=white] {$\id\odot \eta$};
\draw[ ->](A1)--(A2)node [midway , fill=white] {$\epsilon\odot \id$};
\end{tikzpicture}
\hspace{1cm}
\begin{tikzpicture}

\node[draw] (A3) at (0, 4){\begin{tikzpicture}
\ocalt{y}{}{Y}{6}{0}{0}{\en}
\ocalt{x}{}{X}{0}{0}{0}{\en}
\oc{m}{M}{}{4.5}{0}{0}{\et}
  \begin{scope}[on background layer]
\blp{ ({m}b)--({y}t)}
\alp{ ({x}b)--({m}b)}
\end{scope}
\end{tikzpicture}
};

\node[draw] (A1) at (0, 2){\begin{tikzpicture}
\ocalt{y}{}{Y}{6}{0}{0}{\en}
\ocalt{x}{}{X}{0}{0}{0}{\en}
\oc{m}{M}{}{1.5}{0}{0}{\et}
\oc{n}{N}{}{3}{0}{0}{\ep}
\oc{m2}{M}{}{4.5}{0}{0}{\et}

  \begin{scope}[on background layer]
\blp{({n}b)--({m}t)}
\blp{ ({m2}b)--({y}t)}
\alp{({x}b)--({m}b)}
\alp{({n}t)--({m2}t)}
\end{scope}

\end{tikzpicture}
};

\node[draw] (A2) at (0, 0){\begin{tikzpicture}
\ocalt{y}{}{Y}{6}{0}{0}{\en}
\ocalt{x}{}{X}{0}{0}{0}{\en}
\oc{m}{M}{}{1.5}{0}{0}{\et}
  \begin{scope}[on background layer]
\blp{ ({m}b)--({y}t)}
\alp{({x}b)--({m}b)}
\end{scope}
\end{tikzpicture}
};
\draw[ ->](A3)--(A1)node [midway , fill=white] {$\eta\odot \id$};
\draw[ ->](A1)--(A2)node [midway , fill=white] {$\id\odot \epsilon$};
\end{tikzpicture}}

        \caption{Triangle identities assert these two composites are the identity map}
    \end{subfigure}%
\caption{Circuit diagrams for dualizable 1-cells}\label{fig:dualizable_1_cell}
\end{figure}
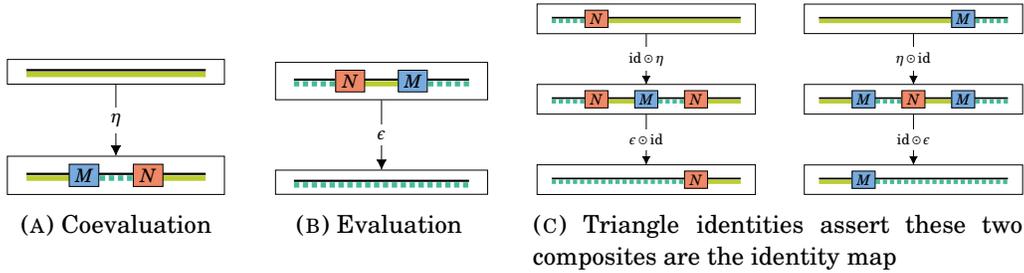

\end{example}

\begin{example}\label{ex:define_Gamma}
In a symmetric monoidal bicategory (see \cref{sec:formal_def}) for each pair of 0-cells $A, B$, there is a 1-cell $\Gamma$ from  $A\otimes B$ to  $B\otimes A$.  We represent this graphically as the following exchange of lines. 
\begin{equation}\label{fig:symmetry}
\begin{tikzpicture}

\ocalt{t1}{}{\coo{A}_1}{1}{5}{4}{\et}

\ocalt{ex}{}{\evo{A}_x}{4}{5}{4}{\ep}
\gc{g1}{2.5}{5}{4}

  \begin{scope}[on background layer]

\alp{ ({t1}b)\gpb{g1}({ex}t)}
\blp{({t1}t)\gpt{g1}({ex}b)}
\end{scope}
\end{tikzpicture}
\end{equation}
The 1-cell $\Gamma$ is part of an adjoint equivalence.  An adjoint equivalence is a dual pair where the coevaluation and evaluation 2-cells are isomorphisms. (In \cite{campbell_ponto}, motived by the applications of the paper, we called this a Morita equivalence.)  In particular, we have maps as in \cref{fig:dualizable_1_cell_symmetry_1} and \cref{fig:dualizable_1_cell_symmetry_2} and the composites in \cref{fig:dualizable_1_cell_symmetry_3} are the identity.

\begin{figure}
    \begin{subfigure}[t]{0.2\textwidth}
        \centering
\resizebox{\textwidth}{!}{\begin{tikzpicture}

\node[draw] (A1) at (0, 3){\begin{tikzpicture}
\ocalt{y}{}{Y}{4.5}{1}{0}{\en}
\ocalt{x}{}{X}{0}{1}{0}{\en}

  \begin{scope}[on background layer]
\alp{ ({x}t)--({y}t)}
\blp{ ({x}b)--({y}b)}
\end{scope}
\end{tikzpicture}
};
\node[draw] (A2) at (0, 0){\begin{tikzpicture}
\ocalt{y}{}{Y}{4.5}{1}{0}{\en}
\ocalt{x}{}{X}{0}{1}{0}{\en}
\gc{g1}{1.5}{1}{0}
\gc{g2}{3}{1}{0}
  \begin{scope}[on background layer]
\alp{ 
(3*\d-\w,0*\h)--(3*\d+\w,1*\h)--({y}t)}
\blp{ ({x}b)--(1.5*\d-\w,0*\h)--(1.5*\d+\w,1*\h)--(3*\d-\w,1*\h)--(3*\d+\w,0*\h)--({y}b)}
\alp{ ({x}t)--(1.5*\d-\w,1*\h)--(1.5*\d+\w,0*\h)--(3*\d-\w,0*\h)
}
\end{scope}

\end{tikzpicture}
};

\draw[ ->](A1)--(A2)node [midway , fill=white] {$\eta$};
\end{tikzpicture}}

        \caption{Coevaluation}\label{fig:dualizable_1_cell_symmetry_1}
    \end{subfigure}%
\hspace{.5cm}
    \begin{subfigure}[t]{0.2\textwidth}
        \centering
\resizebox{\textwidth}{!}{
\begin{tikzpicture}

\node[draw] (A1) at (0, 3){\begin{tikzpicture}
\ocalt{y}{}{Y}{4.5}{1}{0}{\en}
\ocalt{x}{}{X}{0}{1}{0}{\en}

  \begin{scope}[on background layer]
\alp{ 
(3*\d-\w,0*\h)--(3*\d+\w,1*\h)--({y}t)}
\blp{ ({x}b)--(1.5*\d-\w,0*\h)--(1.5*\d+\w,1*\h)--(3*\d-\w,1*\h)--(3*\d+\w,0*\h)--({y}b)}
\alp{ ({x}t)--(1.5*\d-\w,1*\h)--(1.5*\d+\w,0*\h)--(3*\d-\w,0*\h)
}
\end{scope}
\end{tikzpicture}
};

\node[draw] (A2) at (0, 0){\begin{tikzpicture}
\ocalt{y}{}{Y}{4.5}{1}{0}{\en}
\ocalt{x}{}{X}{0}{1}{0}{\en}
  \begin{scope}[on background layer]

\alp{ ({x}t)--({y}t)}
\blp{ ({x}b)--({y}b)}
\end{scope}
\end{tikzpicture}
};

\draw[ ->](A1)--(A2)node [midway , fill=white] {$\epsilon$};
\end{tikzpicture}}

        \caption{Evaluation}\label{fig:dualizable_1_cell_symmetry_2}
    \end{subfigure}%
 \hspace{.5cm}
    \begin{subfigure}[t]{0.45\textwidth}
        \centering
\resizebox{\textwidth}{!}{
\begin{tikzpicture}

\node[draw] (A3) at (0, 6){\begin{tikzpicture}
\ocalt{y}{}{Y}{6}{1}{0}{\en}
\ocalt{x}{}{X}{0}{1}{0}{\en}

  \begin{scope}[on background layer]
\alp{({x}b)--(1.5*\d-\w,0*\h)--(1.5*\d+\w,1*\h)--({y}t)}
\blp{({x}t)--(1.5*\d-\w,1*\h)--(1.5*\d+\w,0*\h)--({y}b)}
\end{scope}
\end{tikzpicture}
};

\node[draw] (A1) at (0, 3){\begin{tikzpicture}
\ocalt{y}{}{Y}{6}{1}{0}{\en}
\ocalt{x}{}{X}{0}{1}{0}{\en}
  \begin{scope}[on background layer]
\alp{({x}b)--(1.5*\d-\w,0*\h)--(1.5*\d+\w,1*\h)--(3*\d-\w,1*\h)
}
\blp{({x}t)--(1.5*\d-\w,1*\h)--(1.5*\d+\w,0*\h)--(3*\d-\w,0*\h)
}

\blp{
(3*\d-\w,0*\h)--(3*\d+\w,1*\h)--(4.5*\d-\w,1*\h)
}
\alp{
(3*\d-\w,1*\h)--(3*\d+\w,0*\h)--(4.5*\d-\w,0*\h)
}

\alp{
(4.5*\d-\w,0*\h)--(4.5*\d+\w,1*\h)--({y}t)}
\blp{
(4.5*\d-\w,1*\h)--(4.5*\d+\w,0*\h)--({y}b)}
\end{scope}
\end{tikzpicture}
};

\node[draw] (A2) at (0, 0){\begin{tikzpicture}
\ocalt{y}{}{Y}{6}{1}{0}{\en}
\ocalt{x}{}{X}{0}{1}{0}{\en}
  \begin{scope}[on background layer]
\alp{({x}b)--(4.5*\d-\w,0*\h)--(4.5*\d+\w,1*\h)--({y}t)}
\blp{({x}t)--(4.5*\d-\w,1*\h)--(4.5*\d+\w,0*\h)--({y}b)}
\end{scope}
\end{tikzpicture}
};
\draw[ ->](A3)--(A1)node [midway , fill=white] {$\id\odot \eta$};
\draw[ ->](A1)--(A2)node [midway , fill=white] {$\epsilon\odot \id$};
\end{tikzpicture}
\hspace{1cm}
\begin{tikzpicture}

\node[draw] (A3) at (0, 6){\begin{tikzpicture}
\ocalt{y}{}{Y}{6}{1}{0}{\en}
\ocalt{x}{}{X}{0}{1}{0}{\en}
  \begin{scope}[on background layer]
\alp{({x}b)--(4.5*\d-\w,0*\h)--(4.5*\d+\w,1*\h)--({y}t)}
\blp{({x}t)--(4.5*\d-\w,1*\h)--(4.5*\d+\w,0*\h)--({y}b)}
\end{scope}
\end{tikzpicture}
};

\node[draw] (A1) at (0, 3){\begin{tikzpicture}
\ocalt{y}{}{Y}{6}{1}{0}{\en}
\ocalt{x}{}{X}{0}{1}{0}{\en}
  \begin{scope}[on background layer]
\alp{({x}b)--(1.5*\d-\w,0*\h)--(1.5*\d+\w,1*\h)--(3*\d-\w,1*\h)
}
\blp{({x}t)--(1.5*\d-\w,1*\h)--(1.5*\d+\w,0*\h)--(3*\d-\w,0*\h)
}
\blp{
(3*\d-\w,0*\h)--(3*\d+\w,1*\h)--(4.5*\d-\w,1*\h)
}
\alp{
(3*\d-\w,1*\h)--(3*\d+\w,0*\h)--(4.5*\d-\w,0*\h)
}
\alp{
(4.5*\d-\w,0*\h)--(4.5*\d+\w,1*\h)--({y}t)}

\blp{
(4.5*\d-\w,1*\h)--(4.5*\d+\w,0*\h)--({y}b)}
\end{scope}

\end{tikzpicture}
};

\node[draw] (A2) at (0, 0){\begin{tikzpicture}
\ocalt{y}{}{Y}{6}{1}{0}{\en}
\ocalt{x}{}{X}{0}{1}{0}{\en}
  \begin{scope}[on background layer]
\alp{({x}b)--(1.5*\d-\w,0*\h)--(1.5*\d+\w,1*\h)--({y}t)}
\blp{({x}t)--(1.5*\d-\w,1*\h)--(1.5*\d+\w,0*\h)--({y}b)}
\end{scope}
\end{tikzpicture}
};
\draw[ ->](A3)--(A1)node [midway , fill=white] {$\eta\odot \id$};
\draw[ ->](A1)--(A2)node [midway , fill=white] {$\id\odot \epsilon$};
\end{tikzpicture}}

        \caption{Triangle identities assert these two composites are the identity map}\label{fig:dualizable_1_cell_symmetry_3}
    \end{subfigure}%
\caption{Circuit diagrams demonstrating dualizability for symmetry}\label{fig:dualizable_1_cell_symmetry}
\end{figure}
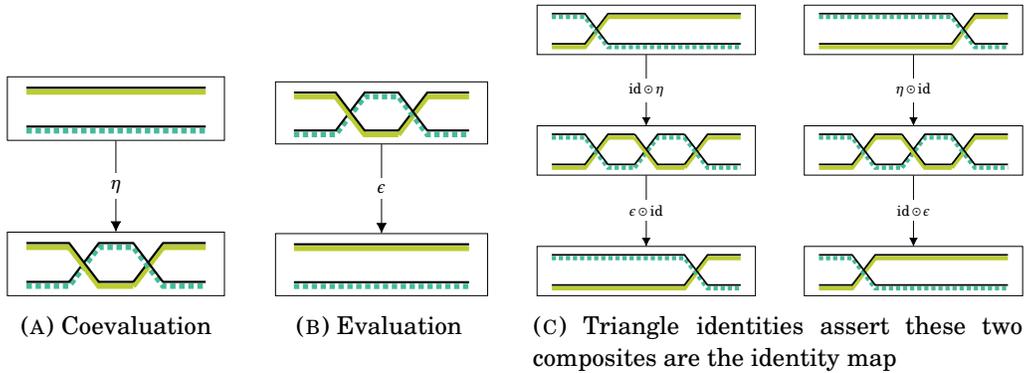
\end{example}

\begin{rmk}
For the graphical representation of the 1-cell $\Gamma$  in  \eqref{fig:symmetry} the data of which line crosses over the other is an artifact of the representation that does not reflect the underlying mathematics. So the order of crossing should be disregarded.  When read from left to right we will always have the top line cross over the bottom, but this is just to minimize distraction.
\end{rmk}

\begin{example}
Another illuminating example of the translation between circuit diagrams and more conventional expressions can be found in \cref{defn:1_dualizable}.  The map in \eqref{eq:trictop} is written graphically in \cref{fig:dualizable_0_cell} and the same for \eqref{eq:trietop} and \cref{fig:dualizable_0_cell_a}.  These have both vertical stacking and horizontal composition and 1-cells have sides with no edges and sides with multiple edges.   We find the graphical descriptions in  \cref{fig:dualizable_0_cell,fig:dualizable_0_cell_a} significantly more clear than the descriptions in \eqref{eq:trictop} and \eqref{eq:trietop} and so we will default to the graphical representations. 
\end{example}

\cref{fig:toroidal_framing} is an interpretation of \cref{fig:toroidal_no_framing} that is consistent with these graphical conventions.
\begin{figure}
\begin{subfigure}[b]{0.20\textwidth}
\resizebox{\textwidth}{!}{\begin{tikzpicture}

\newcommand\vshift{2}
\newcommand\edgeanglein{80}
\newcommand\edgeangleout{-100}

\coordinate (a1) at (-3,1*\vshift);
\coordinate (a4) at (3,1*\vshift);

\draw[linestyle,fuzzlefta]
 (a1) to  [looseness=.4,out=-90, in=-90] 
(a4) to  [looseness=.4,out=90, in=90] 
(a1)
;

\coordinate (b1) at (-3,0);
\coordinate (b2) at (-1.6,.5);
\coordinate (b3) at (-1.5,-.5);
\coordinate (b4) at (1.6,.5);
\coordinate (b5) at (1.5,-.5);
\coordinate (b6) at (3,0);

\draw[linestyle,fuzzleft]
(b3)to  [looseness=1,out=-10, in=-170] (b5);

\draw[linestyle,fuzzlefta]
	 (b5)to  [looseness=1,out=10, in=-100]
	(b6)to  [looseness=1,out=80, in=0] 
	(b4);

\draw[linestyle,fuzzleft]
	(b4)to  [looseness=1,out=170, in=10] 	(b2);

\draw[linestyle,fuzzlefta]
	(b2)to  [looseness=1,out=-160, in=80] 
	(b1)to  [looseness=1,out=-100, in=170] 
	(b3);

\coordinate (c1) at (-3,-1*\vshift);
\coordinate (c2) at (-2.25,-1*\vshift);
\coordinate (c3) at (-1,-1*\vshift);
\coordinate (c4) at (1,-1*\vshift);
\coordinate (c5) at (1.75,-1*\vshift);
\coordinate (c6) at (3,-1*\vshift);

\draw[linestyle,fuzzlefta]
(c2)  to [looseness=1.6,out=135, in=\edgeanglein] 
(c1) to  [looseness=1.6,out=\edgeangleout, in=-135] 
(c2)
;
\draw[linestyle,fuzzleft]
(c2) to [looseness=1.6,out=35, in=125]  
(c3) to [looseness=1.6,out=-55, in=-145]  
(c4)to  [looseness=1.6,out=35, in=125] (c5) 
;

\draw[linestyle,fuzzlefta]
(c5) to [looseness=1.6,out=-45, in=\edgeangleout]  
(c6)  to [looseness=1.6,out=\edgeanglein, in=45]  
(c5) 
;

\draw[linestyle,fuzzleft]
(c5) to [looseness=1.6,out=-145, in=-55] 
 (c4) 
to  [looseness=1.6,out=125, in=35] 
(c3)
to  [looseness=1.6,out=-145, in=-55] (c2)
;

\coordinate (ca1) at (-3,-2*\vshift);
\coordinate (ca2) at (-2.25,-2*\vshift);
\coordinate (ca3) at (-1,-2*\vshift);
\coordinate (ca4) at (1,-2*\vshift);
\coordinate (ca5) at (1.75,-2*\vshift);
\coordinate (ca6) at (3,-2*\vshift);

\draw[linestyle,fuzzlefta]
(ca2)  to [looseness=1.6,out=135, in=\edgeanglein] 
(ca1) to  [looseness=1.6,out=\edgeangleout, in=-135] 
(ca2) 
;
\draw[linestyle,fuzzleft]
(ca2) to [looseness=1.6,out=45, in=\edgeanglein]  
(ca3) to [looseness=1.6,out=\edgeangleout, in=-45]  
(ca2) 
;

\draw[linestyle,fuzzlefta]
(ca5) to [looseness=1.6,out=-45, in=\edgeangleout]  
(ca6)  to [looseness=1.6,out=\edgeanglein, in=45]  
(ca5) 
;

\draw[linestyle,fuzzleft]
(ca5) to [looseness=1.6,out=-135, in=\edgeangleout] 
 (ca4) 
to  [looseness=1.6,out=\edgeanglein, in=135] 
(ca5)
;

\coordinate (d1) at (-3,-3.5*\vshift);
\coordinate (d2) at (-2.25,-3.5*\vshift);
\coordinate (d3) at (-1,-3.5*\vshift);
\coordinate (d4) at (1,-3.5*\vshift);
\coordinate (d5) at (1.75,-3.5*\vshift);
\coordinate (d6) at (3,-3.5*\vshift);

\draw[linestyle,fuzzright]
(d2)  to [looseness=1.6,out=135, in=\edgeanglein] 
(d1) to  [looseness=1.6,out=\edgeangleout, in=-135] 
(d2) 
;
\draw[linestyle,fuzzrighta]
(d2) to [looseness=1.6,out=45, in=\edgeanglein]  
(d3) to [looseness=1.6,out=\edgeangleout, in=-45]  
(d2) 
;

\draw[linestyle,fuzzright]
(d5) to [looseness=1.6,out=-45, in=\edgeangleout]  
(d6)  to [looseness=1.6,out=\edgeanglein, in=45]  
(d5) 
;

\draw[linestyle,fuzzrighta]
(d5) to [looseness=1.6,out=-135, in=\edgeangleout] 
 (d4) 
to  [looseness=1.6,out=\edgeanglein, in=135] 
(d5)
;

\coordinate (e1) at (-3,-4.5*\vshift);
\coordinate (e2) at (-2.25,-4.5*\vshift);
\coordinate (e3) at (-1,-4.5*\vshift);
\coordinate (e4) at (1,-4.5*\vshift);
\coordinate (e5) at (1.75,-4.5*\vshift);
\coordinate (e6) at (3,-4.5*\vshift);

\draw[linestyle,fuzzleft]
(e2)  to [looseness=1.6,out=135, in=\edgeanglein] 
(e1) to  [looseness=1.6,out=\edgeangleout, in=-135] 
(e2) 
;
\draw[linestyle,fuzzlefta]
(e2) to [looseness=1.6,out=35, in=125]  
(e3) to [looseness=1.6,out=-55, in=-145]  
(e4)to  [looseness=1.6,out=35, in=125] (e5) 
;

\draw[linestyle,fuzzleft]
(e5) to [looseness=1.6,out=-45, in=\edgeangleout]  
(e6)  to [looseness=1.6,out=\edgeanglein, in=45]  
(e5) 
;

\draw[linestyle,fuzzlefta]
(e5) to [looseness=1.6,out=-145, in=-55] 
 (e4) 
to  [looseness=1.6,out=125, in=35] 
(e3)
to  [looseness=1.6,out=-145, in=-55] (e2)
;

\coordinate (f1) at (-3,-5.5*\vshift);
\coordinate (f2) at (-1.6,-5.5*\vshift+.5);
\coordinate (f3) at (-1.5,-5.5*\vshift+-.5);
\coordinate (f4) at (1.6,-5.5*\vshift+.5);
\coordinate (f5) at (1.5,-5.5*\vshift+-.5);
\coordinate (f6) at (3,-5.5*\vshift);

\draw[linestyle,fuzzlefta]
(f3)to  [looseness=1,out=-10, in=-170] (f5);

\draw[linestyle,fuzzleft]
	 (f5)to  [looseness=1,out=10, in=-100]
	(f6)to  [looseness=1,out=80, in=0] 
	(f4);

\draw[linestyle,fuzzlefta]
	(f4)to  [looseness=1,out=170, in=10] 	(f2);

\draw[linestyle,fuzzleft]
	(f2)to  [looseness=1,out=-160, in=80] 
	(f1)to  [looseness=1,out=-100, in=170] 
	(f3);

\coordinate (g3) at (3,-6.5*\vshift);
\coordinate (g4) at (-3,-6.5*\vshift);

\draw[linestyle,fuzzright]
 (g4) to  [looseness=.4,out=\edgeangleout, in=\edgeangleout] 
(g3) to  [looseness=.4,out=\edgeanglein, in=\edgeanglein] 
(g4)
;

\draw[dotted]
(a1)--(b1)--(c1)--(ca1)to [out =-90, in =90, looseness =1](d4);

\draw[dotted]
(a1)to [out =90, in =90, looseness =1.4 ] (a4)--(b6)--(c6)--(ca6)to [out =-90, in =90, looseness =1] (d3)to [out =-90, in =-90, looseness =1] (d4);

\draw[dotted]
(ca3)to [out =-90, in =90, looseness =1](d6)--(f6)--(g3) to [out =-90, in =-90, looseness =1] (g4)--(f1)--
(d1)to [out =90, in =-90, looseness =1](ca4) to [out =90, in =90, looseness =1] (ca3);

\end{tikzpicture}}
\caption{}
\end{subfigure}
\hspace{.1\textwidth}
\begin{subfigure}[b]{0.20\textwidth}
\resizebox{\textwidth}{!}{\begin{tikzpicture}

\newcommand\vshift{2}
\newcommand\edgeanglein{80}
\newcommand\edgeangleout{-100}

\coordinate (a1) at (-3,1*\vshift);
\coordinate (a4) at (3,1*\vshift);

\draw[linestyle,fuzzlefta]
 (a1) to  [looseness=.4,out=-90, in=-90] 
(a4) to  [looseness=.4,out=90, in=90] 
(a1)
;

\coordinate (b1) at (-3,0);
\coordinate (b2) at (-2.25,0);
\coordinate (b3) at (-1,0);
\coordinate (b4) at (1,0);
\coordinate (b5) at (1.75,0);
\coordinate (b6) at (3,0);

\draw[linestyle,fuzzlefta]
(b2) to [looseness=1.6,out=45, in=\edgeanglein]  
(b3) to [looseness=1.6,out=\edgeangleout, in=-45]  
(b2)  to [looseness=1.6,out=135, in=\edgeanglein] 
 (b1) to  [looseness=1.6,out=\edgeangleout, in=-135] 
(b2)
;

\draw[linestyle,fuzzrighta]
(b5) to [looseness=1.6,out=45, in=\edgeanglein]  
(b6) to [looseness=1.6,out=\edgeangleout, in=-45]  
(b5)  to [looseness=1.6,out=135, in=\edgeanglein] 
 (b4) 
to  [looseness=1.6,out=\edgeangleout, in=-135] 
(b5)
;

\coordinate (c1) at (-3,-1*\vshift);
\coordinate (c2) at (-2.25,-1*\vshift);
\coordinate (c3) at (-1,-1*\vshift);
\coordinate (c4) at (1,-1*\vshift);
\coordinate (c5) at (1.75,-1*\vshift);
\coordinate (c6) at (3,-1*\vshift);

\draw[linestyle,fuzzlefta]
(c2)  to [looseness=1.6,out=135, in=\edgeanglein] 
(c1) to  [looseness=1.6,out=\edgeangleout, in=-135] 
(c2) to [looseness=1.6,out=45, in=\edgeanglein]  
(c3)
;
\draw[linestyle,fuzzleft]
(c3) to [looseness=1.6,out=\edgeangleout, in=-45]  
(c2) 
;

\draw[linestyle,fuzzleft]
(c5) to [looseness=1.6,out=-45, in=\edgeangleout]  
(c6) 
;

\draw[linestyle,fuzzlefta]
(c6) to [looseness=1.6,out=\edgeanglein, in=45]  
(c5)  to [looseness=1.6,out=-135, in=\edgeangleout] 
 (c4) 
to  [looseness=1.6,out=\edgeanglein, in=135] 
(c5)
;

\coordinate (d1) at (-3,-2*\vshift);
\coordinate (d2) at (-2.25,-2*\vshift);
\coordinate (d3) at (-1,-2*\vshift);
\coordinate (d4) at (1,-2*\vshift);
\coordinate (d5) at (1.75,-2*\vshift);
\coordinate (d6) at (3,-2*\vshift);

\draw[linestyle,fuzzleft]
(d2)  to [looseness=1.6,out=135, in=\edgeanglein] 
(d1) to  [looseness=1.6,out=\edgeangleout, in=-135] 
(d2) to [looseness=1.6,out=45, in=\edgeanglein]  
(d3) 
;
\draw[linestyle,fuzzlefta]
(d3) to [looseness=1.6,out=\edgeangleout, in=-45]  
(d2) 
;

\draw[linestyle,fuzzlefta]
(d5) to [looseness=1.6,out=-45, in=\edgeangleout]  
(d6) 
;

\draw[linestyle,fuzzleft]
(d6) to [looseness=1.6,out=\edgeanglein, in=45]  
(d5)  to [looseness=1.6,out=-135, in=\edgeangleout] 
 (d4) 
to  [looseness=1.6,out=\edgeanglein, in=135] 
(d5)
;

\coordinate (e1) at (-3,-3*\vshift);
\coordinate (e2) at (-2.25,-3*\vshift);
\coordinate (e3) at (-1,-3*\vshift);
\coordinate (e4) at (1,-3*\vshift);
\coordinate (e5) at (1.75,-3*\vshift);
\coordinate (e6) at (3,-3*\vshift);

\draw[linestyle,fuzzleft]
(e2) to [looseness=1.6,out=45, in=\edgeanglein]  
(e3) to [looseness=1.6,out=\edgeangleout, in=-45]  
(e2)  to [looseness=1.6,out=135, in=\edgeanglein] 
 (e1) to  [looseness=1.6,out=\edgeangleout, in=-135] 
(e2)
;

\draw[linestyle,fuzzright]
(e5) to [looseness=1.6,out=45, in=\edgeanglein]  
(e6) to [looseness=1.6,out=\edgeangleout, in=-45]  
(e5)  to [looseness=1.6,out=135, in=\edgeanglein] 
 (e4) 
to  [looseness=1.6,out=\edgeangleout, in=-135] 
(e5)
;

\coordinate (f1) at (1,-4.5*\vshift);
\coordinate (f2) at (1.75,-4.5*\vshift);
\coordinate (f3) at (3,-4.5*\vshift);
\coordinate (f4) at (-3,-4.5*\vshift);
\coordinate (f5) at (-2.25,-4.5*\vshift);
\coordinate (f6) at (-1,-4.5*\vshift);

\draw[linestyle,fuzzleft]
(f2) to [looseness=1.6,out=135, in=\edgeanglein]  
(f1) to [looseness=1.6,out=\edgeangleout, in=-135]  
(f2)  to [looseness=1.6,out=45, in=\edgeanglein] 
 (f3) to  [looseness=1.6,out=\edgeangleout, in=-45] 
(f2)
;

\draw[linestyle,fuzzright]
(f5) to [looseness=1.6,out=135, in=80]  
(f4) to [looseness=1.6,out=-100, in=-135]  
(f5)  to [looseness=1.6,out=45, in=80] 
 (f6) 
to  [looseness=1.6,out=-100, in=-45] 
(f5)
;

\coordinate (g3) at (3,-5.5*\vshift);
\coordinate (g4) at (-3,-5.5*\vshift);

\draw[linestyle,fuzzright]
 (g4) to  [looseness=.4,out=\edgeangleout, in=\edgeangleout] 
(g3) to  [looseness=.4,out=\edgeanglein, in=\edgeanglein] 
(g4)
;

\draw[dotted]
(a1)--(b1)--(c1)--(d1)--(e1)to [out =-90, in =90, looseness =1](f1);

\draw[dotted]
(a1)to [out =90, in =90, looseness =1.4 ] (a4)--(b6)--(c6)--(d6)--(e6)to [out =-90, in =90, looseness =1] (f6) to [out =-90, in =-90, looseness =1](f1);

\draw[dotted]
(b3)--(c3)--(d3)--(e3)to [out =-90, in =90, looseness =1](f3)--(g3) to [out =-90, in =-90, looseness =1] (g4)--(f4)to [out =90, in =-90, looseness =1](e4)--(d4)--(c4)--(b4) to [out =90, in =90, looseness =1] (b3);

\end{tikzpicture}}
\caption{}
\end{subfigure}
\caption{Toroidal traces with framing}\label{fig:toroidal_framing}
\end{figure}
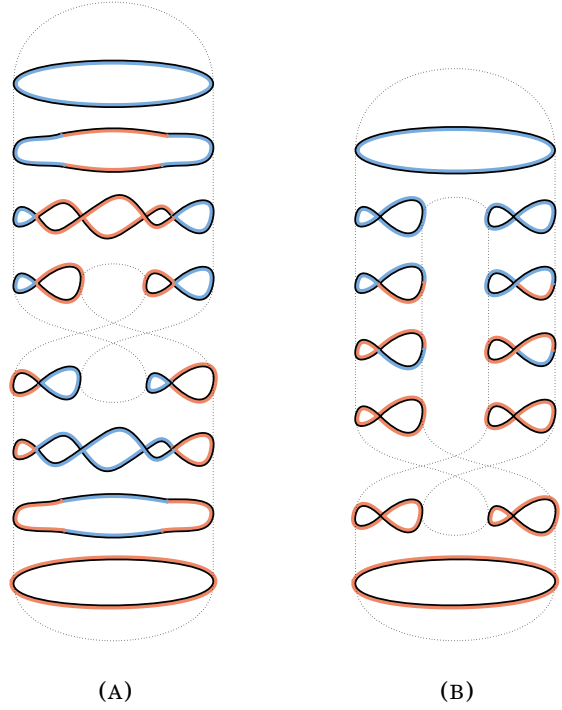

\subsection{Formal definitions}\label{sec:formal_def}
For the definition of a monoidal bicategory we follow \cite{stay} and briefly recall the relevant definitions here.  

\begin{defn}\cite{stay} A {\bf monoidal bicategory} consists of:
\begin{itemize}
\item  a bicategory $\sB$.
\item  a  functor  $\otimes \colon  \sB \times \sB \to \sB$ and with an invertible  2-morphism 
\[( M \otimes N) \odot  ( M ' \otimes  N') \Rightarrow ( M \odot M ') \otimes (N \odot N').\] 
\begin{itemize}
\item An adjoint equivalence 
\[a: (A\otimes B)\otimes C \to  A\otimes (B\otimes C)\] that is pseudonatural in $A, B, C$.
\item  An invertible modification $\pi$ relating the two different ways of moving parentheses from being clustered at the left of four objects to being clustered at the right. 
\item An equation of modifications relating the various ways of getting from the parentheses clustered at the left of five objects to clustered at the right.
\end{itemize}
\item  a 0-cell $I$
\begin{itemize}
\item  Adjoint equivalence 1-cells $L_A\colon I\otimes A \to A$ and  $R_A\colon A\otimes I \to  A$ that are pseudonatural in A.
\item  Invertible modifications $\lambda\colon l\otimes B\Rightarrow l\circ a$, $\mu\colon r\otimes B\Rightarrow (A\otimes l)\circ a$, and $\rho\colon r \Rightarrow (A\otimes r)\circ a$. 
\item Four equations of modifications relating the unit modifications. 
\end{itemize}
\end{itemize}
\end{defn}

\begin{defn}\cite{stay}\label{def:smb} 
A {\bf symmetric monoidal bicategory} consists of the following:
\begin{itemize}
	\item A monoidal bicategory $\mc{B}$.
	\item An adjoint equivalence 
			\[\Gamma: A\otimes B \to B\otimes A\]
		pseudonatural in $A$ and  $B$.
	\item Invertible modifications $R$ and $S$ filling the hexagons
			\[\xymatrix{A\otimes (B\otimes C)\ar[r]^\Gamma&(B\otimes C)\otimes A\ar[d]^a
				\\
				(A\otimes B)\otimes C\ar[u]^a\ar[d]_{\Gamma\otimes \id_C}&B\otimes (C\otimes A)
				\\
				(B\otimes A)\otimes C\ar[r]^a&B\otimes (A\otimes C)\ar[u]_{\id_B\otimes \Gamma}} 
	\hspace{1cm}
			\xymatrix{(A\otimes B)\otimes C\ar[r]^\Gamma&C\otimes (A\otimes B)\ar[d]_a
				\\
				A\otimes (B\otimes C)\ar[u]^a\ar[d]_{\id_A\otimes \Gamma}&(C\otimes A)\otimes B
				\\
				A\otimes (C\otimes B)\ar[r]^a&(A\otimes C)\otimes B\ar[u]^{\Gamma\otimes \id_B}}\]
	\begin{itemize}
		\item Two equations shuffling one object into three objects.
		\item An equation shuffling two objects into two objects 
		\item An equation relating multiple applications of $\Gamma$.
	\end{itemize}
	\item An invertible modification 
		\[\nu\colon \Gamma\to \Gamma\]
		satisfying two equations relating $\nu$ and the modifications $R$ and $S$ above and an equation relating $\nu$ applied to $A\otimes B$ and $B\otimes A$.
\end{itemize}
\end{defn}

\subsection{Pseudonaturality of $\Gamma$}\label{rmk:little_gamma}
We will make extensive use of the pseudonaturality of $\Gamma$ to rearrange 1-cells to allow for the applications of 2-cells.  (See, for example, \cref{fig:ec_e_and_c}.) The pseudonaturality isomorphism is graphically represented as the 
following 2-cell.
\begin{equation}\label{ex:pseduonaturality_gamma}
\resizebox{.5\textwidth}{!}{
\begin{tikzpicture}

\node[draw] (Z2) at (6, 4){\begin{tikzpicture}

\oc{x}{M}{}{2}{3}{3}{\en}
\oc{y}{N}{}{2}{2}{2}{\en}

  \begin{scope}[on background layer]
\alp{(0*\d-\w,2*\h)--(1*\d-\w,2*\h)--(1*\d+\w,3*\h)--({x}b)}
\blp{({x}b)--(3*\d+\w,3*\h)}
\clp{(0*\d-\w,3*\h)--(1*\d-\w,3*\h)--(1*\d+\w,2*\h)--({y}b)}
\dlp{({y}b)--(3*\d+\w,2*\h)}
\end{scope}
\end{tikzpicture}};

\node[draw] (Z5) at (0, 4){\begin{tikzpicture}

\oc{x}{M}{}{1}{2}{2}{\en}
\oc{y}{N}{}{1}{3}{3}{\en}

  \begin{scope}[on background layer]
\alp{(0*\d-\w,2*\h)--({x}b)}
\blp{({x}b)--(2*\d-\w,2*\h)--(2*\d+\w,3*\h)--(3*\d+\w,3*\h)}
\clp{(0*\d-\w,3*\h)--({y}b)}
\dlp{({y}b)--(2*\d-\w,3*\h)--(2*\d+\w,2*\h)--(3*\d+\w,2*\h)}
\end{scope}
\end{tikzpicture}};

\draw[<-](Z2)--(Z5);
\end{tikzpicture}}
\end{equation}
We will very rarely use such a simple form of this isomorphism, instead we will compose multiple instances or apply it to monoidal or bicategorical composites.  For an example of the first, the map $\gammao$ in \cref{gamma_list_1} is the composite 
\begin{center}
\resizebox{.8\textwidth}{!}{
\begin{tikzpicture}

\node[draw] (Z2) at (0, 4){\begin{tikzpicture}

\oc{t1}{\coo{A}}{_1}{1}{5}{4}{\et}
\oc{x}{P}{}{2}{4}{4}{\en}
\oc{y}{Q}{}{3}{2}{2}{\en}
\oc{t5}{\coo{B}}{_5}{2}{3}{2}{\et\dk}
\oc{e6}{\evo{B}}{_6}{3}{4}{3}{\ep\dk}

\oc{ex}{\evo{A}}{_x}{5}{5}{2}{\ep}
\gc{g1}{4}{5}{2}

  \begin{scope}[on background layer]
\dblp{({t5}t)--({e6}b)}

\alp{({t1}b)--({x}b)}
\blp{({x}b)--({e6}t)}
\blp{ ({t5}b)--({y}b)}
\alp{({y}b)\gpb{g1}({ex}t)}
\dlp{({t1}t)\gpt{g1}({ex}b)}
\end{scope}
\end{tikzpicture}};

\node[draw] (Z5) at (6, 4){\begin{tikzpicture}

\oc{t1}{\coo{A}}{_1}{-2}{5}{2}{\et}
\oc{x}{P}{}{-1}{2}{2}{\en}
\oc{y}{Q}{}{2}{3}{3}{\en}
\oc{t5}{\coo{B}}{_5}{1}{4}{3}{\et\dk}
\oc{e6}{\evo{B}}{_6}{2}{5}{4}{\ep\dk}

\oc{ex}{\evo{A}}{_x}{3}{3}{2}{\ep}
\gc{g1}{4}{3}{2}

  \begin{scope}[on background layer]

\dblp{({t5}t)--({e6}b)}

\alp{({t1}b)--({x}b)}
\blp{({x}b)--(0*\d-\w,2*\h)--(0*\d+\w,5*\h)--
({e6}t)}
\blp{ ({t5}b)--({y}b)}
\alp{({y}b)--({ex}t)}
\dlp{({t1}t)--
(0*\d-\w,5*\h)--(0*\d+\w,2*\h)--
({ex}b)}
\end{scope}
\end{tikzpicture}};

\node[draw] (Y3) at (12, 4){\begin{tikzpicture}

\oc{t1}{\coo{A}}{_1}{2}{3}{2}{\et}
\oc{x}{P}{}{3}{2}{2}{\en}
\oc{y}{Q}{}{2}{4}{4}{\en}
\oc{t5}{\coo{B}}{_5}{1}{5}{4}{\et\dk}
\oc{e6}{\evo{B}}{_6}{5}{5}{2}{\ep\dk}

\oc{ex}{\evo{A}}{_x}{3}{4}{3}{\ep}
\gc{g1}{4}{5}{2}
  \begin{scope}[on background layer]
\dlp{ ({t1}t)--({ex}b)}

\alp{ ({t1}b)--({x}b)}
\blp{({x}b)\gpb{g1}({e6}t)}
\blp{ ({t5}b)--({y}b)}
\alp{({y}b)--({ex}t)}
\dblp{({t5}t)\gpt{g1}({e6}b)}
\end{scope}
\end{tikzpicture}};

\draw[->](Z2)--(Z5);
\draw[->](Z5)--(Y3);
\end{tikzpicture}}
\end{center}
where the first map is the pseudonaturality isomorphism where the role of $M$ is played by the composite of $\evo{B}$, $\coo{B}$, and $Q$ and $N$ is an identity 1-cell.  The second map is  the inverse of the pseudonaturality isomorphism where the role of $N$ is played by the composite of $\coo{B}$, $Q$, and $\evo{A}$.  In this case $M$ is an identity 1-cell.
For an example of applying \eqref{ex:pseduonaturality_gamma} to a monoidal composite see $\gammaf$ and $\gammafp$ in \cref{gamma_list_2}.

To ensure that all such diagrams are related by unique 2-cells, we repeatedly invoke the following coherence theorem of Gurski--Osorno.

\begin{thm}\cite[Theorem 1.25]{gurski_osorno}
Let $B$ be a bicategory that is free on an underlying 1-globular set. Then in the free symmetric monoidal bicategory on $B$, every diagram of 2-cells commutes. 
\end{thm}

This is the ``many objects and 1-morphisms'' version of \cite[Thm. 1.23]{gurski_osorno}. In the case of interest for us, we consider the free symmetric monoidal bicategory generated by the 1-cells $M$ and $N$ in \eqref{ex:pseduonaturality_gamma}.

The composites of  \eqref{ex:pseduonaturality_gamma} we use in this paper are displayed in \cref{gamma_list_1,gamma_list_2,gamma_list_3,gamma_list_5} and they are organized by where they are used.  The maps in \cref{gamma_list_1} appear throughout this paper while those in \cref{gamma_list_2,gamma_list_3,gamma_list_5} appear only in very specific diagrams.  There is a very valid argument that is strange to treat the maps in  \cref{gamma_list_1,gamma_list_2,gamma_list_3,gamma_list_5} as distinct since they are all manifestations of  \eqref{ex:pseduonaturality_gamma}.  We choose to do so since identifying this list simplifies verifying the commutativity in later diagrams.  For example, in \cref{proof:euler_char_gamma,fig:ec_e_and_c,fig:ec_e_and_c_2} it makes it more clear which diagrams commute because of naturality. 

In addition to the maps in \cref{gamma_list_1,gamma_list_2,gamma_list_3,gamma_list_5} there are shifts that only involve identity 0-cells. Examples of these include the maps in \cref{gamma_list_4} from \cref{fig:ec_e_and_c}.  These occur less commonly and so we will use $\gammas$ to refer to all instances of \eqref{ex:pseduonaturality_gamma} involving only the monoidal until 0-cell except those for two separate components like $\gammaDo$. 

Together the maps in \cref{gamma_list_1,gamma_list_2,gamma_list_3,gamma_list_4,gamma_list_5} are all the ways we reorder 1-cells so  we can regard each of these maps as a fixed composite of the map in \eqref{ex:pseduonaturality_gamma} that we leave unchanged through the paper.  Alternatively, we can cite \cite[Theorem 1.25]{gurski_osorno} again to see that a particular choice of composite is not necessary. 

In the diagrams in \cref{fig:comp_1_duals_trace,proof:euler_char_gamma,fig:ec_e_and_c,fig:ec_e_and_c_2} there are small regions entirely consisting of the maps in  \cref{gamma_list_1,gamma_list_2,gamma_list_3,gamma_list_4,gamma_list_5}.  Again we can invoke  \cite[Theorem 1.25]{gurski_osorno} to assert that all these regions commute. 

\input{gamma_list}
\input{gamma_list_2}
\begin{figure}
\begin{tabular}{ c}

\resizebox{.48\textwidth}{!}
{
\begin{tikzpicture}[
    mystyle/.style={%
    },
   my style/.style={%
   },
  ]
 
\def\scale{1.15}

\node[draw,my style] (D2) at (\scale*-4, 0){
\begin{tikzpicture}
\oc{t8}{\evo{B}}{_8}{-7}{1}{0}{\ep}
\oc{m}{M}{}{-9}{0}{0}{\en}

\oc{e3}{\coo{B}}{_3}{-12}{1}{-2}{\et}

\oc{t1}{\evo{\zdual{B}}}{_1}{-9}{-1}{-2}{\ep\dk}
\oc{n}{N}{}{-10}{-1}{-1}{\en}

\oc{e2}{\coo{\zdual{A}}}{_2}{-11}{0}{-1}{\et\dk}

\begin{scope}[on background layer]
\dblp{({n}t)--
({t1}t)}
\blp{({m}t)--(-8*\d-\w,0*\h)--(-8*\d+\w,1*\h)--
({t8}t)}

\dlp{({e2}b)--({n}t)}

\alp{({e2}t)--
({m}b)}

\blp{({e3}b)--({t1}b)}
\dblp{({e3}t)--(-8*\d-\w,1*\h)--(-8*\d+\w,0*\h)--({t8}b)}
\end{scope}
\end{tikzpicture}
};

\node[draw,my style] (D3) at (\scale*-10, 0){
\begin{tikzpicture}
\oc{t8}{\evo{B}}{_8}{-8}{1}{0}{\ep}
\oc{m}{M}{}{-9}{1}{1}{\en}

\oc{e3}{\coo{B}}{_3}{-9}{0}{-1}{\et}

\oc{t1}{\evo{\zdual{B}}}{_1}{-7}{-1}{-2}{\ep\dk}
\oc{n}{N}{}{-10}{-2}{-2}{\en}

\oc{e2}{\coo{\zdual{A}}}{_2}{-11}{1}{-2}{\et\dk}

\begin{scope}[on background layer]
\dblp{({n}t)--(-8*\d-\w,-2*\h)--(-8*\d+\w,-1*\h)--
({t1}t)}
\blp{({m}t)--
({t8}t)}

\dlp{({e2}b)--({n}t)}

\alp{({e2}t)--
({m}b)}

\blp{({e3}b)--(-8*\d-\w,-1*\h)--(-8*\d+\w,-2*\h)--({t1}b)}
\dblp{({e3}t)--
({t8}b)}
\end{scope}
\end{tikzpicture}
};

\draw[->](D3)--(D2)node [midway , fill=white] {$\gammael$};
\end{tikzpicture}
}

\end{tabular}
\caption{}\label{gamma_list_5}
\end{figure}

\begin{figure}
\begin{tabular}{ c c}
\resizebox{.48\textwidth}{!}
{
\begin{tikzpicture}[
    mystyle/.style={%
    },
   my style/.style={%
   },
  ]

\def\x{1.1*3.5}
\def\y{1.1*5.25}

\node[draw,my style] (X2) at (-6*\x,6*\y){\begin{tikzpicture}
\oc{t1}{\coo{A}}{_1}{-2}{5}{4}{\et}
\gc{g1}{3}{1}{0}
\gc{g2}{4}{1}{0}
\oc{t2}{\rdual{\coo{A}}}{_2}{1}{1}{0}{\etr}
\oc{e3}{\rdual{\evo{A}}}{_3}{-1}{2}{1}{\epr\lt}
\oc{t4}{\coo{A}}{_4}{-2}{3}{0}{\et\lt}

\oc{t5}{\rdual{\coo{A}}}{_5}{1}{5}{2}{\etr\lt}
\oc{e6}{\evo{A}}{_6}{0}{4}{3}{\ep\lt}
\begin{scope}[on background layer]
\dlp{({e3}b)--({t2}t)}
\alp{({t4}b)--({t2}b)}

\alp{({t1}b)--
({e6}t)}
\dlp{({t1}t)--
({t5}t)}
\dlp{({t4}t)--
({e6}b)}
\alp{({e3}t)--
({t5}b)}
\end{scope}
\end{tikzpicture}
};

\node[draw,my style] (D2) at (-4.5*\x, 6*\y){\begin{tikzpicture}
\oc{t1}{\coo{A}}{_1}{-3}{5}{4}{\et}
\gc{g1}{3}{1}{0}
\gc{g2}{4}{1}{0}
\oc{t4}{\coo{A}}{_4}{-3}{3}{2}{\et\lt}
\oc{e6}{\evo{A}}{_6}{-2}{4}{3}{\ep\lt}
\oc{t2}{\rdual{\coo{A}}}{_2}{0}{3}{2}{\etr}
\oc{e7}{\rdual{\evo{A}}}{_7}{-1}{4}{3}{\epr\lt}
\oc{t8}{\rdual{\coo{A}}}{_8}{0}{5}{4}{\etr\lt}
\begin{scope}[on background layer]

\alp{({t4}b)--({t2}b)}
\dlp{({e7}b)--({t2}t)}
\alp{({t1}b)--({e6}t)}
\alp{({e7}t)--({t8}b)}
\dlp{({t4}t)--({e6}b)}
\dlp{({t1}t)--({t8}t)}
\end{scope}
\end{tikzpicture}
};

\draw[<-](D2)to node [midway , fill=white] {$\gammas$}(X2);
\end{tikzpicture}
}
&

\resizebox{.48\textwidth}{!}
{
\begin{tikzpicture}[
    mystyle/.style={%
    },
   my style/.style={%
   },
  ]

\def\x{1.1*3.5}
\def\y{1.1*5.25}

\node[draw,my style] (B4a) at (-.5*\x,0*\y){\begin{tikzpicture}
\gc{g1}{3}{1}{0}
\gc{g2}{4}{1}{0}
\oc{t4}{\coo{A}}{_4}{-2}{3}{2}{\et\lt}
\oc{e6}{\evo{A}}{_6}{0}{3}{2}{\ep\lt}
\oc{e7}{\rdual{\evo{A}}}{_7}{-2}{5}{4}{\epr\lt}
\oc{t8}{\rdual{\coo{A}}}{_8}{0}{5}{4}{\etr\lt}
\begin{scope}[on background layer]

\alp{({t4}b)--(-1*\d-\w,2*\h)--(-1*\d+\w,3*\h)--({e6}t)}

\dlp{({e7}b)--(-1*\d-\w,4*\h)--(-1*\d+\w,5*\h)--({t8}t)}
\alp{({e7}t)--(-1*\d-\w,5*\h)--(-1*\d+\w,4*\h)--({t8}b)}
\dlp{({t4}t)--(-1*\d-\w,3*\h)--(-1*\d+\w,2*\h)--({e6}b)}
\end{scope}
\end{tikzpicture}
};

\node[draw,my style] (B4) at (\x,0*\y){\begin{tikzpicture}
\gc{g1}{3}{1}{0}
\gc{g2}{4}{1}{0}
\oc{t4}{\coo{A}}{_4}{-2}{3}{2}{\et\lt}
\oc{e6}{\evo{A}}{_6}{0}{3}{2}{\ep\lt}
\oc{e7}{\rdual{\evo{A}}}{_7}{1}{3}{2}{\epr\lt}
\oc{t8}{\rdual{\coo{A}}}{_8}{3}{3}{2}{\etr\lt}
\begin{scope}[on background layer]

\alp{({t4}b)--(-1*\d-\w,2*\h)--(-1*\d+\w,3*\h)--({e6}t)}

\dlp{({e7}b)--(2*\d-\w,2*\h)--(2*\d+\w,3*\h)--({t8}t)}
\alp{({e7}t)--(2*\d-\w,3*\h)--(2*\d+\w,2*\h)--({t8}b)}
\dlp{({t4}t)--(-1*\d-\w,3*\h)--(-1*\d+\w,2*\h)--({e6}b)}
\end{scope}
\end{tikzpicture}
};

\draw[->](B4a)--(B4)node [midway , fill=white] {$\gammaDo$};

\end{tikzpicture}
}
\end{tabular}
\caption{}\label{gamma_list_4}
\end{figure}

\section{1-Dualizability}\label{sec:1_dualizable}
With the addition of a monoidal structure on a bicategory we gain more notions of duality.
The first of these is duality for 0-cells. This duality and the
duality in \cref{defn:dualizable} are distinct generalizations of
dualizability for objects in a monoidal category \cite{dold_puppe}. In
addition to its inherent interest, a crucial consequence is
\cref{prop:shadow} below. 

In a monoidal 1-category, a 0-cell (a.k.a. object) is {\bf dualizable} $A$ if
there is a 0-cell $\zdual{A}$ and maps $I \to A \otimes \zdual{A}$ and $\zdual{A} \otimes A \to I$
such that the two maps
\[
\begin{tikzcd}
A \cong I \otimes A \ar{r} & A \otimes \zdual{A} \otimes A \ar{r} & A
\otimes I \cong A 
\end{tikzcd}
\]
and
\[
\begin{tikzcd}
\zdual{A} \cong \zdual{A} \otimes I \ar{r} & \zdual{A} \otimes A \otimes \zdual{A}
\ar{r}&  I \otimes \zdual{A} \cong \zdual{A}
\end{tikzcd}
\]
are the identity.  This generalizes the discussion for vector spaces in \cref{sec:duality_bicat}.

 In a bicategory we cannot say that two 1-morphisms
``are'' identity maps, instead we must supply a 2-cell witnessing that
statement. This accounts for the extra complexity of the definition
below.

\begin{defn}\label{defn:1_dualizable}
A 0-cell $A$ in a monoidal bicategory $\sB$ is {\bf 1-dualizable} if there is
\begin{itemize}
\item a zero-cell $\zdual{A}$, 
\item  1-cells $\coo{A}\in \sB(I, A\otimes \zdual{A})$ and $\evo{A}\in \sB(\zdual{A}\otimes A, I)$ and 
\item invertible 2-cells (\cref{fig:dualizable_0_cell,fig:dualizable_0_cell_a})
\begin{equation}\label{eq:trietop}U_A\xto{\trictop} {L_A}^{-1} \odot
	(\coo{A}\otimes U_A)\odot
	(U_A\otimes \evo{A}) \odot
	{R_A} \end{equation}
\begin{equation}\label{eq:trictop}U_{\zdual{A}}\xto{\trietop} {R_{\zdual{A}}^{-1}}\odot
	(U_{\zdual{A}}\otimes \coo{A})\odot
	(\evo{A}\otimes U_{\zdual{A}})\odot
	L_{\zdual{A}}\end{equation}
so that the diagrams in \cref{fig:add_dualizable_0_cell,fig:add_dualizable_0_cell_a} commute. 
\end{itemize}
\end{defn}

\begin{figure}[ht]
\hspace{1cm}
\begin{subfigure}{.35\textwidth}
\resizebox{\textwidth}{!}
{
\begin{tikzpicture}

\node[draw] (A2) at (0, 0){\begin{tikzpicture}
\ocalt{ta}{}{X}{2}{-1}{-2}{\en}
\ocalt{t2}{}{Y}{5}{-1}{-2}{\en}

  \begin{scope}[on background layer]
\alp{({ta}t)--({t2}t)}
\end{scope}
\end{tikzpicture}
};

\node[draw] (A3) at (4, 0){\begin{tikzpicture}
\ocalt{ta}{}{X}{1}{1}{-1}{\en}
\ocalt{t2}{}{Y}{4}{1}{-1}{\en}
\oc{t3}{\coo{A}}{_3}{2}{0}{-1}{\et}
\oc{e4}{\evo{A}}{_4}{3}{1}{0}{\ep}
  \begin{scope}[on background layer]
\dlp{({t3}t)--({e4}b)}
\alp{({t3}b)--({t2}b)}
\alp{({ta}t)--({e4}t)}
\end{scope}
\end{tikzpicture}
};

\draw[->](A2)--(A3)node [midway , fill=white] {$\trietop$};
\end{tikzpicture}
}
\caption{$\trietop$}\label{fig:dualizable_0_cell}
\end{subfigure}
\hfill
\begin{subfigure}{.35\textwidth}
\resizebox{\textwidth}{!}
{
\begin{tikzpicture}

\node[draw] (A2) at (0, 0){\begin{tikzpicture}
\ocalt{ta}{}{X}{2}{-1}{-2}{\en}
\ocalt{t2}{}{Y}{5}{-1}{-2}{\en}
  \begin{scope}[on background layer]
\dlp{({ta}t)--({t2}t)}
\end{scope}
\end{tikzpicture}
};

\node[draw] (A3) at (4, 0){\begin{tikzpicture}
\ocalt{ta}{}{X}{1}{1}{-1}{\en}
\ocalt{t2}{}{Y}{4}{1}{-1}{\en}
\oc{t3}{\coo{A}}{_3}{2}{1}{0}{\et}
\oc{e4}{\evo{A}}{_4}{3}{0}{-1}{\ep}
  \begin{scope}[on background layer]
\alp{ ({t3}b)--({e4}t)}
\dlp{({t3}t)--({t2}t)}
\dlp{({ta}b)--({e4}b)}
\end{scope}
\end{tikzpicture}
};

\draw[->](A2)--(A3)node [midway , fill=white] {$\trictop$};
\end{tikzpicture}
}
\caption{$\trictop$}\label{fig:dualizable_0_cell_a}
\end{subfigure}
\hspace{1cm}
{\hspace{1cm}
\begin{subfigure}{.4\textwidth}
\resizebox{\textwidth}{!}
{
\begin{tikzpicture}[
    mystyle/.style={%
    },
   my style/.style={%
   },
  ]

\node[draw,my style] (A9) at (0, 0){\begin{tikzpicture}
\oc{t1}{\coo{A}}{_1}{-3}{5}{4}{\et}
\gc{g1}{3}{1}{0}
\gc{g2}{4}{1}{0}
\ocalt{x2}{}{X_2}{-1}{5}{4}{\en}

\begin{scope}[on background layer]

\dlp{({t1}t)--({x2}t)}
\alp{({t1}b)--({x2}b)}
\end{scope}
\end{tikzpicture}
};

\node[draw,my style] (A1) at (3, -3){\begin{tikzpicture}
\oc{t1}{\coo{A}}{_1}{-3}{5}{4}{\et}
\gc{g1}{3}{1}{0}
\gc{g2}{4}{1}{0}
\ocalt{x2}{}{X_2}{-1}{5}{2}{\en}
\oc{t4}{\coo{A}}{_4}{-3}{3}{2}{\et\lt}

\oc{e6}{\evo{A}}{_6}{-2}{4}{3}{\ep\lt}

\begin{scope}[on background layer]

\alp{({t4}b)--({x2}b)}
\dlp{({t1}t)--({x2}t)}
\alp{({t1}b)--({e6}t)}
\dlp{({t4}t)--({e6}b)}
\end{scope}
\end{tikzpicture}
};

\node[draw,my style] (B9) at (6, 0){\begin{tikzpicture}
\gc{g1}{3}{1}{0}
\gc{g2}{4}{1}{0}
\ocalt{x2}{}{X_2}{-1}{5}{4}{\en}
\oc{t4}{\coo{A}}{_4}{-3}{5}{4}{\et\lt}

\begin{scope}[on background layer]

\alp{({t4}b)--({x2}b)}
\dlp{({t4}t)--({x2}t)}
\end{scope}
\end{tikzpicture}
};

\draw[->](B9)--(A1)node [midway , fill=white] {$\tricto{16}$};
\draw[->](A9)--(A1)node [midway , fill=white] {$\trieto{46}$};

\draw[->](A9)--(B9)node [midway , fill=white] {$\id$};
\end{tikzpicture}
}
\caption{}\label{fig:add_dualizable_0_cell}
\end{subfigure}
\hfill
\begin{subfigure}{.4\textwidth}
\resizebox{\textwidth}{!}
{
\begin{tikzpicture}[
    mystyle/.style={%
    },
   my style/.style={%
   },
  ]

\node[draw,my style] (A1) at (3, -3){\begin{tikzpicture}
\oc{t1}{\coo{A}}{_1}{-3}{5}{4}{\et}
\ocalt{x4}{}{X_4}{-4}{6}{3}{\en}

\oc{e6}{\evo{A}}{_6}{-2}{4}{3}{\ep\lt}

\oc{ec}{\evo{A}}{_c}{-2}{6}{5}{\ep\dk}
\begin{scope}[on background layer]

\dlp{({t1}t)--({ec}b)}
\alp{({t1}b)--({e6}t)}
\alp{({x4}t)--({ec}t)}
\dlp{({x4}b)--({e6}b)}
\end{scope}
\end{tikzpicture}
};

\node[draw,my style] (B2) at (6, 0){\begin{tikzpicture}
\ocalt{x4}{}{X_4}{-4}{4}{3}{\en}

\oc{e6}{\evo{A}}{_6}{-2}{4}{3}{\ep\lt}

\begin{scope}[on background layer]

\alp{({x4}t)--({e6}t)}

\dlp{({x4}b)--({e6}b)}
\end{scope}
\end{tikzpicture}
};

\node[draw,my style] (B9) at (0, 0){\begin{tikzpicture}
\ocalt{x4}{}{X_4}{-4}{6}{5}{\en}

\oc{ec}{\evo{A}}{_c}{-2}{6}{5}{\ep\dk}
\begin{scope}[on background layer]

\dlp{({x4}b)--({ec}b)}
\alp{({x4}t)--({ec}t)}

\end{scope}
\end{tikzpicture}
};

\draw[->](B9)--(A1)node [midway , fill=white] {$\tricto{16}$};
\draw[<-](A1)--(B2)node [midway , fill=white] {$\trieto{1c}$};

\draw[->](B9)--(B2)node [midway , fill=white] {$\id$};
\end{tikzpicture}
}\caption{}\label{fig:add_dualizable_0_cell_a}
\end{subfigure}
\hspace{1cm}}
\caption{2-cells and conditions for a 1-dualizable 0-cell $A$}
\end{figure}
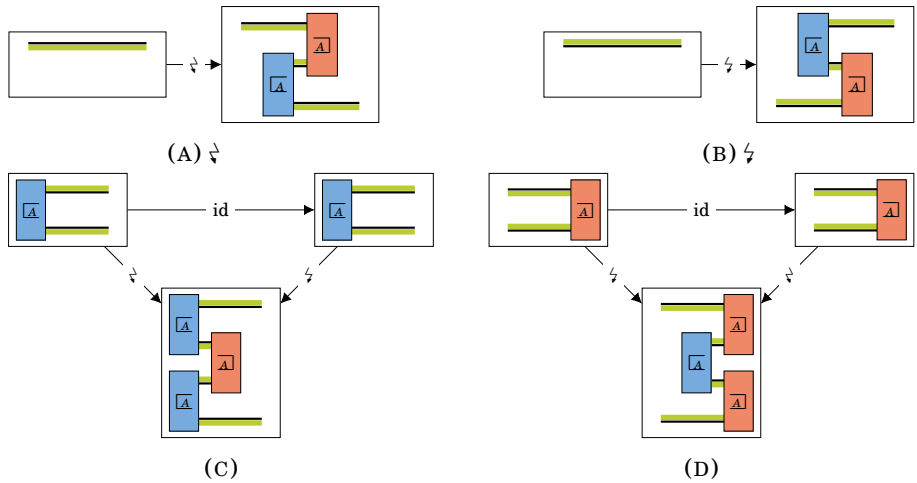

\begin{rmk}\label{rmk:cp_1_dual_change}
\cref{defn:1_dualizable} is  \cite[Def. 6.4]{cp2} with one substantive change and two notational changes.  

In \cite[Def. 6.4]{cp2} the maps $\trictop$ and $\trietop$ were both denoted $\triangle$.  The 1-cell $\coo{A}$ was denoted $C$ and $\evo{A}$ was denoted $E$.  Unlike \cite{cp2}, the clarity provided by giving these distinct maps distinct names is necessary here.  

The substantive difference is the condition that  \cref{fig:add_dualizable_0_cell,fig:add_dualizable_0_cell_a} commute.  This was not required for any of the results in \cite{cp2}, but will be required in \cref{lem:compatiblity_1-witnesses_if_2-dual,lem:euler_char_gamma} (and \cref{lem:ec_e_and_c,thm:main_pairing,thm:pairing_var_1}).
\end{rmk}

\begin{example}
\hfill
\begin{enumerate}
\item 
A ring $A$ is 1-dualizable and $\zdual{A}$ is $A^{\mathrm{op}}$.  The 1-cells $\coo{A}$ and $\evo{A}$ are $A$ regarded as a $(\bZ,A\otimes A^{\mathrm{op}})$-bimodule and 
$(A^{\mathrm{op}}\otimes A,\bZ)$-bimodule.  The maps $\trictop$ and $\trietop$ are nearly tautological in this case. 

\item 
  If $A$ is a category then $A$ is 1-dualizable with dual $A^{\mathrm{op}}$.  The 1-cells $\coo{A}$ and $\evo{A}$ are the functors $\ast\times (A\times A^{\mathrm{op}})\to \mathrm{Set}$ and $(A^{\mathrm{op}}\times A)\times \ast\to \mathrm{Set}$ given by the hom sets in the category $A$. 

  In this case the maps $\trictop$ and $\trietop$ are slightly more complicated. Writing the objects in terms of standard tensors over categories,    $\trictop$ is the isomorphism
  \[
  A(a, b)  \cong  A(a,-)\odot_A A(-,-) \odot_A A(-,-) \odot_A A(-,b).
  \]
and $\trietop$ is the same map for $A^{\mathrm{op}}$.

  \item If $\mc{C}$ is a dg-category or spectral category, entirely analogous statements hold. 
\end{enumerate}
\end{example}

A symmetric monoidal bicategory where all  0-cells are 1-dualizable endows the bicategory with a shadow.  This shadow is the familiar one in many examples, including those in \cref{ex:all_the_ex}.

\begin{prop}[\cref{goal:smc}, {\cite{ben_zvi_nadler_1}}, {\cite[Prop. 6.11]{cp2}}]\label{prop:shadow} If $\mc{B}$ is a symmetric monoidal bicategory and all 0-cells of $\sB$ are 1-dualizable, the composition in \cref{fig:shadow} defines a shadow on $\sB$  that takes values in the (symmetric monoidal)
category $\sB(I,I)$.
\end{prop}

If $M\in \sB(A,A)$ and $A$ is 1-dualizable, the shadow of $M$ is the  bicategorical composition in \cref{fig:shadow}.
The shadow isomorphism is defined in \cref{fig:shadow_isomorphism}.

\begin{figure}[ht]
\begin{subfigure}{.2\textwidth}
\resizebox{\textwidth}{!}{
\begin{tikzpicture}

\oc{t1}{\coo{A}}{_1}{0}{5}{4}{\et}
\oc{x}{M}{}{1.5}{4}{4}{\en}

\oc{ex}{\evo{A}}{_x}{4}{5}{4}{\ep}
\gc{g1}{3}{5}{4}

  \begin{scope}[on background layer]

\alp{ ({t1}b)--({x}b)\gpb{g1}({ex}t)}
\dlp{({t1}t)\gpt{g1}({ex}b)}
\end{scope}
\end{tikzpicture}}
\caption{Shadow}\label{fig:shadow}
\end{subfigure}

\begin{subfigure}{.9\textwidth}
\resizebox{\textwidth}{!}
{\begin{tikzpicture}

\node[draw] (Z2) at (0, 4){\begin{tikzpicture}

\oc{t1}{\coo{A}}{_1}{0}{5}{4}{\et}
\oc{x}{M}{}{1}{4}{4}{\en}
\oc{y}{N}{}{2}{4}{4}{\en}

\oc{ex}{\evo{A}}{_x}{4}{5}{4}{\ep}
\gc{g1}{3}{5}{4}

  \begin{scope}[on background layer]
\alp{ ({t1}b)--({x}b)}
\blp{({x}b)--({y}b)}
\alp{({y}b)\gpb{g1}({ex}t)}
\dlp{({t1}t)\gpt{g1}({ex}b)}

\end{scope}
\end{tikzpicture}};

\node[draw] (Z3) at (6, 4){\begin{tikzpicture}

\oc{t1}{\coo{A}}{_1}{1}{5}{4}{\et}
\oc{x}{M}{}{2}{4}{4}{\en}
\oc{y}{N}{}{3}{2}{2}{\en}
\oc{t5}{\coo{B}}{_5}{2}{3}{2}{\et\dk}
\oc{e6}{\evo{B}}{_6}{3}{4}{3}{\ep\dk}

\oc{ex}{\evo{A}}{_x}{5}{5}{2}{\ep}
\gc{g1}{4}{5}{2}

  \begin{scope}[on background layer]
\dblp{({t5}t)--({e6}b)}

\alp{({t1}b)--({x}b)}
\blp{({x}b)--({e6}t)}
\blp{ ({t5}b)--({y}b)}
\alp{({y}b)\gpb{g1}({ex}t)}
\dlp{({t1}t)\gpt{g1}({ex}b)}
\end{scope}
\end{tikzpicture}};

\node[draw] (X5) at (18, 4){\begin{tikzpicture}

\oc{x}{M}{}{3}{4}{4}{\en}
\oc{y}{N}{}{2}{4}{4}{\en}
\oc{t5}{\coo{B}}{_5}{1}{5}{4}{\et\dk}
\oc{e6}{\evo{B}}{_6}{5}{5}{4}{\ep\dk}
\gc{g1}{4}{5}{4}
  \begin{scope}[on background layer]
\blp{({t5}b)--({y}b)}
\alp{({y}b)--({x}b)}
\blp{({x}b)\gpb{g1}({e6}t)}
\dblp{({t5}t)\gpt{g1}({e6}b)}
\end{scope}
\end{tikzpicture}};

\node[draw] (Y3) at (12, 4){\begin{tikzpicture}

\oc{t1}{\coo{A}}{_1}{2}{3}{2}{\et}
\oc{x}{M}{}{3}{2}{2}{\en}
\oc{y}{N}{}{2}{4}{4}{\en}
\oc{t5}{\coo{B}}{_5}{1}{5}{4}{\et\dk}
\oc{e6}{\evo{B}}{_6}{5}{5}{2}{\ep\dk}

\oc{ex}{\evo{A}}{_x}{3}{4}{3}{\ep}
\gc{g1}{4}{5}{2}
  \begin{scope}[on background layer]
\dlp{ ({t1}t)--({ex}b)}

\alp{ ({t1}b)--({x}b)}
\blp{({x}b)\gpb{g1}({e6}t)}
\blp{ ({t5}b)--({y}b)}
\alp{({y}b)--({ex}t)}
\dblp{({t5}t)\gpt{g1}({e6}b)}
\end{scope}
\end{tikzpicture}};

\draw[->](Z2)--(Z3)node [midway , fill=white] {$\trietop$};
\draw[->](Z3)--(Y3)node [midway , fill=white] {$\gammao$};
\draw[->](Y3)--(X5)node [midway , fill=white] {$\trietop^{-1}$};
\end{tikzpicture}}
\caption{The shadow isomorphism}\label{fig:shadow_isomorphism}
\end{subfigure}
\caption{Shadows and shadow isomorphisms in symmetric monoidal bicategories with 1-dualizable 0-cells.  (\cref{prop:shadow})}
\end{figure}
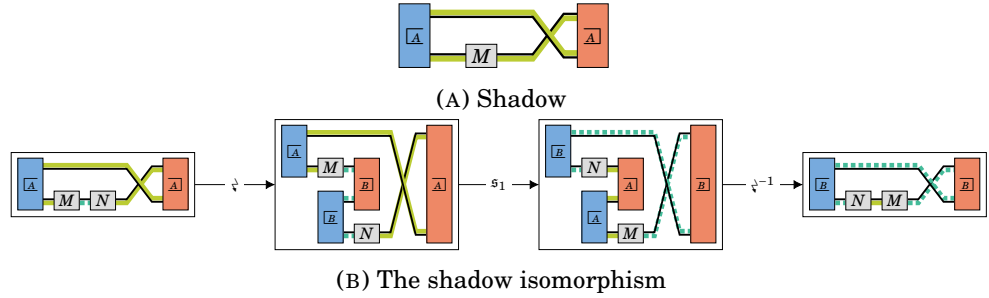

If $(\coo{A},\evo{A})$ and $(\coo{A}',\evo{A}')$ are witnessing 1-cells for a 1-dualizable 0-cell $A$  the maps $\trictop$ and $\trietop$ for each pair define isomorphisms 
\[\coo{A}\to \coo{A}'\quad \text{and} \quad \evo{A}\to \evo{A}'\]
These isomorphisms also show that the shadows defined by $(\coo{A},\evo{A})$ and $(\coo{A}',\evo{A}')$ are isomorphic and the isomorphisms compatible with the trace construction.

\begin{prop}\label{lem:comp_1_duals_shad_trace}
If $(\coo{A},\evo{A})$ and $(\coo{A}',\evo{A}')$ are witnessing 1-cells for a 1-dualizable 0-cell $A$, and $Q\in \sB(A,A)$, there is an  isomorphism 
\begin{equation}\label{eq:comp_1_duals_shad_trace}
\phi\colon \sh{Q}\to \sh{Q}'
\end{equation}
natural with respect to 2-cells in $\sB(A,A)$.  

If $M$ is right dualizable and  $f\colon Q\odot M\to M\odot P$, the diagram 
\begin{equation}\label{small-fig:comp_1_duals_trace}
\xymatrix{\sh{Q}\ar[r]^-\phi\ar[d]^{\tr(f)}&\sh{Q}'\ar[d]^{\tr'(f)}
\\
\sh{P}\ar[r]^-\phi&\sh{P}'}
\end{equation}
commutes. 
\end{prop}

\begin{proof}
The isomorphism $\phi$ is a special case of the shadow isomorphism in \cref{fig:shadow_isomorphism}.
See \cref{fig:comp_1_duals_shad} 
for explicit definition.

To see the diagram in \eqref{small-fig:comp_1_duals_trace} commutes, we expand this diagram to that in \cref{fig:comp_1_duals_trace}.  The top and bottom composites are the map in \cref{fig:comp_1_duals_shad}  and the left and right are the traces where the shadow isomorphisms have been expanded using the definition in \cref{fig:shadow_isomorphism}.
All small regions in  \cref{fig:comp_1_duals_trace} notated with Nat. commute by naturality.  This is the vast majority of the squares.  The exceptions are the squares labeled with (1) and (2).  These commute by the coherence theorem for symmetric monoidal bicategories \cite[Thm.~1.25]{gurski_osorno}.
\end{proof}

\begin{figure}
\begin{subfigure}{.8\textwidth}
\resizebox{\textwidth}{!}{\begin{tikzpicture}

\node[draw] (Z2) at (0, 4){\begin{tikzpicture}

\oc{t1}{\coo{A}}{_1}{1}{5}{4}{\et}
\oc{x}{M}{}{2}{4}{4}{\en}

\oc{ex}{\evo{A}}{_x}{4}{5}{4}{\ep}
\gc{g1}{3}{5}{4}

  \begin{scope}[on background layer]

\alp{ ({t1}b)--({x}b)}
\alp{({x}b)\gpb{g1}({ex}t)}
\dlp{({t1}t)\gpt{g1}({ex}b)}
\end{scope}
\end{tikzpicture}};

\node[draw] (Z3) at (5, 4){\begin{tikzpicture}

\oc{t1}{\coo{A}}{_1}{1}{5}{4}{\et}
\oc{x}{M}{}{2}{4}{4}{\en}
\oc{t5}{\coo{A}'}{_5}{2}{3}{2}{\et\dk}
\oc{e6}{\evo{A}'}{_6}{3}{4}{3}{\ep\dk}

\oc{ex}{\evo{A}}{_x}{5}{5}{2}{\ep}
\gc{g1}{4}{5}{2}

  \begin{scope}[on background layer]
\dlp{({t5}t)--({e6}b)}

\alp{({t1}b)--({x}b)}
\alp{({x}b)--({e6}t)}
\alp{ ({t5}b)\gpb{g1}({ex}t)}
\dlp{({t1}t)\gpt{g1}({ex}b)}
\end{scope}
\end{tikzpicture}};

\node[draw] (X5) at (15, 4){\begin{tikzpicture}

\oc{x}{M}{}{3}{4}{4}{\en}
\oc{t5}{\coo{A}'}{_5}{2}{5}{4}{\et\dk}
\oc{e6}{\evo{A}'}{_6}{5}{5}{4}{\ep\dk}
\gc{g1}{4}{5}{4}
  \begin{scope}[on background layer]
\alp{({t5}b)--({x}b)}
\alp{({x}b)\gpb{g1}({e6}t)}
\dlp{({t5}t)\gpt{g1}({e6}b)}
\end{scope}
\end{tikzpicture}};

\node[draw] (Y3) at (10, 4){\begin{tikzpicture}

\oc{t1}{\coo{A}}{_1}{2}{3}{2}{\et}
\oc{x}{M}{}{3}{2}{2}{\en}
\oc{t5}{\coo{A}'}{_5}{2}{5}{4}{\et\dk}
\oc{e6}{\evo{A}'}{_6}{5}{5}{2}{\ep\dk}

\oc{ex}{\evo{A}}{_x}{3}{4}{3}{\ep}
\gc{g1}{4}{5}{2}
  \begin{scope}[on background layer]
\dlp{ ({t1}t)--({ex}b)}

\alp{ ({t1}b)--({x}b)}
\alp{({x}b)\gpb{g1}({e6}t)}
\dlp{({t5}t)\gpt{g1}({e6}b)}
\alp{ ({t5}b)--({ex}t)}
\end{scope}
\end{tikzpicture}};

\draw[->](Z2)--(Z3)node [midway , fill=white] {$\trietop$};
\draw[->](Z3)--(Y3)node [midway , fill=white] {$\gammao$};
\draw[->](Y3)--(X5)node [midway , fill=white] {$\trietop^{-1}$};
\end{tikzpicture}}
\caption{Isomorphisms between shadows for pairs of witnessing 1-cells}\label{fig:comp_1_duals_shad}
\end{subfigure}
\begin{subfigure}{\textwidth}
\resizebox{\textwidth}{!}{\input{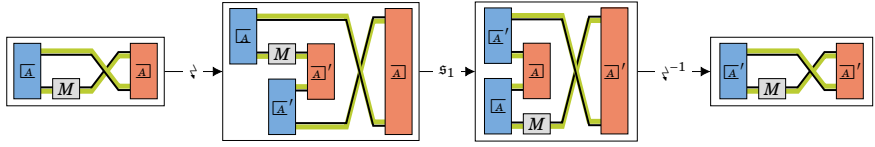}
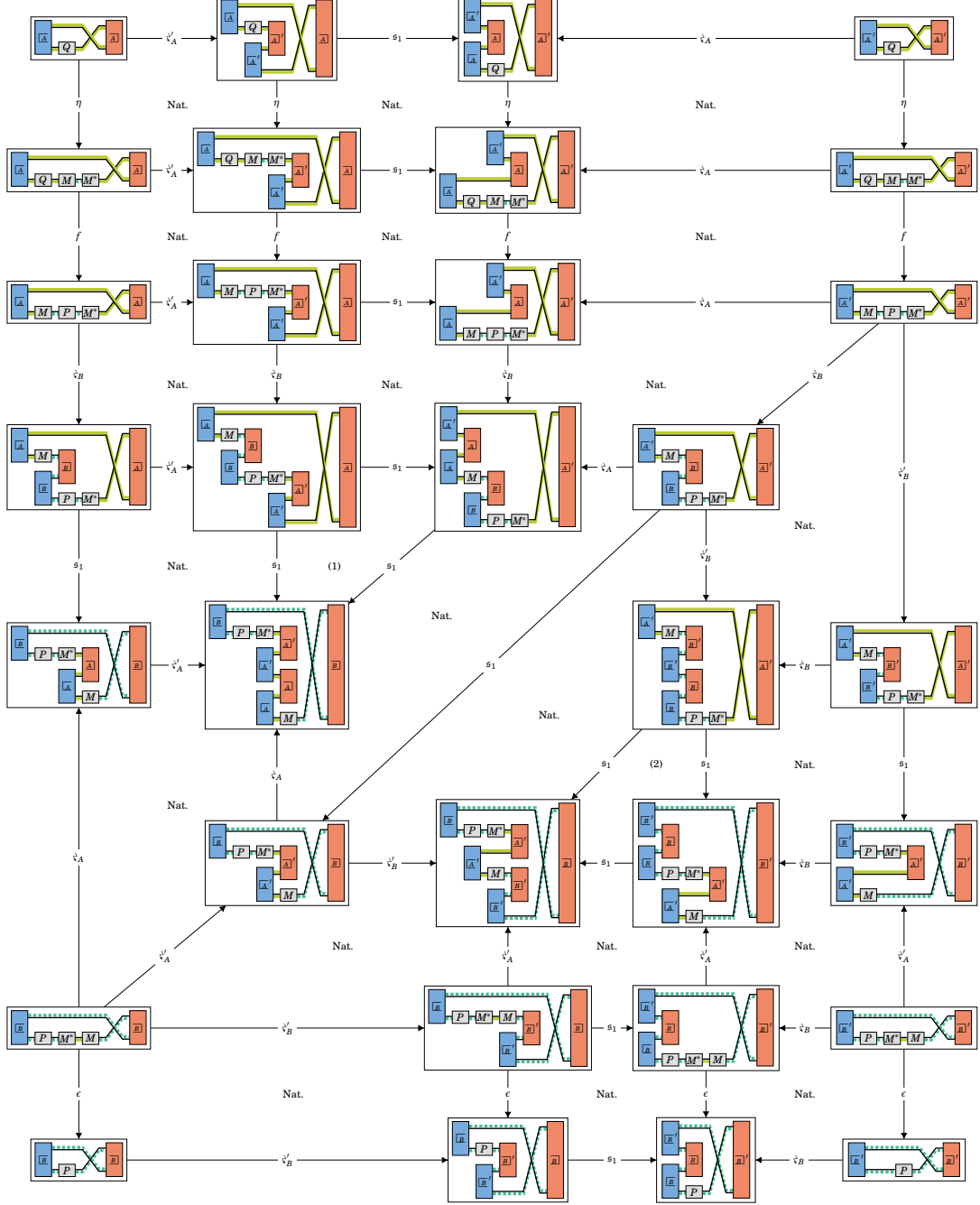}
\caption{Compatibility of traces for pairs of witnessing 1-cells}\label{fig:comp_1_duals_trace}
\end{subfigure}
\caption{Compatibility of shadows and traces with different pairs of witnessing 1-cells.  (\cref{lem:comp_1_duals_shad_trace})}
\end{figure}

Since there will not be a unique choice of witnessing 1-cells for a 1-dualizable 0-cell we will avoid introducing these isomorphisms in what follows by choosing a fixed but arbitrary dual and  witnessing 1-cells for each 1-dualizable 0-cell.  An exception to this is \cref{lem:ec_e_and_c}, where the choice of dual and 1-cells is essential to the way we have set up the proof.

\subsection{Monoidal product and dualizablity}
\cref{lem:duals_compose,thm:composite}  set up an expectation for compatibility between monoidal structures and duality.  In this section we describe how 1-dualizability interacts with the monoidal product in a symmetric monoidal bicategory. The results here exactly parallel \cref{lem:duals_compose,thm:composite} in their statements and proofs.  

\begin{lem}[Compare to  \cref{lem:duals_compose}]\label{lem:one_dual_composites}
If $A$ and $B$ are 1-dualizable then $A\otimes B$ is 1-dualizable.   
\end{lem}

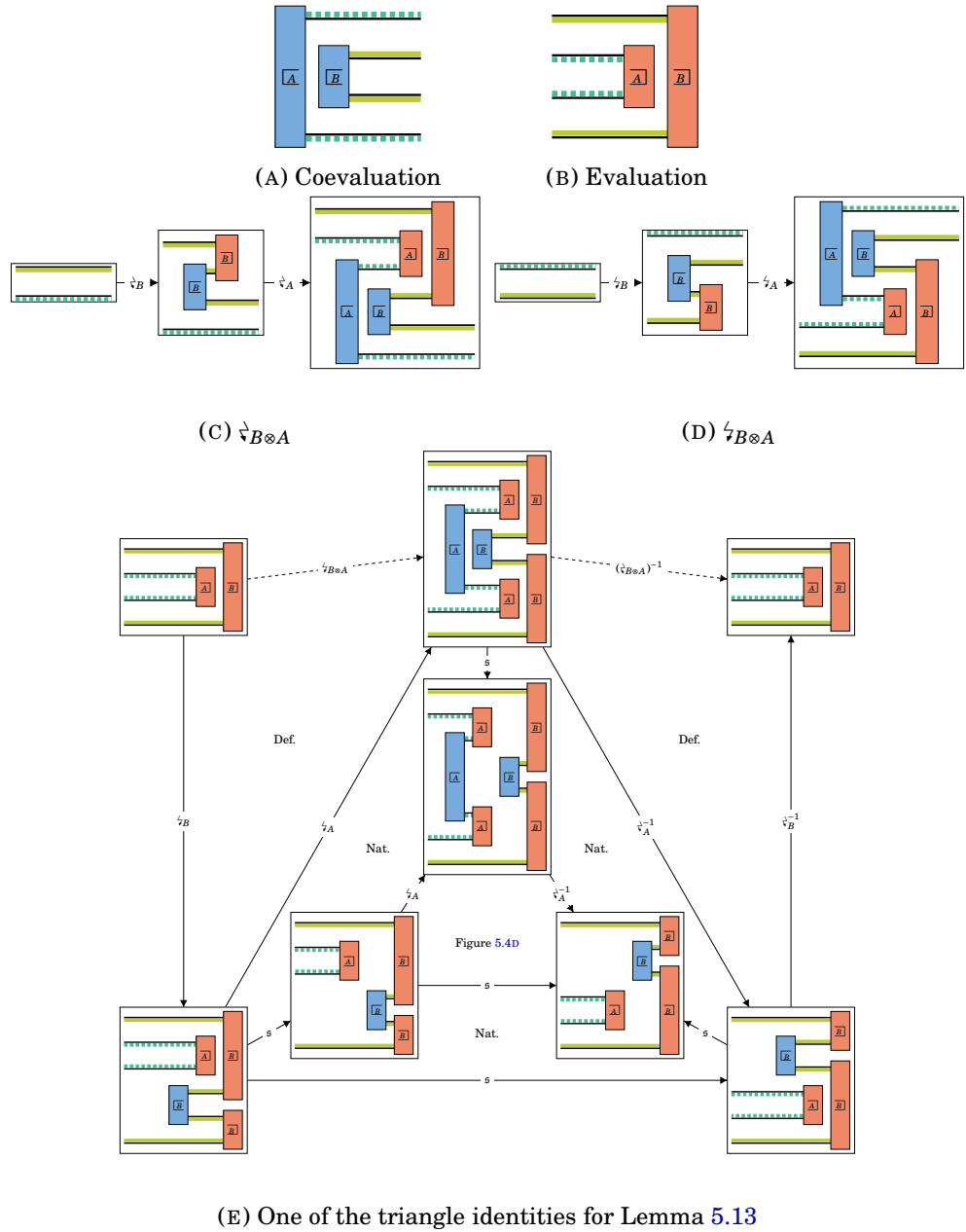
\begin{figure}[h!]
\begin{subfigure}{.18\textwidth}
\centering
\resizebox{.8\textwidth}{!}{
\begin{tikzpicture}

\oc{e1}{\coo{B}}{_1}{-1}{-2}{-3}{\et\lt}
\oc{t2}{\coo{A}}{_2}{-2}{-1}{-4}{\et}

\begin{scope}[on background layer]

\alp{({e1}b)--(1*\d,-3*\h)}
\dlp{({e1}t)--(1*\d,-2*\h)}
\blp{({t2}b)--(1*\d,-4*\h)}
\dblp{({t2}t)--(1*\d,-1*\h)}
\end{scope}
\end{tikzpicture}}
\caption{Coevaluation}\label{fig:composite_coeval}
\end{subfigure}
\hspace{1cm}
\begin{subfigure}{.18\textwidth}
\centering
\resizebox{.8\textwidth}{!}{
\begin{tikzpicture}

\oc{t7}{\evo{B}}{_7}{1}{1}{-2}{\ep\lt}
\oc{e8}{\evo{A}}{_8}{0}{0}{-1}{\ep}

\begin{scope}[on background layer]

\alp{(-2*\d,1*\h)--({t7}t)}
\dlp{(-2*\d,-2*\h)--({t7}b)}
\blp{(-2*\d,0)--({e8}t)}
\dblp{(-2*\d,-1*\h)--({e8}b)}
\end{scope}
\end{tikzpicture}}
\caption{Evaluation}\label{fig:composite_eval}
\end{subfigure}

\begin{subfigure}{.45\textwidth}
\centering
\resizebox{\textwidth}{!}
{
\begin{tikzpicture}[
    mystyle/.style={%
    },
   my style/.style={%
   },
  ]
 
\def\scale{.95}

\node[draw,my style] (A3) at (\scale*10, 0){\begin{tikzpicture}

\oc{e1}{\coo{B}}{_1}{-1}{-2}{-3}{\et\lt}
\oc{t2}{\coo{A}}{_2}{-2}{-1}{-4}{\et}

\oc{t7}{\evo{B}}{_7}{1}{1}{-2}{\ep\lt}
\oc{e8}{\evo{A}}{_8}{0}{0}{-1}{\ep}

\begin{scope}[on background layer]

\alp{(-3*\d,1*\h)--({t7}t)}
\alp{({e1}b)--(2*\d,-3*\h)}
\dlp{({e1}t)--({t7}b)}
\blp{(-3*\d,0)--({e8}t)}
\blp{({t2}b)--(2*\d,-4*\h)}
\dblp{({t2}t)--({e8}b)}
\end{scope}
\end{tikzpicture}
};

\node[draw,my style] (A4) at (\scale*1, 0){\begin{tikzpicture}

\begin{scope}[on background layer]

\alp{(-2*\d,-3*\h)--(1*\d,-3*\h)}
\blp{(-2*\d,-4*\h)--(1*\d,-4*\h)}
\end{scope}
\end{tikzpicture}
};

\node[draw,my style] (B3) at (\scale*5, 0){\begin{tikzpicture}

\oc{e1}{\coo{B}}{_1}{-1}{-2}{-3}{\et\lt}

\oc{t7}{\evo{B}}{_7}{0}{-1}{-2}{\ep\lt}

\begin{scope}[on background layer]

\alp{(-2*\d,-1*\h)--({t7}t)}
\alp{({e1}b)--(1*\d,-3*\h)}
\dlp{({e1}t)--({t7}b)}
\blp{(-2*\d,-4*\h)--(1*\d,-4*\h)}
\end{scope}
\end{tikzpicture}
};

\draw[->](A4)--(B3)node [midway , fill=white] {$\trietop_B$};
\draw[->](B3)--(A3)node [midway , fill=white] {$\trietop_A$};

\end{tikzpicture}
}
\caption{$\trietop_{B\otimes A}$}
\label{fig:one_dual_invertible_2_cell_e}
\end{subfigure} 
\begin{subfigure}{.45\textwidth}
\centering
\resizebox{\textwidth}{!}
{
\begin{tikzpicture}[
    mystyle/.style={%
    },
   my style/.style={%
   },
  ]
 
\def\scale{.95}
\node[draw,my style] (A2) at (1*\scale, 0){\begin{tikzpicture}
\begin{scope}[on background layer]

\dlp{(-2*\d,-2*\h)--(1*\d,-2*\h)}
\dblp{(-2*\d,-1*\h)--(1*\d,-1*\h)}
\end{scope}
\end{tikzpicture}
};

\node[draw,my style] (A3) at (\scale*10, 0){\begin{tikzpicture}

\oc{e1}{\coo{B}}{_1}{-1}{-2}{-3}{\et\lt}
\oc{t2}{\coo{A}}{_2}{-2}{-1}{-4}{\et}

\oc{t3}{\evo{B}}{_3}{1}{-3}{-6}{\ep\lt}
\oc{e4}{\evo{A}}{_4}{0}{-4}{-5}{\ep}

\begin{scope}[on background layer]

\alp{({e1}b)--({t3}t)}
\dlp{({t3}b)--(-3*\d,-6*\h)}
\dlp{({e1}t)--(2*\d,-2*\h)}
\blp{({t2}b)--({e4}t)}
\dblp{({e4}b)--(-3*\d,-5*\h)}
\dblp{({t2}t)--(2*\d,-1*\h)}
\end{scope}
\end{tikzpicture}
};

\node[draw,my style] (B2) at (\scale*5,0){\begin{tikzpicture}

\oc{e1}{\coo{B}}{_1}{0}{-2}{-3}{\et\lt}

\oc{t3}{\evo{B}}{_3}{1}{-3}{-4}{\ep\lt}

\begin{scope}[on background layer]

\alp{({e1}b)--({t3}t)}
\dlp{({t3}b)--(-1*\d,-4*\h)}
\dlp{({e1}t)--(2*\d,-2*\h)}
\dblp{(-1*\d,-1*\h)--(2*\d,-1*\h)}
\end{scope}
\end{tikzpicture}
};

\draw[->](A2)--(B2)node [midway , fill=white] {$\trictop_B$};
\draw[->](B2)--(A3)node [midway , fill=white] {$\trictop_A$};

\end{tikzpicture}
}
\caption{$\trictop_{B\otimes A}$}
\label{fig:one_dual_invertible_2_cell_c}
\end{subfigure} 
\begin{subfigure}{.70\textwidth}
\resizebox{\textwidth}{!}
{
\begin{tikzpicture}[
    mystyle/.style={%
    },
   my style/.style={%
   },
  ]
 
\def\scale{1.15}

\node[draw,my style] (A2) at (0, 0){\begin{tikzpicture}

\oc{t7}{\evo{B}}{_7}{1}{1}{-2}{\ep\lt}
\oc{e8}{\evo{A}}{_8}{0}{0}{-1}{\ep}

\begin{scope}[on background layer]

\alp{(-3*\d,1*\h)--({t7}t)}
\dlp{(-3*\d,-2*\h)--({t7}b)}
\blp{(-3*\d,0)--({e8}t)}
\dblp{(-3*\d,-1*\h)--({e8}b)}
\end{scope}
\end{tikzpicture}
};

\node[draw,my style] (A3) at (\scale*8, \scale*1){\begin{tikzpicture}

\oc{e1}{\coo{B}}{_1}{-1}{-2}{-3}{\et\lt}
\oc{t2}{\coo{A}}{_2}{-2}{-1}{-4}{\et}

\oc{t7}{\evo{B}}{_7}{1}{1}{-2}{\ep\lt}
\oc{e8}{\evo{A}}{_8}{0}{0}{-1}{\ep}

\oc{t3}{\evo{B}}{_3}{1}{-3}{-6}{\ep\lt}
\oc{e4}{\evo{A}}{_4}{0}{-4}{-5}{\ep}

\begin{scope}[on background layer]

\alp{(-3*\d,1*\h)--({t7}t)}
\alp{({e1}b)--({t3}t)}
\dlp{({t3}b)--(-3*\d,-6*\h)}
\dlp{({e1}t)--({t7}b)}
\blp{(-3*\d,0)--({e8}t)}
\blp{({t2}b)--({e4}t)}
\dblp{({e4}b)--(-3*\d,-5*\h)}
\dblp{({t2}t)--({e8}b)}
\end{scope}
\end{tikzpicture}
};

\node[draw,my style] (A3a) at (\scale*8, \scale*-5){\begin{tikzpicture}

\oc{e1}{\coo{B}}{_1}{0}{-2}{-3}{\et\lt}
\oc{t2}{\coo{A}}{_2}{-2}{-1}{-4}{\et}

\oc{t7}{\evo{B}}{_7}{1}{1}{-2}{\ep\lt}
\oc{e8}{\evo{A}}{_8}{-1}{0}{-1}{\ep}

\oc{t3}{\evo{B}}{_3}{1}{-3}{-6}{\ep\lt}
\oc{e4}{\evo{A}}{_4}{-1}{-4}{-5}{\ep}

\begin{scope}[on background layer]

\alp{(-3*\d,1*\h)--({t7}t)}
\alp{({e1}b)--({t3}t)}
\dlp{({t3}b)--(-3*\d,-6*\h)}
\dlp{({e1}t)--({t7}b)}
\blp{(-3*\d,0)--({e8}t)}
\blp{({t2}b)--({e4}t)}
\dblp{({e4}b)--(-3*\d,-5*\h)}
\dblp{({t2}t)--({e8}b)}
\end{scope}
\end{tikzpicture}
};

\node[draw,my style] (B2a) at (\scale*4.5, \scale*-10.5){\begin{tikzpicture}

\oc{e1}{\coo{B}}{_1}{0}{-2}{-3}{\et\lt}

\oc{t7}{\evo{B}}{_7}{1}{1}{-2}{\ep\lt}
\oc{e8}{\evo{A}}{_8}{-1}{0}{-1}{\ep}

\oc{t3}{\evo{B}}{_3}{1}{-3}{-4}{\ep\lt}

\begin{scope}[on background layer]

\alp{(-3*\d,1*\h)--({t7}t)}
\alp{({e1}b)--({t3}t)}
\dlp{({t3}b)--(-3*\d,-4*\h)}
\dlp{({e1}t)--({t7}b)}
\blp{(-3*\d,0)--({e8}t)}
\dblp{({e8}b)--(-3*\d,-1*\h)}
\end{scope}
\end{tikzpicture}
};

\node[draw,my style] (B3a) at (\scale*11.5, \scale*-10.5){\begin{tikzpicture}

\oc{e1}{\coo{B}}{_1}{0}{-2}{-3}{\et\lt}

\oc{t7}{\evo{B}}{_7}{1}{-1}{-2}{\ep\lt}

\oc{t3}{\evo{B}}{_3}{1}{-3}{-6}{\ep\lt}
\oc{e4}{\evo{A}}{_4}{-1}{-4}{-5}{\ep}

\begin{scope}[on background layer]

\alp{(-3*\d,-1*\h)--({t7}t)}
\alp{({e1}b)--({t3}t)}
\dlp{({t3}b)--(-3*\d,-6*\h)}
\dlp{({e1}t)--({t7}b)}
\blp{(-3*\d,-4*\h)--({e4}t)}
\dblp{({e4}b)--(-3*\d,-5*\h)}
\end{scope}
\end{tikzpicture}
};

\node[draw,my style] (A4) at (\scale*16, 0){\begin{tikzpicture}

\oc{t3}{\evo{B}}{_3}{1}{1}{-2}{\ep\lt}
\oc{e4}{\evo{A}}{_4}{0}{0}{-1}{\ep}

\begin{scope}[on background layer]

\alp{(-3*\d,1*\h)--({t3}t)}
\dlp{(-3*\d,-2*\h)--({t3}b)}
\blp{(-3*\d,0)--({e4}t)}
\dblp{(-3*\d,-1*\h)--({e4}b)}
\end{scope}
\end{tikzpicture}
};

\node[draw,my style] (B2) at (\scale*0,\scale*-13){\begin{tikzpicture}

\oc{e1}{\coo{B}}{_1}{-1}{-2}{-3}{\et\lt}

\oc{t7}{\evo{B}}{_7}{1}{1}{-2}{\ep\lt}
\oc{e8}{\evo{A}}{_8}{0}{0}{-1}{\ep}

\oc{t3}{\evo{B}}{_3}{1}{-3}{-4}{\ep\lt}

\begin{scope}[on background layer]

\alp{(-3*\d,1*\h)--({t7}t)}
\alp{({e1}b)--({t3}t)}
\dlp{({t3}b)--(-3*\d,-4*\h)}
\dlp{({e1}t)--({t7}b)}
\blp{(-3*\d,0)--({e8}t)}
\dblp{(-3*\d,-1*\h)--({e8}b)}
\end{scope}
\end{tikzpicture}
};

\node[draw,my style] (B3) at (\scale*16,\scale* -13){\begin{tikzpicture}

\oc{e1}{\coo{B}}{_1}{-1}{-2}{-3}{\et\lt}

\oc{t7}{\evo{B}}{_7}{1}{-1}{-2}{\ep\lt}

\oc{t3}{\evo{B}}{_3}{1}{-3}{-6}{\ep\lt}
\oc{e4}{\evo{A}}{_4}{0}{-4}{-5}{\ep}

\begin{scope}[on background layer]

\alp{(-3*\d,-1*\h)--({t7}t)}
\alp{({e1}b)--({t3}t)}
\dlp{({t3}b)--(-3*\d,-6*\h)}
\dlp{({e1}t)--({t7}b)}
\blp{(-3*\d,-4*\h)--({e4}t)}
\dblp{({e4}b)--(-3*\d,-5*\h)}
\end{scope}
\end{tikzpicture}
};

\draw[dashed,->](A2)--(A3)node [midway , fill=white] {$\trictop_{B\otimes A}$};
\draw[dashed,->](A3)--(A4)node [midway , fill=white] {$(\trietop_{B\otimes A})^{-1}$};

\draw[->](B2)--(B3)node [midway , fill=white] {$\gammas$};
\draw[->](A3)--(A3a)node [midway , fill=white] {$\gammas$};
\draw[->](B2)--(B2a)node [midway , fill=white] {$\gammas$};
\draw[->](B3)--(B3a)node [midway , fill=white] {$\gammas$};
\draw[->](B2a)--(B3a)node [midway , fill=white] {$\gammas$};
\draw[->](A2)--(B2)node [midway , fill=white] {$\trictop_B$};
\draw[->](A3)--(B3)node [midway , fill=white] {$\trietop^{-1}_A$};
\draw[->](A3a)--(B3a)node [midway , fill=white] {$\trietop^{-1}_A$};
\draw[->](B2)--(A3)node [midway , fill=white] {$\trictop_A$};
\draw[->](B2a)--(A3a)node [midway , fill=white] {$\trictop_A$};
\draw[->](B3)--(A4)node [midway , fill=white] {$\trietop^{-1}_B$};

\node at  (barycentric cs:B2a=1,B2=1,A3a=1,A3=1){Nat.};
\node at  (barycentric cs:B3a=1,B3=1,A3a=1,A3=1){Nat.};
\node at  (barycentric cs:B2a=1,B2=1,B3a=1,B3=1){Nat.};
\node at  (barycentric cs:B2a=1,B3a=1,A3a=.5){\cref{fig:add_dualizable_0_cell_a}};
\node at  (barycentric cs:B2=1,A2=1,A3=1){Def.};
\node at  (barycentric cs:A4=1,B3=1,A3=1){Def.};
\end{tikzpicture}
}
\caption{One of the triangle identities for \cref{lem:one_dual_composites}}\label{fig:one_dual_composites}
\end{subfigure}
\caption{Compatibility of the monoidal product and 1-dualizability (\cref{lem:one_dual_composites})}
\end{figure}

\begin{proof}
Using string diagram calculus for symmetric monoidal categories or bicategories, \cref{lem:duals_compose} is shown by ``nesting'' the evaluation and coevaluation maps for $A$ and $B$.  The proof of \cref{lem:one_dual_composites} is given by rotating that picture.
(If we had chosen a different orientation for our circuit diagrams the proofs would be more similar.)
  
The coevaluation and evaluation are in \cref{fig:composite_coeval,fig:composite_eval}.  Note that the order is reversed between the coevaluation and evaluation.
The invertible 2-cells are in \cref{fig:one_dual_invertible_2_cell_e,fig:one_dual_invertible_2_cell_c}.  

The triangle diagrams from \cref{fig:add_dualizable_0_cell} is in \cref{fig:one_dual_composites}.  The top two regions of that diagram commute by definition of $(\trietop_{B\otimes A})^{-1}$ and $\trictop_{B\otimes A}$.  The regions labeled with Nat. commute by naturality.  The remaining region is an example of \cref{fig:add_dualizable_0_cell_a}.  The left, bottom and right composite is the identity since it is also an example of \cref{fig:add_dualizable_0_cell_a}.

The other triangle diagram is essentially a reflection of this diagram.
\end{proof}

\begin{prop}[Compare to  \cref{thm:composite}]\label{lem:trace_otimes_compatible}
If $A,B,C$ and $D$ are dualizable 0-cells, $M\in \sB(A,B)$ and $N\in \sB(C,D)$ are dualizable 1-cells, and $f\colon Q_1\odot M\to M\odot P_1$ and $g\colon Q_2\odot N\to N\odot P_1$ are 2-cells, the diagram in \cref{fig:trace_otimes_compatible} commutes.

\end{prop}

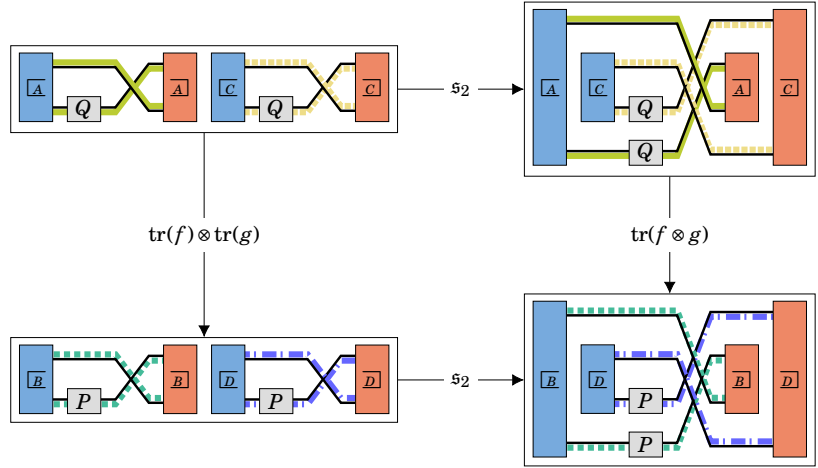
\begin{figure}[h!]
\resizebox{.75\textwidth}{!}
{
\begin{tikzpicture}[
    mystyle/.style={%
    },
   my style/.style={%
   },
  ]

\node[draw,my style] (A1) at (10, -1){\begin{tikzpicture}
	\oc{e3}{\coo{C}}{_3}{1}{2}{1}{\et\lt}
	\oc{t4}{\coo{A}}{_4}{0}{3}{0}{\et\lt}
	\oc{t5}{\evo{C}}{_5}{5}{3}{0}{\ep\lt}
	\oc{e6}{\evo{A}}{_6}{4}{2}{1}{\ep\lt}

	\oc{q2}{Q}{_2}{2}{1}{1}{\en}
	\oc{q1}{Q}{_1}{2}{0}{0}{\en}
	\begin{scope}[on background layer]
	\clp{({e3}b)--({q2}t)--(3*\d-\w,\h)--(3*\d+\w,3*\h)--({t5}t)}
	\alp{({t4}b)--({q1}t)--(3*\d-\w,0*\h)--(3*\d+\w,2*\h)--({e6}t)}
	\dlp{({t4}t)--(3*\d-\w,3*\h)--(3*\d+\w,1*\h)--({e6}b)}
	\dclp{({e3}t)--(3*\d-\w,2*\h)--(3*\d+\w,0*\h)--({t5}b)}
	\end{scope}
	\end{tikzpicture}
	};

\node[draw,my style] (A3) at (10, -6){\begin{tikzpicture}
	\oc{e1}{\coo{D}}{_1}{-6}{0}{-1}{\et}
	\oc{t2}{\coo{B}}{_2}{-7}{1}{-2}{\et}
	\oc{t7}{\evo{D}}{_7}{-2}{1}{-2}{\ep}
	\oc{e8}{\evo{B}}{_8}{-3}{0}{-1}{\ep}
	
	\oc{p2}{P}{_2}{-5}{-1}{-1}{\en}
	\oc{p1}{P}{_1}{-5}{-2}{-2}{\en}
	\begin{scope}[on background layer]
	\elp{({e1}b)--({p2}t)--(-4*\d-\w,-1*\h)--(-4*\d+\w,1*\h)--({t7}t)}
	\blp{({t2}b)--({p1}t)--(-4*\d-\w,-2*\h)--(-4*\d+\w,0*\h)--({e8}t)}
	\dblp{({t2}t)--(-4*\d-\w,1*\h)--(-4*\d+\w,-1*\h)--({e8}b)}
	\delp{({e1}t)--(-4*\d-\w,0*\h)--(-4*\d+\w,-2*\h)--({t7}b)}
	\end{scope}
	\end{tikzpicture}
	};

\node[draw,my style] (B1) at (2,-1){\begin{tikzpicture}
	\oc{e3}{\coo{C}}{_3}{1}{2}{1}{\et\lt}
	\oc{t4}{\coo{A}}{_4}{-3}{2}{1}{\et\lt}
	\oc{t5}{\evo{C}}{_5}{4}{2}{1}{\ep\lt}
	\oc{e6}{\evo{A}}{_6}{0}{2}{1}{\ep\lt}
	\oc{q2}{Q}{_2}{2}{1}{1}{\en}
	\oc{q1}{Q}{_1}{-2}{1}{1}{\en}
	\begin{scope}[on background layer]
	\clp{({e3}b)--({q2}t)--(3*\d-\w,\h)--(3*\d+\w,2*\h)--
	({t5}t)}
	\alp{({t4}b)--({q1}t)--(-1*\d-\w,1*\h)--(-1*\d+\w,2*\h)--
	({e6}t)}
	\dlp{({t4}t)--(-1*\d-\w,2*\h)--(-1*\d+\w,1*\h)--
	({e6}b)}
	\dclp{({e3}t)--(3*\d-\w,2*\h)--(3*\d+\w,1*\h)--
	({t5}b)}
	\end{scope}
	\end{tikzpicture}
	};

\node[draw,my style] (B3) at (2, -6) {\begin{tikzpicture}
	\oc{e1}{\coo{D}}{_1}{-4}{0}{-1}{\et}
	\oc{t2}{\coo{B}}{_2}{-8}{0}{-1}{\et}
	\oc{t7}{\evo{D}}{_7}{-1}{0}{-1}{\ep}
	\oc{e8}{\evo{B}}{_8}{-5}{0}{-1}{\ep}

	\oc{p2}{P}{_2}{-3}{-1}{-1}{\en}
	\oc{p1}{P}{_1}{-7}{-1}{-1}{\en}
	\begin{scope}[on background layer]
	
	\elp{({e1}b)--({p2}t)--(-2*\d-\w,-1*\h)--(-2*\d+\w,0*\h)--({t7}t)}
	\delp{({e1}t)--(-2*\d-\w,0*\h)--(-2*\d+\w,-1*\h)--({t7}b)}
	\blp{({t2}b)--({p1}t)--(-6*\d-\w,-1*\h)--(-6*\d+\w,0*\h)--({e8}t)}
	\dblp{({t2}t)--(-6*\d-\w,0*\h)--(-6*\d+\w,-1*\h)--({e8}b)}
	\end{scope}
	\end{tikzpicture}
	};
\draw[->](B1)--(A1)node [midway , fill=white] {$\gammat$};
\draw[->](B3)--(A3)node [midway , fill=white] {$\gammat$};
\draw[->](B1)--(B3)node [midway , fill=white] {$\tr_{}(f)\otimes \tr_{}(g)$};
\draw[->](A1)--(A3)node [midway , fill=white] {$\tr(f\otimes g)$};

\end{tikzpicture}
}
\caption{Compatibility of the trace and 1-dualizability (\cref{lem:trace_otimes_compatible})} \label{fig:trace_otimes_compatible}
\end{figure}

As with \cref{lem:duals_compose,lem:one_dual_composites}, this proof 
is a rotation of the proof of \cref{thm:composite}.  The string diagram proof of \cref{thm:composite} in either the bicategorical or symmetric monoidal version amounts to sliding disconnected strings past each other.

\begin{proof}
Using the evaluation and coevaluation from  \cref{lem:one_dual_composites} (see  \cref{fig:composite_coeval,fig:composite_eval}), the right vertical composite in  \cref{fig:trace_otimes} is the expansion of  $\tr(f\otimes g)$.  This is also the right map in \cref{fig:trace_otimes_compatible}.  The left composite in \cref{fig:trace_otimes} is the expansion of the left map in \cref{fig:trace_otimes_compatible}.

The square regions in  \cref{fig:trace_otimes_compatible} commute by naturality of the symmetry isomorphism and the regions containing the shadow isomorphism $\theta$ commute by definition (\cref{fig:shadow_isomorphism}).

Finally, for endomorphisms of the unit, composition can be identified with the monoidal product using the Eckman-Hilton argument. 
\end{proof}

\begin{figure}
\resizebox{!}{.8\textheight}{
\input{RR-pairing_euler_2}}
\caption{The diagram that expands \cref{fig:trace_otimes_compatible}. }\label{fig:trace_otimes}
\end{figure}

The following is another simple consequence of \cref{lem:one_dual_composites}.

\begin{lem}\label{lem:euler_char_gamma}
If $A$ and $B$ are 1-dualizable 0-cells, 
the Euler characteristic of $\Gamma\in \sB(A,B)$ (\cref{ex:define_Gamma}) is the isomorphism $\gammath$ from \cref{gamma_list_2}.
\end{lem}

\begin{proof}
In the diagram in  \cref{proof:euler_char_gamma} the Euler characteristic of $\Gamma$ is the composite consisting of the left, bottom, and right maps.  The top is the symmetry map $\gammath$.  Five of the regions in the diagram in  \cref{proof:euler_char_gamma} commute by naturality of the symmetry maps.  These are all labeled by Nat.  Two of the squares commute by the compatibility between $\trieto{}$ and $\tricto{}$ from \cref{fig:add_dualizable_0_cell,fig:add_dualizable_0_cell_a} and are labeled as such.  The remaining two regions, labeled (1) and (2), are examples of \cite[Thm.~1.25]{gurski_osorno}.
\end{proof}
\begin{figure}
\resizebox{.9\textwidth}{!}{
\input{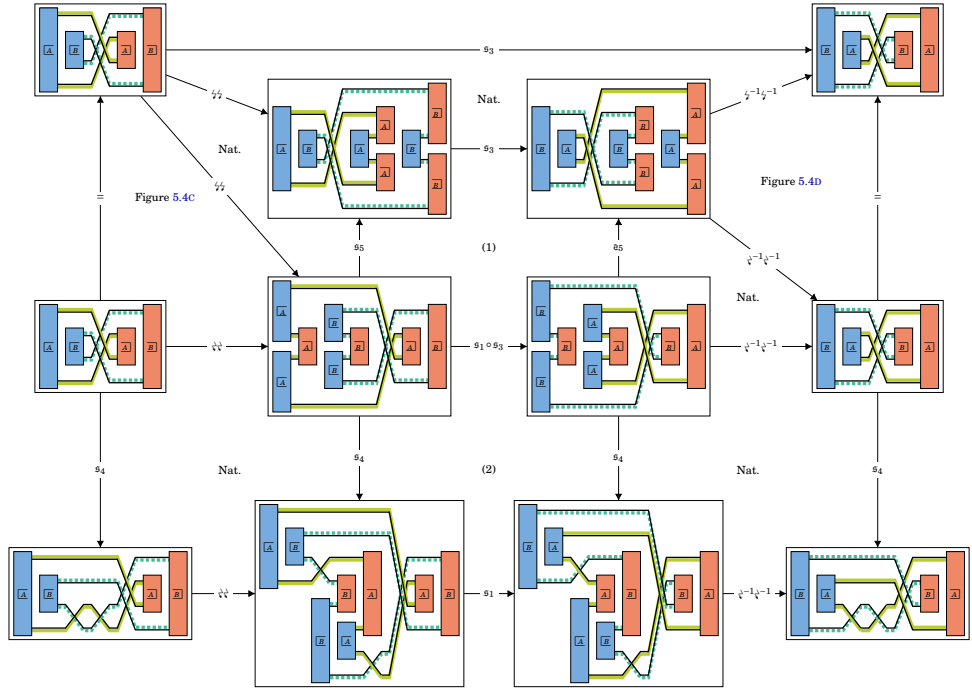}}
\caption{Expanded Euler characteristic comparison for  $\Gamma$ }\label{proof:euler_char_gamma}
\end{figure}

\section{2-Dualizability}\label{sec:2_dualizable}\label{sec:euler_pairing}

As in \cite{cp2} we are interested in examples where $\sh{N}$ is
dualizable for a dualizable 1-cell $N$.  Given our motivating
examples, the most relevant way to achieve this is to impose another
condition on our 0-cells. 

In the definition of 1-dualizability, we have coevaluation and evaluation 1-cells and the cells in \eqref{eq:trictop} and \eqref{eq:trietop} that witness the usual triangle equalities. For 2-dualizability,
the extra input will be that the 1-cells $I \to A \otimes A^\vee$ and
$A^\vee \otimes A \to I$ are \textit{themselves} dualizable in the
ambient bicategory. 

A crucial consequence of 1-dualizability for every 0-cell in a symmetric monoidal
bicategory is the existence of a shadow functor
(\cref{prop:shadow}). Similarly, a crucial consequence of 2-dualizability
for 0-cells is the dualizability of $\sh{U_A}$ in the monoidal
category $\mc{B}(I, I)$ 
(\cref{cor:shadow_dual}) giving the pairing and ``copairing'' maps.  In this section we also prove \cref{thm:goal} as  \cref{thm:main_pairing}.   The only differences between these statements are in the notation used.

\begin{defn}\cite[6.15, 6.16]{cp2}\label{defn:2_dualizable}
A 1-dualizable $0$-cell $A$ is {\bf 2-dualizable} if the witnessing 1-cells $(\coo{A},\evo{A})$ satisfy any of the following equivalent conditions:
\begin{enumerate}
\item $\coo{A}$ and $\evo{A}$ are left dualizable
\item $\coo{A}$ is left and right dualizable
\item $\coo{A}$ and $\evo{A}$ are right dualizable 
\item $\evo{A}$ is left and right dualizable
\end{enumerate} 
\end{defn}
In addition to the 1-cells $\coo{A}$ and $\evo{A}$ and 2-cells $\trictop$ and $\trietop$, a 2-dualizable 0-cell has 1-cells $\rdual{\coo{A}}$ and $\rdual{\evo{A}}$, the right duals of  $\coo{A}$ and $\evo{A}$, and four 2-cells illustrated in  \cref{fig:coeval_1_cell_coev,fig:coeval_1_cell_ev,fig:coeval_1_cell_coev_flip,fig:coeval_1_cell_ev_flip}.  These 2-cells also have to satisfy the triangle identities.  In this case,  the four composites in \cref{fig:coeval_1_cell_triangles,fig:coeval_1_cell_2_triangles} are identity 2-cells.

This is essentially a special case of \emph{fully dualizable} objects in an $n$-category \cite{schommer-pries,lurie_cobordism}.

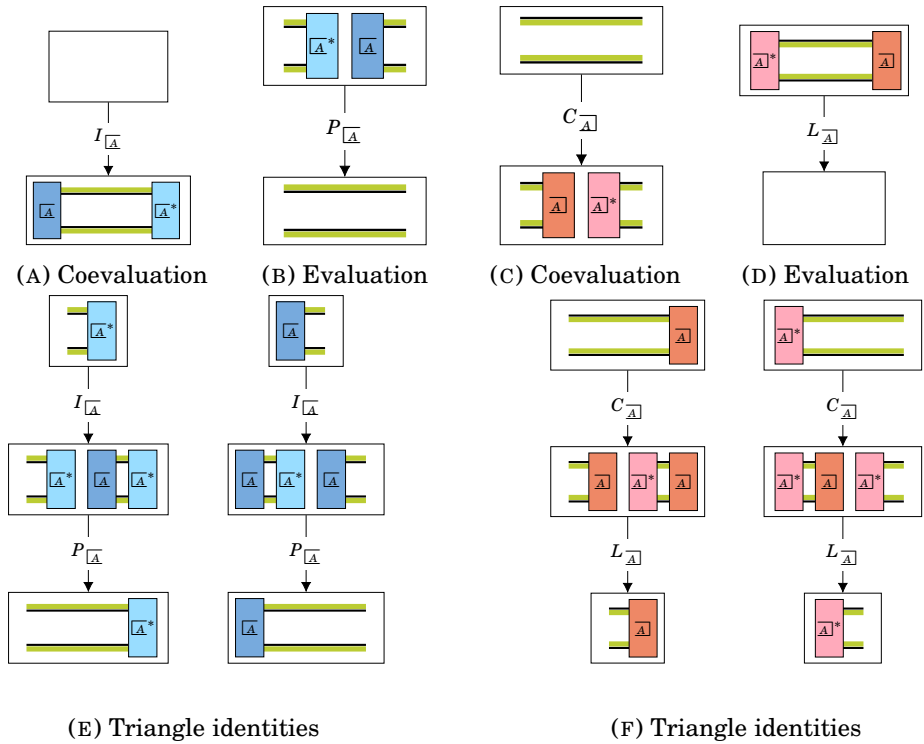
\begin{figure}
\hspace{1cm}
    \begin{subfigure}[t]{0.2\textwidth}
\centering
\resizebox{.8\textwidth}{!}{
		\begin{tikzpicture}

		\node[draw] (A1) at (0, 3){\begin{tikzpicture}
		\node (ta) at (2,-1){};
		\node (ta) at (4,-2){};
		\end{tikzpicture}
		};

		\node[draw] (A2) at (0, 0){\begin{tikzpicture}
		\oc{ta}{\coo{A}}{_a}{2}{-1}{-2}{\et}
		\oc{t2}{\rdual{\coo{A}}}{_2}{5}{-1}{-2}{\etr}
		  \begin{scope}[on background layer]
		\dlp{ ({ta}t)--({t2}t)}
		\alp{({ta}b)--({t2}b)}
		\end{scope}
\end{tikzpicture}
};

\draw[->](A1)--(A2)node [midway , fill=white] {$I_{\coo{A}}$};
\end{tikzpicture}
}
        \caption{Coevaluation}\label{fig:coeval_1_cell_coev}
    \end{subfigure}%
\hfill 
    \begin{subfigure}[t]{0.2\textwidth}
\centering
\resizebox{.8\textwidth}{!}{
\begin{tikzpicture}

\node[draw] (A1) at (0, 3){\begin{tikzpicture}
\oc{t1}{\coo{A}}{_1}{6}{-1}{-2}{\et}
\ocalt{y}{}{Y}{7}{-1}{-2}{\en}
\ocalt{x}{}{X}{4}{-1}{-2}{\en}

\oc{t2}{\rdual{\coo{A}}}{_2}{5}{-1}{-2}{\etr}
  \begin{scope}[on background layer]
\dlp{({x}t)--({t2}t)}
\dlp{({t1}t)--({y}t)}
\alp{({t1}b)--({y}b)}

\alp{({x}b)--({t2}b)}
\end{scope}
\end{tikzpicture}
};

\node[draw] (A2) at (0, 0){\begin{tikzpicture}
\ocalt{y}{}{Y}{7}{-1}{-2}{\en}
\ocalt{x}{}{X}{4}{-1}{-2}{\en}
  \begin{scope}[on background layer]
\dlp{({x}t)--({y}t)}
\alp{({x}b)--({y}b)}
\end{scope}
\end{tikzpicture}
};

\draw[->](A1)--(A2)node [midway , fill=white] {$P_{\coo{A}}$};
\end{tikzpicture}
}
     
        \caption{Evaluation}\label{fig:coeval_1_cell_ev}
    \end{subfigure}
\hfill
    \begin{subfigure}[t]{0.2\textwidth}
\centering
        \resizebox{.8\textwidth}{!}{
\begin{tikzpicture}

\node[draw] (A1) at (0, -3){\begin{tikzpicture}
\oc{e1}{\rdual{\evo{A}}}{_1}{6}{-1}{-2}{\epr}
\ocalt{y}{}{Y}{7}{-1}{-2}{\en}
\ocalt{x}{}{X}{4}{-1}{-2}{\en}

\oc{e2}{{\evo{A}}}{_2}{5}{-1}{-2}{\ep}
  \begin{scope}[on background layer]
\dlp{({x}b)--({t2}b)}
\dlp{ ({t1}b)--({y}b)}
\alp{ ({t1}t)--({y}t)}
\alp{({x}t)--({t2}t)}
\end{scope}
\end{tikzpicture}
};

\node[draw] (A2) at (0, 0){\begin{tikzpicture}
\ocalt{y}{}{Y}{7}{-1}{-2}{\en}
\ocalt{x}{}{X}{4}{-1}{-2}{\en}
  \begin{scope}[on background layer]
\dlp{({x}b)--({y}b)}
\alp{({x}t)--({y}t)}
\end{scope}
\end{tikzpicture}
};

\draw[->](A2)--(A1)node [midway , fill=white] {$C_{\evo{A}}$};
\end{tikzpicture}
 }
       \caption{Coevaluation}\label{fig:coeval_1_cell_coev_flip}
    \end{subfigure}
\hfill 
    \begin{subfigure}[t]{0.2\textwidth}
\centering
\resizebox{.8\textwidth}{!}{
        \centering
		\begin{tikzpicture}

		\node[draw] (A1) at (0, -3){\begin{tikzpicture}
		\node (ta) at (2,-1){};
		\node (ta) at (4,-2){};
		\end{tikzpicture}
		};

		\node[draw] (A2) at (0, 0){\begin{tikzpicture}
		\oc{ea}{\rdual{\evo{A}}}{_a}{2}{-1}{-2}{\epr}
		\oc{e2}{{\evo{A}}}{_2}{5}{-1}{-2}{\ep}
	  \begin{scope}[on background layer]
		\dlp{ ({ea}b)--({e2}b)}
		\alp{({ea}t)--({e2}t)}
	\end{scope}
\end{tikzpicture}
};

\draw[->](A2)--(A1)node [midway , fill=white] {$L_{\evo{A}}$};
\end{tikzpicture}}
       \caption{Evaluation}\label{fig:coeval_1_cell_ev_flip}
    \end{subfigure}%
\hspace{1cm}
\hspace{1cm}
\begin{subfigure}[t]{0.35\textwidth}
       \resizebox{\textwidth}{!}{
\begin{tikzpicture}

\node[draw] (A1) at (0, 6){\begin{tikzpicture}
\ocalt{x}{}{X}{4}{-1}{-2}{\en}

\oc{t2}{\rdual{\coo{A}}}{_2}{5}{-1}{-2}{\etr}
  \begin{scope}[on background layer]
\dlp{({x}t)--({t2}t)}
\alp{({x}b)--({t2}b)}
\end{scope}
\end{tikzpicture}
};

\node[draw] (A2) at (0, 3){\begin{tikzpicture}
\oc{t1}{\coo{A}}{_1}{6}{-1}{-2}{\et}
\oc{t3}{\rdual{\coo{A}}}{_3}{7}{-1}{-2}{\etr\lt}
\ocalt{x}{}{X}{4}{-1}{-2}{\en}

\oc{t2}{\rdual{\coo{A}}}{_2}{5}{-1}{-2}{\etr}
  \begin{scope}[on background layer]
\dlp{ ({x}t)--({t2}t)}
\dlp{ ({t1}t)--({t3}t)}
\alp{({t1}b)--({t3}b)}
\alp{({x}b)--({t2}b)}
\end{scope}
\end{tikzpicture}
};

\node[draw] (A3) at (0, 0){\begin{tikzpicture}
\oc{t3}{\rdual{\coo{A}}}{_3}{7}{-1}{-2}{\etr\lt}
\ocalt{x}{}{X}{4}{-1}{-2}{\en}
  \begin{scope}[on background layer]
\dlp{({x}t)--({t3}t)}
\alp{({x}b)--({t3}b)}
\end{scope}
\end{tikzpicture}
};

\draw[->](A1)--(A2)node [midway , fill=white] {$I_{\coo{A}}$};
\draw[->](A2)--(A3)node [midway , fill=white] {$P_{\coo{A}}$};
\end{tikzpicture}
\hspace{1cm}
\begin{tikzpicture}

\node[draw] (A1) at (0, 6){\begin{tikzpicture}
\oc{t1}{\coo{A}}{_1}{6}{-1}{-2}{\et}
\ocalt{y}{}{Y}{7}{-1}{-2}{\en}
  \begin{scope}[on background layer]
\dlp{({t1}t)--({y}t)}
\alp{({t1}b)--({y}b)}
\end{scope}
\end{tikzpicture}
};

\node[draw] (A2) at (0, 3){\begin{tikzpicture}
\oc{t1}{\coo{A}}{_1}{6}{-1}{-2}{\et}
\ocalt{y}{}{Y}{7}{-1}{-2}{\en}
\oc{t3}{\coo{A}}{_3}{4}{-1}{-2}{\et\lt}

\oc{t2}{\rdual{\coo{A}}}{_2}{5}{-1}{-2}{\etr}
  \begin{scope}[on background layer]
\dlp{ ({t3}t)--({t2}t)}
\dlp{({t1}t)--({y}t)}
\alp{({t1}b)--({y}b)}
\alp{({t3}b)--({t2}b)}
\end{scope}
\end{tikzpicture}
};

\node[draw] (A3) at (0, 0){\begin{tikzpicture}
\ocalt{y}{}{Y}{7}{-1}{-2}{\en}
\oc{t3}{\coo{A}}{_3}{4}{-1}{-2}{\et\lt}
  \begin{scope}[on background layer]
\dlp{ ({t3}t)--({y}t)}
\alp{({t3}b)--({y}b)}
\end{scope}
\end{tikzpicture}
};

\draw[->](A1)--(A2)node [midway , fill=white] {$I_{\coo{A}}$};
\draw[->](A2)--(A3)node [midway , fill=white] {$P_{\coo{A}}$};
\end{tikzpicture}}
        \caption{Triangle identities}\label{fig:coeval_1_cell_triangles}
    \end{subfigure}
\hfill
\begin{subfigure}[t]{0.35\textwidth}
\resizebox{\textwidth}{!}{
\begin{tikzpicture}
\node[draw] (A2) at (0, 0){\begin{tikzpicture}
\oc{e3}{\evo{A}}{_3}{7}{-1}{-2}{\ep\lt}
\ocalt{x}{}{X}{4}{-1}{-2}{\en}
  \begin{scope}[on background layer]
\dlp{({x}b)--({e3}b)}
\alp{({x}t)--({e3}t)}
\end{scope}
\end{tikzpicture}
};

\node[draw] (A1) at (0, -3){\begin{tikzpicture}
\oc{e1}{\rdual{\evo{A}}}{_1}{6}{-1}{-2}{\epr}
\oc{e3}{\evo{A}}{_3}{7}{-1}{-2}{\ep\lt}
\ocalt{x}{}{X}{4}{-1}{-2}{\en}

\oc{e2}{{\evo{A}}}{_2}{5}{-1}{-2}{\ep}
  \begin{scope}[on background layer]
\dlp{({x}b)--({t2}b)}
\dlp{ ({t1}b)--({e3}b)}
\alp{ ({t1}t)--({e3}t)}
\alp{({x}t)--({t2}t)}
\end{scope}
\end{tikzpicture}
};

\node[draw] (A3) at (0, -6){\begin{tikzpicture}
\ocalt{x}{}{X}{4}{-1}{-2}{\en}
\oc{e2}{{\evo{A}}}{_2}{5}{-1}{-2}{\ep}

  \begin{scope}[on background layer]
\dlp{({x}b)--({t2}b)}
\alp{({x}t)--({t2}t)}
\end{scope}
\end{tikzpicture}
};

\draw[->](A2)--(A1)node [midway , fill=white] {$C_{\evo{A}}$};
\draw[->](A1)--(A3)node [midway , fill=white] {$L_{\evo{A}}$};
\end{tikzpicture}
\hspace{1cm}
\begin{tikzpicture}
\node[draw] (A2) at (0, 0){\begin{tikzpicture}
\ocalt{y}{}{Y}{7}{-1}{-2}{\en}
\oc{e3}{\rdual{\evo{A}}}{_3}{4}{-1}{-2}{\epr\lt}

  \begin{scope}[on background layer]
\dlp{({e3}b)--({y}b)}
\alp{({e3}t)--({y}t)}
\end{scope}
\end{tikzpicture}
};

\node[draw] (A1) at (0, -3){\begin{tikzpicture}
\oc{e1}{\rdual{\evo{A}}}{_1}{6}{-1}{-2}{\epr}
\ocalt{y}{}{Y}{7}{-1}{-2}{\en}
\oc{e3}{\rdual{\evo{A}}}{_3}{4}{-1}{-2}{\epr\lt}

\oc{e2}{{\evo{A}}}{_2}{5}{-1}{-2}{\ep}

  \begin{scope}[on background layer]
\dlp{({e3}b)--({t2}b)}
\dlp{({t1}b)--({y}b)}
\alp{({t1}t)--({y}t)}
\alp{({e3}t)--({t2}t)}
\end{scope}
\end{tikzpicture}
};

\node[draw] (A3) at (0, -6){\begin{tikzpicture}
\oc{e1}{\rdual{\evo{A}}}{_1}{6}{-1}{-2}{\epr}
\ocalt{y}{}{Y}{7}{-1}{-2}{\en}

  \begin{scope}[on background layer]
\dlp{({t1}b)--({y}b)}
\alp{({t1}t)--({y}t)}
\end{scope}
\end{tikzpicture}
};

\draw[->](A2)--(A1)node [midway , fill=white] {$C_{\evo{A}}$};
\draw[->](A1)--(A3)node [midway , fill=white] {$L_{\evo{A}}$};
\end{tikzpicture}
}
        \caption{Triangle identities}\label{fig:coeval_1_cell_2_triangles}
    \end{subfigure}
\hspace{1cm}
\caption{Coevaluation, evaluation and triangle identities for right dualizability of 1-cells 
in $\sB(1, A\otimes B)$ and $\sB(A\otimes B,1)$.  Compare to \cref{fig:dualizable_1_cell}.}
\label{fig:coeval_1_cell_2}\label{fig:coeval_1_cell}
\end{figure}

\begin{example}
\begin{enumerate}
\item 
A dg-$k$-algebra $B$ is {\bf proper} if $B$ is perfect as an object of $D(k)$ (i.e. $B$ is perfect when regarded as a chain complex of $k$ modules).
It is {\bf smooth} if $B$ is perfect as an object of $D(B \otimes^{\mathbf{L}}_k B^\op)$ \cite[2.3]{toen}.

\item  If $\mc{A}$ is a dg-category (respectively spectral category), the unit 1-cell $U_\mc{A}$ can be regarded as a $(k,\mc{A}\otimes \mc{A})$-bimodule (respectively $(\bS,\mc{A}\otimes \mc{A})$-bimodule).  Denote this module $\overrightarrow{U}_\mc{A}$.

A  a dg-category (respectively spectral category) 
 $\mc{A}$ is {\bf proper} if $\overrightarrow{U}_\mc{A}$ is left dualizable.  It is {\bf smooth} if $\overrightarrow{U}_\mc{A}$ is right dualizable. 
For spectral or dg-categories, $\mc{A}$ is 2-dualizable if and only if it is smooth and proper.  See \cite[Thm.~5.8]{cisinski_tabuada_dg}.
\end{enumerate}
\end{example}

The following result is a consequence of \cref{lem:duals_compose} and the diagrams in \cref{fig:dualizable_1_cell_symmetry}.

\begin{cor}[{\cref{goal:dualizable}, \cite[5.2 and 6.18]{cp2}}]\label{cor:shadow_dual}
If $A$ is a 2-dualizable 0-cell  in a monoidal bicategory $\sB$ 
then $\sh{U_A}$ is dualizable  in the monoidal category $\sB(I,I)$
\end{cor}

\begin{proof}
By definition, $\sh{U_A}$ is the composite of $\coo{A}$, $U_A$, $\Gamma$, and $\evo{A}$ in \cref{fig:shadow}. In any bicategory, $U_A$ is dualizable with dual $U_A$ and the coevaluation and evaluation are unit isomorphisms.  If $A$ is 2-dualizable $\coo{A}$ and $\evo{A}$ are dualizable by definition.  The 1-cell $\Gamma$ is dualizable by \cref{fig:dualizable_1_cell_symmetry}.  Then the result follows from  \cref{lem:duals_compose}.
\end{proof}

We can describe the coevaluation and evaluation for $\sh{U_A}$ in terms of those for each of the pieces $\coo{A}$, $U_A$, $\Gamma$, and $\evo{A}$.  
The coevaluation is the composite in 
\cref{fig:copairing_defn_fig} and the evaluation is the composite in 
\cref{fig:pairing_defn_fig}.  Note that we can omit the coevaluation and evaluation for $U_A$ since it is the unit isomorphism.
  Following the discussion in \cref{sec:path} we will call the coevaluation a ``copairing'' and the evaluation a ``pairing''. 

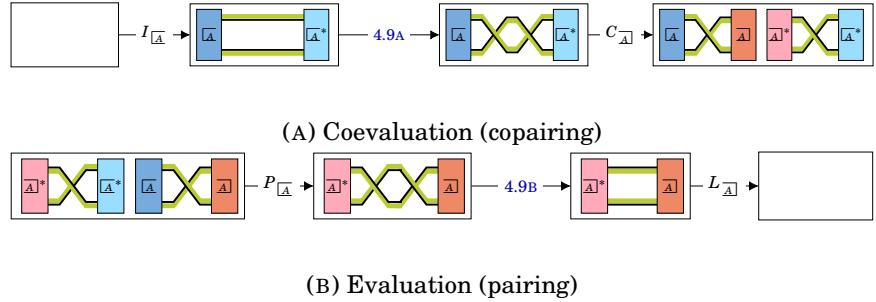
\begin{figure}[ht]
\begin{subfigure}{.8\textwidth}
\resizebox{\textwidth}{!}{
\begin{tikzpicture}

\def\scale{1.15}

\node[draw] (Z1) at (\scale*14, 0){\begin{tikzpicture}
\oc{t1}{\coo{A}}{_1}{1}{1}{0}{\et}
\oc{e2}{\evo{A}}{_2}{3}{1}{0}{\ep}
\gc{g1}{2}{1}{0}
\begin{scope}[on background layer]
\alp{ ({t1}b)\gpb{g1}({e2}t)}
\dlp{({t1}t)\gpt{g1}({e2}b)}
\end{scope}
\oc{e3}{\rdual{\evo{A}}}{_3}{4}{1}{0}{\epr}
\oc{t4}{\rdual{\coo{A}}}{_4}{6}{1}{0}{\etr}
\gc{g2}{5}{1}{0}
\begin{scope}[on background layer]
\dlp{({e3}b)\gpb{g2}({t4}t)}
\alp{ ({e3}t)\gpt{g2}({t4}b)}
\end{scope}
\end{tikzpicture}
};

\node[draw] (Z2) at (\scale*9, 0){\begin{tikzpicture}
\oc{t1}{\coo{A}}{_1}{1}{1}{0}{\et}
\gc{g1}{2}{1}{0}
\oc{t4}{\rdual{\coo{A}}}{_4}{4}{1}{0}{\etr}
\gc{g2}{3}{1}{0}
\begin{scope}[on background layer]

\alp{ ({t1}b)\gpb{g1}(2.5*\d,\h)
}
\dlp{({t1}t)\gpt{g1}(2.5*\d,0) --(2.5*\d,0)\gpb{g2}({t4}t)}
\alp{
(2.5*\d,\h)\gpt{g2}({t4}b)}
\end{scope}
\end{tikzpicture}
};

\node[draw] (Z3) at (\scale*4, 0){\begin{tikzpicture}
\oc{t1}{\coo{A}}{_1}{1}{1}{0}{\et}
\oc{t4}{\rdual{\coo{A}}}{_4}{4}{1}{0}{\etr}
\begin{scope}[on background layer]
\alp{ ({t1}b)--({t4}b)}
\dlp{({t1}t)--({t4}t)}
\end{scope}
\end{tikzpicture}
};

\node[draw] (Z4) at (\scale*0, 0){\begin{tikzpicture}
		\node (ta) at (2,-1){};
		\node (ta) at (4,-2){};
		\end{tikzpicture}};

\draw[<-](Z1)--(Z2)node [midway , fill=white] {$C_{\evo{A}}$};
\draw[<-](Z2)--(Z3)node [midway , fill=white] {\ref{fig:dualizable_1_cell_symmetry_1}};

\draw[<-](Z3)--(Z4)node [midway , fill=white] {$I_{\coo{A}}$};
\end{tikzpicture}
}
\caption{Coevaluation  (copairing)}\label{fig:copairing_defn_fig}
\end{subfigure}
\begin{subfigure}{.8\textwidth}
\resizebox{\textwidth}{!}{
\begin{tikzpicture}

\def\scale{1.15}

\node[draw] (A1) at (\scale*-5, \scale*1){\begin{tikzpicture}
\oc{t1}{\coo{A}}{_1}{1}{1}{0}{\et}
\oc{e2}{\evo{A}}{_2}{3}{1}{0}{\ep}
\gc{g1}{2}{1}{0}
\begin{scope}[on background layer]
\alp{ ({t1}b)\gpb{g1}({e2}t)}
\dlp{({t1}t)\gpt{g1}({e2}b)}
\end{scope}
\oc{e3}{\rdual{\evo{A}}}{_3}{-2}{1}{0}{\epr}
\oc{t4}{\rdual{\coo{A}}}{_4}{0}{1}{0}{\etr}
\gc{g2}{-1}{1}{0}
\begin{scope}[on background layer]
\dlp{({e3}b)\gpb{g2}({t4}t)}
\alp{ ({e3}t)\gpt{g2}({t4}b)}
\end{scope}
\end{tikzpicture}
};

\node[draw] (A2) at (\scale*0, \scale*1){\begin{tikzpicture}
\oc{e2}{\evo{A}}{_2}{1}{1}{0}{\ep}
\gc{g1}{0}{1}{0}

\oc{e3}{\rdual{\evo{A}}}{_3}{-2}{1}{0}{\epr}
\gc{g2}{-1}{1}{0}
\begin{scope}[on background layer]
\dlp{ ({e3}b)\gpb{g2}(-.25,\h)
}

\alp{ ({e3}t)\gpt{g2}(-.25,0)\gpb{g1}({e2}t)}
\dlp{
(-.25,\h)\gpt{g1}({e2}b)}
\end{scope}
\end{tikzpicture}
};

\node[draw] (A3) at (\scale*4.5, \scale*1){\begin{tikzpicture}
\oc{e2}{\evo{A}}{_2}{0}{1}{0}{\ep}

\oc{e3}{\rdual{\evo{A}}}{_3}{-2}{1}{0}{\epr}
\begin{scope}[on background layer]
\dlp{({e3}b)--({e2}b)}
\alp{ ({e3}t)--({e2}t)}
\end{scope}
\end{tikzpicture}
};

\node[draw] (A4) at (\scale*8, \scale*1){\begin{tikzpicture}
		\node (ta) at (2,-1){};
		\node (ta) at (4,-2){};
		\end{tikzpicture}};

\draw[->](A1)--(A2)node [midway , fill=white] {$P_{\coo{A}}$};
\draw[->](A2)--(A3)node [midway , fill=white] {\ref{fig:dualizable_1_cell_symmetry_2}};

\draw[->](A3)--(A4)node [midway , fill=white] {$L_{\evo{A}}$};
\end{tikzpicture}
}
\caption{Evaluation (pairing)}\label{fig:pairing_defn_fig}
\end{subfigure}
\caption{Copairing and pairing maps}\label{fig:copairing_pairing_defn_fig}
\end{figure}

\cref{lem:compatiblity_1-witnesses_if_2-dual} 
 is a fundamental observation from \cite{cp2}.  Most of this result appeared in the proof of \cite[Theorem 6.11]{cp2}, but since  the definition of 1-dualizability is slightly different in this paper (\cref{rmk:cp_1_dual_change}) 
we will formally state and prove the additional conditions required here.

\begin{lem}\label{lem:compatiblity_1-witnesses_if_2-dual}
If $A$ is 2-dualizable then $\zdual{A}$ is 1-dualizable with witnessing 1-cells $(\rdual{\evo{A}},\rdual{\coo{A}})$.
\end{lem}

\begin{proof}
The required 2-cells are defined in \cref{fig:dual_triangle_a,fig:dual_triangle_b} and their inverses are defined in \cref{fig:dual_triangle_1a,fig:dual_triangle_1b}.  See \cite[Figure 18]{cp2} for the proof that these are the inverses.  

The diagram in \cref{fig:dual_triangle_2} is an expansion of  \cref{fig:add_dualizable_0_cell} where $\trictop{}$ and $\trietop{}$ are replaced by $\tridetop{}$ and $\tridctop{}$ respectively.  The regions labeled by Def. commute by definition (\cref{fig:dual_triangle_a,fig:dual_triangle_1b}).  The nine regions labeled by Nat. commute by naturality.   The remaining two regions are triangle identities for the dualizability of $\coo{A}$ and $\evo{A}$.

The diagram showing \cref{fig:add_dualizable_0_cell_a} holds is similar.
\end{proof}

\begin{figure}
\hspace{1cm}
\begin{subfigure}{.35\textwidth}
\resizebox{\textwidth}{!}
{\begin{tikzpicture}
\def\yscale{1}
\def\xscale{1.2}

\node[draw] (A1) at (\xscale*5, \yscale*3.5){\begin{tikzpicture}
\ocalt{y}{}{Y}{6}{0}{-2}{\en}
\ocalt{x}{}{X}{1}{0}{-2}{\en}

  \begin{scope}[on background layer]
\alp{ ({x}t)--({y}t)}
\end{scope}
\end{tikzpicture}
};

\node[draw] (B1) at (\xscale*10.5, \yscale*3.5){\begin{tikzpicture}
\ocalt{y}{}{Y}{6}{0}{-2}{\en}
\ocalt{x}{}{X}{1}{0}{-2}{\en}

\oc{ta}{\coo{A}}{_a}{2}{-1}{-2}{\et}
\oc{t2}{\rdual{\coo{A}}}{_2}{5}{-1}{-2}{\etr}

  \begin{scope}[on background layer]
\dlp{({ta}t)--({t2}t)}
\alp{({x}t)--({y}t)}

\alp{({ta}b)--({t2}b)}
\end{scope}
\end{tikzpicture}
};

\node[draw] (B2) at (\xscale*10.5, \yscale*0){\begin{tikzpicture}
\ocalt{y}{}{Y}{6}{0}{-2}{\en}
\ocalt{x}{}{X}{1}{0}{-2}{\en}

\oc{ta}{\coo{A}}{_a}{2}{-1}{-2}{\et}
\oc{t2}{\rdual{\coo{A}}}{_2}{5}{-1}{-2}{\etr}

\oc{e4}{\rdual{\evo{A}}}{_4}{4}{0}{-1}{\epr}
\oc{e3}{\evo{A}}{_3}{3}{0}{-1}{\ep}

  \begin{scope}[on background layer]
\dlp{({ta}t)--({e3}b)}
\dlp{({e4}b)--({t2}t)}
\alp{({x}t)--({e3}t)}
\alp{ ({e4}t)--({y}t)}

\alp{({ta}b)--({t2}b)}
\end{scope}
\end{tikzpicture}
};

\node[draw] (B3) at (\xscale*5, \yscale*0){\begin{tikzpicture}
\ocalt{y}{}{Y}{6}{0}{-2}{\en}
\ocalt{x}{}{X}{1}{0}{-2}{\en}
\oc{t2}{\rdual{\coo{A}}}{_2}{5}{-1}{-2}{\etr}
\oc{e4}{\rdual{\evo{A}}}{_4}{4}{0}{-1}{\epr}

  \begin{scope}[on background layer]
\dlp{({e4}b)--({t2}t)}
\alp{({x}b)--({t2}b)}
\alp{({e4}t)--({y}t)}
\end{scope}
\end{tikzpicture}
};

\draw[->](A1)--(B1)node [midway , fill=white] {$I_{\coo{A}}$};
\draw[->](B1)--(B2)node [midway , fill=white] {$C_{\evo{A}}$};

\draw[->](B2)--(B3)node [midway , fill=white] {$\trietop^{-1}$};
\draw[dashed, ->](A1)--(B3)node [midway , fill=white] {$\tridetop$};

\end{tikzpicture}}
\caption{$\tridetop$}\label{fig:dual_triangle_a}
\end{subfigure}
\hfill 
\begin{subfigure}{.35\textwidth}
\resizebox{\textwidth}{!}
{\begin{tikzpicture}

\def\xscale{1.2}

\node[draw] (A2) at (\xscale*0, 3.5){\begin{tikzpicture}
\oc{t1}{\rdual{\coo{A}}}{_1}{5}{-1}{-2}{\etr}
\ocalt{x}{}{X}{3}{0}{-2}{\en}
\ocalt{y}{}{Y}{8}{0}{-2}{\en}
\oc{e3}{\rdual{\evo{A}}}{_3}{4}{0}{-1}{\epr}

  \begin{scope}[on background layer]
\dlp{ ({t1}t)--({e3}b)}
\alp{({t1}b)--({x}b)}
\alp{ ({y}t)--({e3}t)}
\end{scope}
\end{tikzpicture}
};

\node[draw] (A3) at (\xscale*0, 0){\begin{tikzpicture}
\ocalt{x}{}{X}{3}{0}{-2}{\en}
\ocalt{y}{}{Y}{8}{0}{-2}{\en}

  \begin{scope}[on background layer]
\alp{ ({y}b)--({x}b)}
\end{scope}
\end{tikzpicture}
};

\node[draw] (B1) at (\xscale*5.5, 3.5){\begin{tikzpicture}
\oc{t1}{\rdual{\coo{A}}}{_1}{5}{-1}{-2}{\etr}
\ocalt{x}{}{X}{3}{0}{-2}{\en}
\ocalt{y}{}{Y}{8}{0}{-2}{\en}
\oc{e3}{\rdual{\evo{A}}}{_3}{4}{0}{-1}{\epr}

\oc{e5}{\evo{A}}{_5}{7}{0}{-1}{\ep}
\oc{t6}{\coo{A}}{_6}{6}{-1}{-2}{\et}

  \begin{scope}[on background layer]
\dlp{ ({t1}t)--({e3}b)}
\dlp{({t6}t)--({e5}b)}
\alp{({t1}b)--({x}b)}
\alp{({y}b)--({t6}b)}
\alp{ ({e5}t)--({e3}t)}
\end{scope}
\end{tikzpicture}
};

\node[draw] (B2) at (\xscale*5.5,0){\begin{tikzpicture}
\ocalt{x}{}{X}{3}{0}{-2}{\en}
\ocalt{y}{}{Y}{8}{0}{-2}{\en}
\oc{e3}{\rdual{\evo{A}}}{_3}{4}{0}{-1}{\epr}
\oc{e5}{\evo{A}}{_5}{7}{0}{-1}{\ep}

  \begin{scope}[on background layer]
\dlp{({e3}b)--({e5}b)}
\alp{({y}b)--({x}b)}
\alp{ ({e5}t)--({e3}t)}
\end{scope}
\end{tikzpicture}
};

\draw[->](A2)--(B1)node [midway , fill=white] {$\trietop$};
\draw[->](B1)--(B2)node [midway , fill=white] {$P_{\coo{A}}$};

\draw[dashed, ->](A2)--(A3)node [midway , fill=white] {${\tridetop}^{-1}$};

\draw[->](B2)--(A3)node [midway , fill=white] {$L_{\evo{A}}$}; 
\end{tikzpicture}}
\caption{$\tridetop^{-1}$}\label{fig:dual_triangle_1a}
\end{subfigure}
\hspace{1cm}
\vspace{.25cm}

\hspace{1cm}
\begin{subfigure}{.35\textwidth}
\resizebox{\textwidth}{!}
{\begin{tikzpicture}
\def\xscale{1.2}

\node[draw] (A1) at (\xscale*5, 3.5){\begin{tikzpicture}
\ocalt{y}{}{Y}{6}{0}{-2}{\en}
\ocalt{x}{}{X}{1}{0}{-2}{\en}

  \begin{scope}[on background layer]
\dlp{ ({x}b)--({y}b)}
\end{scope}
\end{tikzpicture}
};

\node[draw] (B1) at (\xscale*10.5, 3.5){\begin{tikzpicture}
\ocalt{y}{}{Y}{6}{0}{-2}{\en}
\ocalt{x}{}{X}{1}{0}{-2}{\en}

\oc{ta}{\coo{A}}{_a}{2}{0}{-1}{\et}
\oc{t2}{\rdual{\coo{A}}}{_2}{5}{0}{-1}{\etr}

  \begin{scope}[on background layer]
\alp{({ta}b)--({t2}b)}
\dlp{({x}b)--({y}b)}

\dlp{({ta}t)--({t2}t)}
\end{scope}
\end{tikzpicture}
};

\node[draw] (B2) at (\xscale*10.5, 0){\begin{tikzpicture}
\ocalt{y}{}{Y}{6}{0}{-2}{\en}
\ocalt{x}{}{X}{1}{0}{-2}{\en}

\oc{ta}{\coo{A}}{_a}{2}{0}{-1}{\et}
\oc{t2}{\rdual{\coo{A}}}{_2}{5}{0}{-1}{\etr}

\oc{e4}{\rdual{\evo{A}}}{_4}{4}{-1}{-2}{\epr}
\oc{e3}{\evo{A}}{_3}{3}{-1}{-2}{\ep}

  \begin{scope}[on background layer]
\alp{({ta}b)--({e3}t)}
\alp{({e4}t)--({t2}b)}
\dlp{({x}b)--({e3}b)}
\dlp{ ({e4}b)--({y}b)}

\dlp{({ta}t)--({t2}t)}
\end{scope}
\end{tikzpicture}
};

\node[draw] (B3) at (\xscale*5, 0){\begin{tikzpicture}
\ocalt{y}{}{Y}{6}{0}{-2}{\en}
\ocalt{x}{}{X}{1}{0}{-2}{\en}
\oc{t2}{\rdual{\coo{A}}}{_2}{5}{0}{-1}{\etr}
\oc{e4}{\rdual{\evo{A}}}{_4}{4}{-1}{-2}{\epr}

  \begin{scope}[on background layer]
\alp{({e4}t)--({t2}b)}
\dlp{({x}t)--({t2}t)}
\dlp{({e4}b)--({y}b)}
\end{scope}
\end{tikzpicture}
};

\draw[->](A1)--(B1)node [midway , fill=white] {$I_{\coo{A}}$};
\draw[->](B1)--(B2)node [midway , fill=white] {$C_{\evo{A}}$};

\draw[->](B2)--(B3)node [midway , fill=white] {$\trictop^{-1}$};
\draw[dashed, ->](A1)--(B3)node [midway , fill=white] {$\tridctop$};

\end{tikzpicture}}
\caption{$\tridctop$}\label{fig:dual_triangle_b}
\end{subfigure}
\hfill 
\begin{subfigure}{.35\textwidth}
\resizebox{\textwidth}{!}
{\begin{tikzpicture}

\def\xscale{1.2}

\node[draw] (A2) at (\xscale*0, 3.5){\begin{tikzpicture}
\oc{t1}{\rdual{\coo{A}}}{_1}{5}{0}{-1}{\etr}
\ocalt{x}{}{X}{3}{0}{-2}{\en}
\ocalt{y}{}{Y}{8}{0}{-2}{\en}
\oc{e3}{\rdual{\evo{A}}}{_3}{4}{-1}{-2}{\epr}

  \begin{scope}[on background layer]
\alp{ ({t1}b)--({e3}t)}
\dlp{({t1}t)--({x}t)}
\dlp{ ({y}b)--({e3}b)}
\end{scope}
\end{tikzpicture}
};

\node[draw] (A3) at (\xscale*0, 0){\begin{tikzpicture}
\ocalt{x}{}{X}{3}{0}{-2}{\en}
\ocalt{y}{}{Y}{8}{0}{-2}{\en}

  \begin{scope}[on background layer]
\dlp{ ({y}t)--({x}t)}
\end{scope}
\end{tikzpicture}
};

\node[draw] (B1) at (\xscale*5.5, 3.5){\begin{tikzpicture}
\oc{t1}{\rdual{\coo{A}}}{_1}{5}{0}{-1}{\etr}
\ocalt{x}{}{X}{3}{0}{-2}{\en}
\ocalt{y}{}{Y}{8}{0}{-2}{\en}
\oc{e3}{\rdual{\evo{A}}}{_3}{4}{-1}{-2}{\epr}

\oc{e5}{\evo{A}}{_5}{7}{-1}{-2}{\ep}
\oc{t6}{\coo{A}}{_6}{6}{0}{-1}{\et}

  \begin{scope}[on background layer]
\alp{ ({t1}b)--({e3}t)}
\alp{({t6}b)--({e5}t)}
\dlp{({t1}t)--({x}t)}
\dlp{({y}t)--({t6}t)}
\dlp{ ({e5}b)--({e3}b)}
\end{scope}
\end{tikzpicture}
};

\node[draw] (B2) at (\xscale*5.5,0){\begin{tikzpicture}
\ocalt{x}{}{X}{3}{1}{-1}{\en}
\ocalt{y}{}{Y}{8}{1}{-1}{\en}
\oc{e3}{\rdual{\evo{A}}}{_3}{4}{0}{-1}{\epr}
\oc{e5}{\evo{A}}{_5}{7}{0}{-1}{\ep}

  \begin{scope}[on background layer]
\alp{({e3}t)--({e5}t)}
\dlp{({y}t)--({x}t)}
\dlp{ ({e5}b)--({e3}b)}
\end{scope}
\end{tikzpicture}
};

\draw[->](A2)--(B1)node [midway , fill=white] {$\trictop$};
\draw[->](B1)--(B2)node [midway , fill=white] {$P_{\coo{A}}$};

\draw[dashed, ->](A2)--(A3)node [midway , fill=white] {${\tridctop}^{-1}$};

\draw[->](B2)--(A3)node [midway , fill=white] {$L_{\evo{A}}$}; 
\end{tikzpicture}}
\caption{$\tridctop^{-1}$}\label{fig:dual_triangle_1b}
\end{subfigure}
\hspace{1cm}
\vspace{.25cm}

\begin{subfigure}{.8\textwidth}
\resizebox{\textwidth}{!}
{\begin{tikzpicture}[
    mystyle/.style={%
    },
   my style/.style={%
   },
  ]

\def\scale{1.15}
\node[draw, my style] (X1) at (\scale*9, \scale*20){\begin{tikzpicture}
\oc{tc}{\rdual{\coo{A}}}{_c}{6}{0}{-1}{\etr\lt}
\ocalt{x}{}{X}{3}{0}{-1}{\en}

  \begin{scope}[on background layer]
\alp{ ({x}b)--({tc}b)}
\dlp{ ({x}t)--({tc}t)}
\end{scope}
\end{tikzpicture}
};

\node[draw, my style ] (Y1) at (\scale*11, \scale*14.5){\begin{tikzpicture}

\oc{tc}{\rdual{\coo{A}}}{_c}{5}{1}{0}{\etr\lt}
\ocalt{x}{}{X}{2}{1}{0}{\en}

\oc{ta}{\coo{A}}{_a}{2}{-1}{-2}{\et}
\oc{t2}{\rdual{\coo{A}}}{_2}{5}{-1}{-2}{\etr}

  \begin{scope}[on background layer]
\dlp{({ta}t)--({t2}t)}
\dlp{({x}t)--({tc}t)}
\alp{({x}b)--({tc}b)}

\alp{({ta}b)--({t2}b)}
\end{scope}
\end{tikzpicture}
};

\node[draw, my style] (Y2) at (\scale*11, \scale*10){\begin{tikzpicture}

\oc{tc}{\rdual{\coo{A}}}{_c}{5}{1}{0}{\etr\lt}
\ocalt{x}{}{X}{2}{1}{0}{\en}

\oc{ta}{\coo{A}}{_a}{2}{-1}{-2}{\et}
\oc{t2}{\rdual{\coo{A}}}{_2}{5}{-1}{-2}{\etr}

\oc{e4}{\rdual{\evo{A}}}{_4}{4}{0}{-1}{\epr}
\oc{e3}{\evo{A}}{_3}{3}{0}{-1}{\ep}

  \begin{scope}[on background layer]

\dlp{ ({x}t)--({tc}t)}
\dlp{({ta}t)--({e3}b)}
\dlp{({e4}b)--({t2}t)}
\alp{({x}b)--({e3}t)}
\alp{ ({e4}t)--({tc}b)}

\alp{({ta}b)--({t2}b)}
\end{scope}
\end{tikzpicture}
};

\node[draw, my style] (Z1) at (\scale*16, \scale*14.5){\begin{tikzpicture}

\oc{tc}{\rdual{\coo{A}}}{_c}{6}{1}{0}{\etr\lt}
\ocalt{x}{}{X}{4}{1}{0}{\en}

\oc{ta}{\coo{A}}{_a}{4}{-1}{-2}{\et}
\oc{t2}{\rdual{\coo{A}}}{_2}{9}{1}{-2}{\etr}

\oc{t5}{\coo{A}}{_5}{7}{1}{0}{\et\dk}
\oc{e6}{\evo{A}}{_6}{8}{0}{-1}{\ep\dk}

  \begin{scope}[on background layer]

\dlp{ ({x}t)--({tc}t)}
\dlp{({ta}t)--({e6}b)}
\dlp{({t5}t)--({t2}t)}
\alp{({t5}b)--({e6}t)}
\alp{({x}b)--({tc}b)}

\alp{({ta}b)--({t2}b)}
\end{scope}
\end{tikzpicture}
};

\node[draw, my style] (Z0) at (\scale*16, \scale*18.5){\begin{tikzpicture}

\oc{tc}{\rdual{\coo{A}}}{_c}{6}{1}{0}{\etr\lt}
\ocalt{x}{}{X}{4}{1}{0}{\en}

\oc{t2}{\rdual{\coo{A}}}{_2}{9}{1}{0}{\etr}

\oc{t5}{\coo{A}}{_5}{7}{1}{0}{\et\dk}

  \begin{scope}[on background layer]

\dlp{ ({x}t)--({tc}t)}
\dlp{({t5}t)--({t2}t)}
\alp{({t5}b)--({t2}b)}
\alp{({x}b)--({tc}b)}

\end{scope}
\end{tikzpicture}
};

\node[draw, my style] (A1) at (\scale*29.5,\scale*20){\begin{tikzpicture}

\ocalt{x}{}{X}{3}{1}{0}{\en}

\oc{t2}{\rdual{\coo{A}}}{_2}{5}{1}{0}{\etr}

  \begin{scope}[on background layer]

\dlp{ ({x}t)--({t2}t)}
\alp{({x}b)--({t2}b)}

\end{scope}
\end{tikzpicture}
};

\node[draw, my style] (A2) at (\scale*27,\scale* 14.5){\begin{tikzpicture}
\ocalt{x}{}{X}{5}{1}{0}{\en}

\oc{ta}{\coo{A}}{_a}{5}{-1}{-2}{\et}
\oc{t2}{\rdual{\coo{A}}}{_2}{9}{1}{-2}{\etr}

\oc{e6}{\evo{A}}{_6}{8}{0}{-1}{\ep\dk}

  \begin{scope}[on background layer]

\dlp{ ({x}t)--({t2}t)}
\dlp{({ta}t)--({e6}b)}
\alp{({x}b)--({e6}t)}

\alp{({ta}b)--({t2}b)}
\end{scope}
\end{tikzpicture}
};

\node[draw, my style] (Z2) at (\scale*16, \scale*10){\begin{tikzpicture}

\oc{tc}{\rdual{\coo{A}}}{_c}{5}{1}{0}{\etr\lt}
\ocalt{x}{}{X}{2}{1}{0}{\en}

\oc{ta}{\coo{A}}{_a}{2}{-1}{-2}{\et}
\oc{t2}{\rdual{\coo{A}}}{_2}{8}{1}{-2}{\etr}

\oc{e4}{\rdual{\evo{A}}}{_4}{4}{0}{-1}{\epr}
\oc{e3}{\evo{A}}{_3}{3}{0}{-1}{\ep}

\oc{t5}{\coo{A}}{_5}{6}{1}{0}{\et\dk}
\oc{e6}{\evo{A}}{_6}{7}{0}{-1}{\ep\dk}

  \begin{scope}[on background layer]

\dlp{ ({x}t)--({tc}t)}
\dlp{({ta}t)--({e3}b)}
\dlp{({e4}b)--({e6}b)}
\dlp{({t5}t)--({t2}t)}
\alp{({t5}b)--({e6}t)}
\alp{({x}b)--({e3}t)}
\alp{ ({e4}t)--({tc}b)}

\alp{({ta}b)--({t2}b)}
\end{scope}
\end{tikzpicture}
};

\node[draw, my style] (Z3) at (\scale*22, \scale*10){\begin{tikzpicture}

\ocalt{x}{}{X}{2}{1}{0}{\en}

\oc{ta}{\coo{A}}{_a}{2}{-1}{-2}{\et}
\oc{t2}{\rdual{\coo{A}}}{_2}{7}{1}{-2}{\etr}

\oc{e4}{\rdual{\evo{A}}}{_4}{4}{0}{-1}{\epr}
\oc{e3}{\evo{A}}{_3}{3}{0}{-1}{\ep}

\oc{e6}{\evo{A}}{_6}{6}{0}{-1}{\ep\dk}

  \begin{scope}[on background layer]

\dlp{ ({x}t)--({t2}t)}
\dlp{({ta}t)--({e3}b)}
\dlp{({e4}b)--({e6}b)}
\alp{({x}b)--({e3}t)}
\alp{ ({e4}t)--({e6}t)}

\alp{({ta}b)--({t2}b)}
\end{scope}
\end{tikzpicture}
};

\node[draw, my style] (Z4) at (\scale*27, \scale*10){\begin{tikzpicture}

\ocalt{x}{}{X}{2}{1}{0}{\en}

\oc{ta}{\coo{A}}{_a}{2}{-1}{-2}{\et}
\oc{t2}{\rdual{\coo{A}}}{_2}{5}{1}{-2}{\etr}
\oc{e3}{\evo{A}}{_3}{3}{0}{-1}{\ep}

  \begin{scope}[on background layer]

\dlp{ ({x}t)--({t2}t)}
\dlp{({ta}t)--({e3}b)}
\alp{({x}b)--({e3}t)}

\alp{({ta}b)--({t2}b)}
\end{scope}
\end{tikzpicture}
};

\node[draw, my style] (Y3) at (\scale*9, \scale*3.5){\begin{tikzpicture}

\oc{tc}{\rdual{\coo{A}}}{_c}{5}{1}{0}{\etr\lt}
\ocalt{x}{}{X}{3}{1}{-2}{\en}
\oc{t2}{\rdual{\coo{A}}}{_2}{5}{-1}{-2}{\etr}
\oc{e4}{\rdual{\evo{A}}}{_4}{4}{0}{-1}{\epr}

  \begin{scope}[on background layer]

\dlp{ ({x}t)--({tc}t)}
\dlp{({e4}b)--({t2}t)}
\alp{({x}b)--({t2}b)}
\alp{({e4}t)--({tc}b)}
\end{scope}
\end{tikzpicture}
};

\node[draw, my style] (Y4) at (\scale*16, \scale*5.5){\begin{tikzpicture}

\oc{tc}{\rdual{\coo{A}}}{_c}{5}{1}{0}{\etr\lt}
\ocalt{x}{}{X}{3}{1}{-2}{\en}
\oc{t2}{\rdual{\coo{A}}}{_2}{8}{1}{-2}{\etr}
\oc{e4}{\rdual{\evo{A}}}{_4}{4}{0}{-1}{\epr}

\oc{t5}{\coo{A}}{_5}{6}{1}{0}{\et\dk}
\oc{e6}{\evo{A}}{_6}{7}{0}{-1}{\ep\dk}
  \begin{scope}[on background layer]

\dlp{ ({x}t)--({tc}t)}
\dlp{({e4}b)--({e6}b)}
\dlp{({t5}t)--({t2}t)}
\alp{({t5}b)--({e6}t)}
\alp{({x}b)--({t2}b)}
\alp{({e4}t)--({tc}b)}
\end{scope}
\end{tikzpicture}
};

\node[draw, my style] (Y5) at (\scale*22, \scale*5.5){\begin{tikzpicture}

\ocalt{x}{}{X}{3}{1}{-2}{\en}
\oc{t2}{\rdual{\coo{A}}}{_2}{7}{1}{-2}{\etr}
\oc{e4}{\rdual{\evo{A}}}{_4}{4}{0}{-1}{\epr}

\oc{e6}{\evo{A}}{_6}{6}{0}{-1}{\ep\dk}
  \begin{scope}[on background layer]

\dlp{ ({x}t)--({t2}t)}
\dlp{({e4}b)--({e6}b)}
\alp{({x}b)--({t2}b)}
\alp{({e4}t)--({e6}t)}
\end{scope}
\end{tikzpicture}
};

\node[draw, my style] (Y6) at (\scale*29.5, \scale*3.5){\begin{tikzpicture}

\ocalt{x}{}{X}{3}{1}{-2}{\en}
\oc{t2}{\rdual{\coo{A}}}{_2}{8}{1}{-2}{\etr}
  \begin{scope}[on background layer]

\dlp{ ({x}t)--({t2}t)}
\alp{({x}b)--({t2}b)}
\end{scope}
\end{tikzpicture}
};


\draw[->](X1)--(Z0)node [midway , fill=white] {$I_{\coo{A}}$};
\draw[->](X1)--(Y1)node [midway , fill=white] {$I_{\coo{A}}$};
\draw[->](Y1)--(Y2)node [midway , fill=white] {$C_{\evo{A}}$};
\draw[->](Z1)--(Z2)node [midway , fill=white] {$C_{\evo{A}}$};
\draw[->](A2)--(Z3)node [midway , fill=white] {$C_{\evo{A}}$};

\draw[->](Y2)--(Y3)node [midway , fill=white] {$\trieto{3a}^{-1}$};
\draw[dashed, ->](X1)--(Y3)node [midway , fill=white] {$\trideto{24}$};
\draw[dashed, ->](Y3)--(Y6)node [midway , fill=white] {$\tridcto{c4}^{-1}$};

\draw[->](Z2)--(Y4)node [midway , fill=white] {$\trieto{3a}^{-1}$};
\draw[->](Z3)--(Y5)node [midway , fill=white] {$\trieto{3a}^{-1}$};

\draw[->](Y1)--(Z1)node [midway , fill=white] {$\tricto{56}$};
\draw[->](Y2)--(Z2)node [midway , fill=white] {$\tricto{56}$};
\draw[->](Y3)--(Y4)node [midway , fill=white] {$\tricto{56}$};

\draw[->](A1)--(A2)node [midway , fill=white] {$\tricto{a6}$};
\draw[->](Z0)--(Z1)node [midway , fill=white] {$\tricto{a6}$};
\draw[->](Y4)--(Y5)node [midway , fill=white] {$P_{\coo{A}}$};
\draw[->](Y5)--(Y6)node [midway , fill=white] {$L_{\evo{A}}$};
\draw[->](Z2)--(Z3)node [midway , fill=white] {$P_{\coo{A}}$};
\draw[->](Z1)--(A2)node [midway , fill=white] {$P_{\coo{A}}$};
\draw[->](Z0)--(A1)node [midway , fill=white] {$P_{\coo{A}}$};
\draw[->](Z3)--(Z4)node [midway , fill=white] {$L_{\evo{A}}$};
\draw[->](Z4)--(Y6)node [midway , fill=white] {$\trieto{3a}^{-1}$};

\draw[->](Y1)--(Z0)node [midway , fill=white] {$\gammaDo$};
\draw[->](A2)--(Z4)node [midway , fill=white] {$\id$};
\draw[->](X1)--(A1)node [midway , fill=white] {$\id$};
\draw[->](A1)--(Y6)node [midway , fill=white] {$\id$};

\node at  (barycentric cs:X1=1,Y1=1,Z0=1){Nat.};

\node at  (barycentric cs:Z0=2,Y1=1,Z1=1){Nat.}; 
\node at  (barycentric cs:Z0=1,X1=1,A1=1){\cref{fig:coeval_1_cell_triangles}};
\node at  (barycentric cs:Z0=1,Z1=1,A1=1,A2=1){Nat.};
\node at  (barycentric cs:Y1=1,Z1=1,Y2=1,Z2=1){Nat.};
\node at  (barycentric cs:A2=1,Z1=1,Z3=1,Z2=1){Nat.};
\node at  (barycentric cs:A2=2,Z3=1,Z4=1){\cref{fig:coeval_1_cell_2_triangles}};
\node at  (barycentric cs:A1=1,A2=1,Z4=1,Y6=1){Nat.};
\node at  (barycentric cs:Z3=1,Z4=1,Y5=1,Y6=1){Nat.};
\node at  (barycentric cs:Z3=1,Z2=1,Y5=1,Y4=1){Nat.};
\node at  (barycentric cs:Y2=1,Y3=1,Z2=1,Y4=1){Nat.};
\node at  (barycentric cs:X1=1,Y1=1,Y2=1,Y3=1){Def.};
\node at  (barycentric cs:Y4=1,Y5=1,Y6=1,Y3=1){Def.};

\end{tikzpicture}}
\caption{Verifying the  condition in \cref{fig:add_dualizable_0_cell_a}}
\end{subfigure}
\caption{1-dualizablity of $\zdual{A}$ (\cref{lem:compatiblity_1-witnesses_if_2-dual})}
\label{fig:dual_triangle_2}
\end{figure}

\cref{lem:compatiblity_1-witnesses_if_2-dual,lem:one_dual_composites} give a preferred pair of 1-cells recognizing the 1-dualizability of $\zdual{A}\otimes A$ if $A$ is 1-dualizable.  See \cref{fig:2dual_1_cell}.  Whenever $A$ is 2-dualizable we will use these 1-cells to define the shadow for a 1-cell in 
\[\sB(\zdual{A}\otimes A,\zdual{A}\otimes A).\]

\begin{figure}
\begin{subfigure}{.12\textwidth}
\resizebox{\textwidth}{!}{
\begin{tikzpicture}
	\oc{e3}{\rdual{\evo{A}}}{_3}{2}{2}{1}{\epr\lt}
	\oc{t4}{\coo{A}}{_4}{1}{3}{0}{\et\lt}
	\begin{scope}[on background layer]
	\dlp{({e3}b)--(4*\d-\w,\h)}
	\alp{({t4}b)--(4*\d-\w,0*\h)}
	
	\dlp{({t4}t)--(4*\d-\w,3*\h)}
	\alp{({e3}t)--(4*\d-\w,2*\h)}
	\end{scope}
	\end{tikzpicture}}
	\caption{}\label{fig:2dual_C_1_cell}
\end{subfigure}
\hspace{1cm}
\begin{subfigure}{.12\textwidth}
\resizebox{\textwidth}{!}{
\begin{tikzpicture}
	\oc{t5}{\rdual{\coo{A}}}{_5}{5}{3}{0}{\etr\lt}
	\oc{e6}{\evo{A}}{_6}{4}{2}{1}{\ep\lt}
	\begin{scope}[on background layer]
	\dlp{(2*\d+\w,3*\h)--({t5}t)}
	\alp{(2*\d+\w,2*\h)--({e6}t)}
	
	\dlp{(2*\d+\w,1*\h)--({e6}b)}
	\alp{(2*\d+\w,0*\h)--({t5}b)}
	\end{scope}
	\end{tikzpicture}}
	\caption{}\label{fig:2dual_E_1_cell}
\end{subfigure}
\caption{1-cells recognizing the 1-dualizability of $\zdual{A}\otimes A$}\label{fig:2dual_1_cell}
\end{figure}
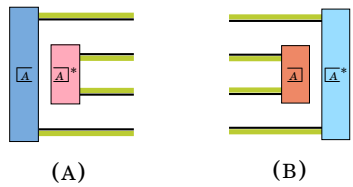

\begin{prop}\label{lem:ec_e_and_c}
If $A$ is 2-dualizable then the Euler characteristics of $\coo{A}$ and $\evo{A}$ are defined and  the triangles in \cref{fig:ec_e_and_c_statement} commute.
\end{prop}
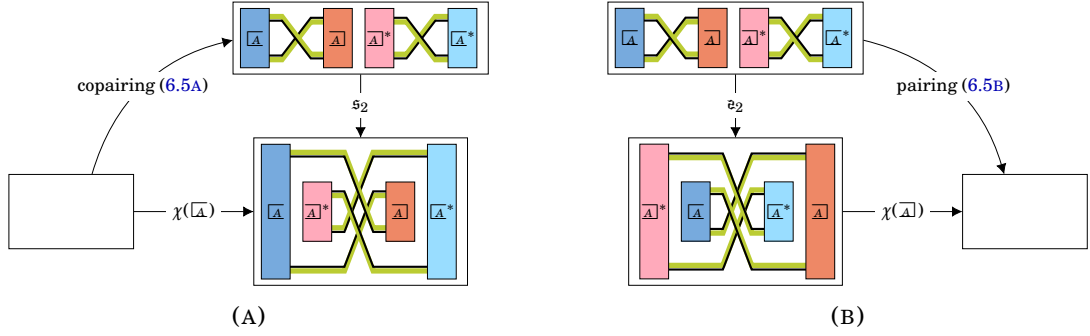
\begin{figure}[h!]
  \begin{subfigure}[b]{0.45\textwidth}
\resizebox{\textwidth}{!}
{
\begin{tikzpicture}[
    mystyle/.style={%
    },
   my style/.style={%
   },
  ]
 
\def\scale{1.15}
\node[draw,my style] (A0) at (0,  \scale*-4){\begin{tikzpicture}
		\node (ta) at (2,-1){};
		\node (ta) at (4,-2){};
		\end{tikzpicture}};

\node[draw,my style] (A1) at ( \scale*5, \scale* -4){\begin{tikzpicture}
	\oc{e3}{\rdual{\evo{A}}}{_3}{2}{2}{1}{\epr\lt}
	\oc{t4}{\coo{A}}{_4}{1}{3}{0}{\et\lt}
	\oc{t5}{\rdual{\coo{A}}}{_5}{5}{3}{0}{\etr\lt}
	\oc{e6}{\evo{A}}{_6}{4}{2}{1}{\ep\lt}
	\begin{scope}[on background layer]
	\dlp{({e3}b)--(3*\d-\w,\h)--(3*\d+\w,3*\h)--({t5}t)}
	\alp{({t4}b)--(3*\d-\w,0*\h)--(3*\d+\w,2*\h)--({e6}t)}	
	\dlp{({t4}t)--(3*\d-\w,3*\h)--(3*\d+\w,1*\h)--({e6}b)}
	\alp{({e3}t)--(3*\d-\w,2*\h)--(3*\d+\w,0*\h)--({t5}b)}
	\end{scope}
	\end{tikzpicture}
	};

\node[draw,my style] (B1) at ( \scale*5, \scale*-1){\begin{tikzpicture}
	\oc{e3}{\rdual{\evo{A}}}{_3}{1}{2}{1}{\epr\lt}
	\oc{t4}{\coo{A}}{_4}{-2}{2}{1}{\et\lt}
	
	\oc{t5}{\rdual{\coo{A}}}{_5}{3}{2}{1}{\etr\lt}
	\oc{e6}{\evo{A}}{_6}{0}{2}{1}{\ep\lt}
	\begin{scope}[on background layer]
	\dlp{({e3}b)--(2*\d-\w,\h)--(2*\d+\w,2*\h)--
	({t5}t)}
	\alp{({t4}b)--(-1*\d-\w,1*\h)--(-1*\d+\w,2*\h)--
	({e6}t)}
	\dlp{({t4}t)--(-1*\d-\w,2*\h)--(-1*\d+\w,1*\h)--
	({e6}b)}
	\alp{({e3}t)--(2*\d-\w,2*\h)--(2*\d+\w,1*\h)--
	({t5}b)}
	\end{scope}
	\end{tikzpicture}
	};
\draw[->](B1)--(A1)node [midway , fill=white] {$\gammat$};

\draw[->](A0)--(A1)node [midway , fill=white] {$\chi(\coo{A})$};

\draw[->](A0)edge[bend left=30]node [midway , fill=white] {copairing \eqref{fig:copairing_defn_fig}}(B1) ;

\end{tikzpicture}
}
\caption{}\label{fig:ec_e_and_c_a}
\end{subfigure}
\hfill
\begin{subfigure}[b]{0.45\textwidth}
\resizebox{\textwidth}{!}
{
\begin{tikzpicture}[
    mystyle/.style={%
    },
   my style/.style={%
   },
  ]
 
\def\scale{1.15}
\node[draw,my style] (A4) at ( \scale*15,  \scale*-4){\begin{tikzpicture}
		\node (ta) at (2,-1){};
		\node (ta) at (4,-2){};
		\end{tikzpicture}};

\node[draw,my style] (A3) at ( \scale*10,  \scale*-4){\begin{tikzpicture}
	\oc{e1}{\rdual{\evo{A}}}{_1}{-6}{1}{-2}{\epr}
	\oc{t2}{\coo{A}}{_2}{-5}{0}{-1}{\et}
	\oc{t7}{\rdual{\coo{A}}}{_7}{-3}{0}{-1}{\etr}
	\oc{e8}{\evo{A}}{_8}{-2}{1}{-2}{\ep}
	\begin{scope}[on background layer]
	\dlp{({e1}b)--(-4*\d-\w,-2*\h)--(-4*\d+\w,0*\h)--({t7}t)}
	\alp{({t2}b)--(-4*\d-\w,-1*\h)--(-4*\d+\w,1*\h)--({e8}t)}
	\dlp{({t2}t)--(-4*\d-\w,0*\h)--(-4*\d+\w,-2*\h)--({e8}b)}
	\alp{({e1}t)--(-4*\d-\w,1*\h)--(-4*\d+\w,-1*\h)--({t7}b)}
	\end{scope}
	\end{tikzpicture}
	};

\node[draw,my style] (B3) at ( \scale*10,  \scale*-1){\begin{tikzpicture}
	\oc{e1}{\rdual{\evo{A}}}{_1}{-4}{0}{-1}{\epr}
	\oc{t2}{\coo{A}}{_2}{-7}{0}{-1}{\et}
	\oc{t7}{\rdual{\coo{A}}}{_7}{-2}{0}{-1}{\etr}
	\oc{e8}{\evo{A}}{_8}{-5}{0}{-1}{\ep}
	\begin{scope}[on background layer]
	\dlp{({e1}b)--(-3*\d-\w,-1*\h)--(-3*\d+\w,0*\h)--({t7}t)}
	\alp{({t2}b)--(-6*\d-\w,-1*\h)--(-6*\d+\w,0*\h)--({e8}t)}
	\alp{({e1}t)--(-3*\d-\w,0*\h)--(-3*\d+\w,-1*\h)--({t7}b)}
	\dlp{({t2}t)--(-6*\d-\w,0*\h)--(-6*\d+\w,-1*\h)--({e8}b)}
	\end{scope}
	\end{tikzpicture}
	};
\draw[->](B3)--(A3)node [midway , fill=white] {$\gammattw$};

\draw[->](A3)--(A4)node [midway , fill=white] {$\chi(\evo{A})$};
\draw[->](B3) edge[bend left=30]node [midway , fill=white] {pairing \eqref{fig:pairing_defn_fig}} (A4);

\end{tikzpicture}
}
\caption{}\label{fig:ec_e_and_c_b}
\end{subfigure}
\caption{Compatibility between pairing/copairing and Euler characteristics (\cref{lem:ec_e_and_c})}\label{fig:ec_e_and_c_statement}
\end{figure}

\begin{proof}   Using the 2-cells in \cref{fig:2dual_1_cell}, 
 the Euler characteristic of $\coo{A}$ is the left composite in  \cref{fig:ec_e_and_c}.  This may feel a little less complicated than the usual shadow isomorphism since there are unit 0-cells.  The right composite is the copairing.  

Note that the left composite contains the map $P_{\coo{A}}$ and the right composite contains the map $C_{\evo{A}}$.  The primary goal of this diagram is to allow these maps to slide past each other.  To allow these maps to be applied simultaneously we introduce a new pair of $\coo{A}$ and $\evo{A}$ using $\trietop{}$ and rearrange the 1-cells to obtain the diagram in the node in the middle of the diagram with the thicker boarder.  To the right of this node is a commutative square where the orders of  $P_{\coo{A}}$ and $C_{\evo{A}}$ are exchanged.  This square is on a upper left to lower right diagonal of five naturality squares that pulls the map $C_{\evo{A}}$ across the diagram.  It is also on a lower left to upper right diagonal of four naturality that pulls $P_{\coo{A}}$ across the diagram.  These maps both vanish when they encounter the map $I_{\coo{A}}$.  The map $C_{\evo{A}}$ can be replaced by instances of $\trietop$ and $\tridctop$ when it encounters $I_{\coo{A}}$.  The maps $P_{\coo{A}}$ and  $I_{\coo{A}}$ are related by the triangle diagrams for $\coo{A}$ and so the composite is the identity map.

Aside from many more naturality squares, there are  five squares labeled (1)--(5) that commute by \cite[Thm.~1.25]{gurski_osorno}.
The top commuting triangle labeled with $=$ is included to simplify the formating of the triangle directly below it.  The vertical maps are the same.  The remaining three regions labeled by Figures and Lemma commute by the relevant result.

The proof for the pairing is Eckman-Hilton dual to that for the pairing and the relevant diagram is a 180 degree rotation.  Since it is pretty elaborate we include it in its entirety.  See \cref{fig:ec_e_and_c_2}.
\end{proof}
\begin{figure}[p]
\resizebox{\textwidth}{!}{
{\input{RR-june_copairing_euler_3_july}}
}
\caption{The expanded version of \cref{fig:ec_e_and_c_a}. The right composite is the copairing and the left composite is the Euler characteristic of $C$ (\cref{lem:ec_e_and_c})}\label{fig:ec_e_and_c}
\end{figure}

\begin{figure}[p]
\resizebox{\textwidth}{!}{
\input{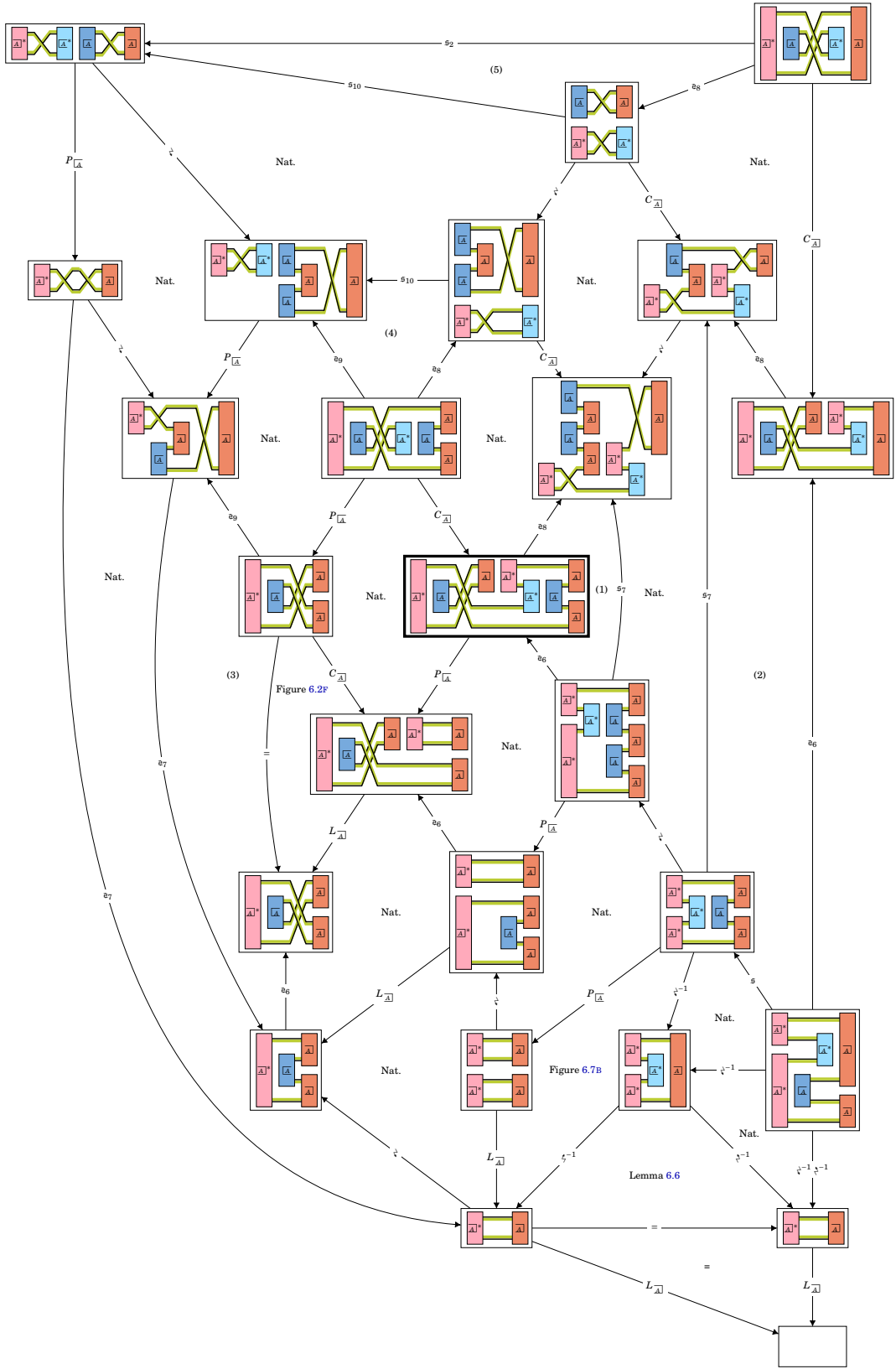}
}
\caption{The expanded version of \cref{fig:ec_e_and_c_b}.  The left map is the pairing and the right map is the Euler characteristic of $E$ (\cref{lem:ec_e_and_c})}\label{fig:ec_e_and_c_2}
\end{figure}

For 2-cells $f\colon M\to M$ and $g\colon N\to N$ 
consider  the endomorphism of 
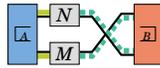
\begin{center}
\resizebox{.15\textwidth}{!}{
\begin{tikzpicture}
\oc{t2}{\coo{A}}{_2}{-8}{0}{-1}{\et}
\oc{m}{M}{}{-7}{-1}{-1}{\en}
\oc{n}{N}{}{-7}{0}{0}{\en}
\oc{e8}{\evo{B}}{_8}{-5}{0}{-1}{\ep}

\begin{scope}[on background layer]

\alp{({t2}b)--({m}t)}
\blp{({m}t)--(-6*\d-\w,-1*\h)--(-6*\d+\w,0*\h)--({e8}t)}

\dlp{({t2}t)--({n}t)}
\dblp{({n}t)--(-6*\d-\w,0*\h)--(-6*\d+\w,-1*\h)--({e8}b)}

\end{scope}
\end{tikzpicture}}
\end{center}
induced by $f$ and $g$ (and the identity maps of $\coo{A}$, $\evo{B}$, and $\Gamma$).   This endomorphism will be written 
\[\id_{\coo{A}}\odot (f\otimes g)\odot \id_{\Gamma}\odot \id_{\evo{B}}\]
The following is the generalization of \cref{intro:shklyarov}.  

\begin{thm}
\label{thm:main_pairing} Suppose $A$ and $B$ are 2-dualizable 0-cells and $M\in\sB(A,B)$ and $N\in \sB(\zdual{A},\zdual{B})$ are dualizable 1-cells.  If $f\colon M\to M$ and $g\colon N\to N$ are 2-cells the  diagram \ref{fig:main_pairing} 
commutes. 
\end{thm}

\begin{equation}\label{fig:main_pairing}
\resizebox{.7\textwidth}{!}{\begin{tikzpicture}[
    mystyle/.style={%
    },
   my style/.style={%
   },
  ]

\node[draw,my style] (A0) at (5, -4){\begin{tikzpicture}
		\node (ta) at (2,-1){};
		\node (ta) at (4,-2){};
		\end{tikzpicture}};

\node[draw,my style] (A4) at (13, -4){\begin{tikzpicture}
		\node (ta) at (2,-1){};
		\node (ta) at (4,-2){};
		\end{tikzpicture}};

\node[draw,my style] (B1) at (5,-1){\begin{tikzpicture}
	\oc{e3}{\rdual{\evo{A}}}{_3}{1}{2}{1}{\epr\lt}
	\oc{t4}{\coo{A}}{_4}{-2}{2}{1}{\et\lt}

	\oc{t5}{\rdual{\coo{A}}}{_5}{3}{2}{1}{\etr\lt}
	\oc{e6}{\evo{A}}{_6}{0}{2}{1}{\ep\lt}
	\begin{scope}[on background layer]

	\dlp{({e3}b)--(2*\d-\w,\h)--(2*\d+\w,2*\h)--
	({t5}t)}
	\alp{({t4}b)--(-1*\d-\w,1*\h)--(-1*\d+\w,2*\h)--
	({e6}t)}
	
	\dlp{({t4}t)--(-1*\d-\w,2*\h)--(-1*\d+\w,1*\h)--
	({e6}b)}
	\alp{({e3}t)--(2*\d-\w,2*\h)--(2*\d+\w,1*\h)--
	({t5}b)}
	\end{scope}
	\end{tikzpicture}
	};

\node[draw,my style] (B3) at (13, -1){\begin{tikzpicture}
	\oc{e1}{\rdual{\evo{B}}}{_1}{-4}{0}{-1}{\epr}
	\oc{t2}{\coo{B}}{_2}{-7}{0}{-1}{\et}
	
	\oc{t7}{\rdual{\coo{B}}}{_7}{-2}{0}{-1}{\etr}
	\oc{e8}{\evo{B}}{_8}{-5}{0}{-1}{\ep}

	\begin{scope}[on background layer]
	
	\dblp{({e1}b)--(-3*\d-\w,-1*\h)--(-3*\d+\w,0*\h)--({t7}t)}
	\blp{({e1}t)--(-3*\d-\w,0*\h)--(-3*\d+\w,-1*\h)--({t7}b)}
	\blp{({t2}b)--(-6*\d-\w,-1*\h)--(-6*\d+\w,0*\h)--({e8}t)}
	
	\dblp{({t2}t)--(-6*\d-\w,0*\h)--(-6*\d+\w,-1*\h)--({e8}b)}
	
	\end{scope}
	\end{tikzpicture}
	};

\draw[->](B1)--(B3)node [midway , fill=white] {$\tr
(f)\otimes \tr
(g)$};
\draw[->](A0)-- node [midway , fill=white] {copairing}(B1) ;
\draw[->](B3) --node [midway , fill=white] {pairing} (A4);
\draw[->](A0)--node [midway , fill=white] {$\tr\left(\id_{\coo{A}}\odot (f\otimes g)\odot \id_\Gamma\odot \id_{\evo{B}}
	\right)$} (A4);

\end{tikzpicture}}
\end{equation}

\begin{figure}
\resizebox{\textwidth}{!}
{
\resizebox{.95\textwidth}{!}
{
\begin{tikzpicture}[
    mystyle/.style={%
    },
   my style/.style={%
   },
  ]

\node[draw,my style] (A0) at (-1, -7.5){\begin{tikzpicture}
		\node (ta) at (2,-1){};
		\node (ta) at (4,-2){};
		\end{tikzpicture}};

\node[draw,my style] (A4) at (20, -7.5){\begin{tikzpicture}
		\node (ta) at (2,-1){};
		\node (ta) at (4,-2){};
		\end{tikzpicture}};

\node[draw,my style] (A1) at (2, -4){\begin{tikzpicture}
\oc{e3}{\rdual{\evo{A}}}{_3}{2}{2}{1}{\epr\lt}
\oc{t4}{\coo{A}}{_4}{1}{3}{0}{\et\lt}

\oc{t5}{\rdual{\coo{A}}}{_5}{5}{3}{0}{\etr\lt}
\oc{e6}{\evo{A}}{_6}{4}{2}{1}{\ep\lt}
\begin{scope}[on background layer]

\dlp{({e3}b)--(3*\d-\w,\h)--(3*\d+\w,3*\h)--({t5}t)}
\alp{({t4}b)--(3*\d-\w,0*\h)--(3*\d+\w,2*\h)--({e6}t)}

\dlp{({t4}t)--(3*\d-\w,3*\h)--(3*\d+\w,1*\h)--({e6}b)}
\alp{({e3}t)--(3*\d-\w,2*\h)--(3*\d+\w,0*\h)--({t5}b)}
\end{scope}
\end{tikzpicture}
};

\node[draw,my style] (A3) at (8.5, -4){\begin{tikzpicture}
\oc{e1}{\rdual{\evo{B}}}{_1}{-5}{0}{-1}{\epr}
\oc{t2}{\coo{B}}{_2}{-6}{1}{-2}{\et}

\oc{t7}{\rdual{\coo{B}}}{_7}{-2}{1}{-2}{\etr}
\oc{e8}{\evo{B}}{_8}{-3}{0}{-1}{\ep}

\begin{scope}[on background layer]

\dblp{({e1}b)--(-4*\d-\w,-1*\h)--(-4*\d+\w,1*\h)--({t7}t)}
\blp{({t2}b)--(-4*\d-\w,-2*\h)--(-4*\d+\w,0*\h)--({e8}t)}

\blp{({e1}t)--(-4*\d-\w,0*\h)--(-4*\d+\w,-2*\h)--({t7}b)}
\dblp{({t2}t)--(-4*\d-\w,1*\h)--(-4*\d+\w,-1*\h)--({e8}b)}

\end{scope}
\end{tikzpicture}
};

\node[draw,my style] (A3a) at (17, -4){\begin{tikzpicture}
\oc{e1}{\rdual{\evo{B}}}{_1}{-6}{1}{-2}{\epr}
\oc{t2}{\coo{B}}{_2}{-5}{0}{-1}{\et}

\oc{t7}{\rdual{\coo{B}}}{_7}{-3}{0}{-1}{\etr}
\oc{e8}{\evo{B}}{_8}{-2}{1}{-2}{\ep}

\begin{scope}[on background layer]

\dblp{({e1}b)--(-4*\d-\w,-2*\h)--(-4*\d+\w,0*\h)--({t7}t)}
\blp{({t2}b)--(-4*\d-\w,-1*\h)--(-4*\d+\w,1*\h)--({e8}t)}

\blp{({e1}t)--(-4*\d-\w,1*\h)--(-4*\d+\w,-1*\h)--({t7}b)}
\dblp{({t2}t)--(-4*\d-\w,0*\h)--(-4*\d+\w,-2*\h)--({e8}b)}

\end{scope}
\end{tikzpicture}
};

\node[draw,my style] (B1) at (-1,-.5){\begin{tikzpicture}
\oc{e3}{\rdual{\evo{A}}}{_3}{1}{2}{1}{\epr\lt}
\oc{t4}{\coo{A}}{_4}{-2}{2}{1}{\et\lt}

\oc{t5}{\rdual{\coo{A}}}{_5}{3}{2}{1}{\etr\lt}
\oc{e6}{\evo{A}}{_6}{0}{2}{1}{\ep\lt}
\begin{scope}[on background layer]

\dlp{({e3}b)--(2*\d-\w,\h)--(2*\d+\w,2*\h)--
({t5}t)}
\alp{({t4}b)--(-1*\d-\w,1*\h)--(-1*\d+\w,2*\h)--
({e6}t)}

\dlp{({t4}t)--(-1*\d-\w,2*\h)--(-1*\d+\w,1*\h)--
({e6}b)}
\alp{({e3}t)--(2*\d-\w,2*\h)--(2*\d+\w,1*\h)--
({t5}b)}
\end{scope}
\end{tikzpicture}
};

\node[draw,my style] (B3) at (20, -.5){\begin{tikzpicture}
\oc{e1}{\rdual{\evo{B}}}{_1}{-4}{0}{-1}{\epr}
\oc{t2}{\coo{B}}{_2}{-7}{0}{-1}{\et}

\oc{t7}{\rdual{\coo{B}}}{_7}{-2}{0}{-1}{\etr}
\oc{e8}{\evo{B}}{_8}{-5}{0}{-1}{\ep}

\begin{scope}[on background layer]

\dblp{({e1}b)--(-3*\d-\w,-1*\h)--(-3*\d+\w,0*\h)--({t7}t)}
\blp{({e1}t)--(-3*\d-\w,0*\h)--(-3*\d+\w,-1*\h)--({t7}b)}
\blp{({t2}b)--(-6*\d-\w,-1*\h)--(-6*\d+\w,0*\h)--({e8}t)}

\dblp{({t2}t)--(-6*\d-\w,0*\h)--(-6*\d+\w,-1*\h)--({e8}b)}

\end{scope}
\end{tikzpicture}
};
\draw[->](B1)--(A1)node [midway , fill=white] {$\gammat$};
\draw[->](B3)edge [bend right=10]node [midway , fill=white] {$\gammat$}(A3);
\draw[->](B3)--(A3a)node [midway , fill=white] {$\gammattw$};

\draw[->](B1)--(B3)node [midway , fill=white] {$\tr
(f)\otimes \tr
(g)$};
\draw[->](A1)--(A3)node [midway , fill=white] {$\tr(f\otimes g)$};
\draw[->](A0)--(A1)node [midway , fill=white] {$\chi(\coo{A})$};
\draw[->](A3a)--(A4)node [midway , fill=white] {$\chi(\evo{B})$};
\draw[->](A3)edge[bend left =20]node [midway , fill=white] {$\gammath$}(A3a);
\draw[->](A3)edge[bend right =20]node [midway , fill=white] {$\chi(\Gamma)$}(A3a);

\draw[->](A0)edge[bend left=50]node [midway , fill=white] {copairing}(B1) ;
\draw[->](B3) edge[bend left=50]node [midway , fill=white] {pairing} (A4);

\draw[->](A0) edge node [midway , fill=white] {$\tr\left(\id_{\coo{A}}\odot (f\otimes g)\odot \id_\Gamma\odot \id_{\evo{B}}
	\right)$} (A4);

\node at  (barycentric cs:A1=.8,B1=1,B3=.5,A3=.4){\cref{lem:trace_otimes_compatible}};

\node at  (barycentric cs:A1=-.2,B1=1,A0=.7){\cref{fig:ec_e_and_c_a}};

\node at  (barycentric cs:A4=.7,B3=1,A3a=-.2){\cref{fig:ec_e_and_c_b}};

\node at  (barycentric cs:A3=.5,B3=1,A3a=.3){(1)};
\node at  (barycentric cs:A3=1,A3a=1){\cref{lem:euler_char_gamma}};

\node at  (barycentric cs:A3=1,A0=1,A4=1){\cref{thm:composite}};
thm:composite
\end{tikzpicture}
}
}
\caption{Expansion of \cref{fig:main_pairing} to verify the compatibility between trace, pairing and copairing (\cref{thm:main_pairing})}\label{fig:main_pairing_2}
\end{figure}
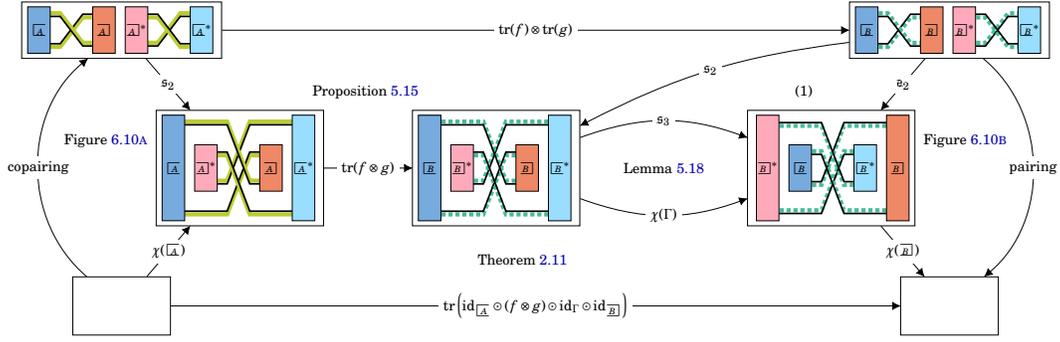

\begin{proof}
This theorem is the culmination of many of the preceding results.  The first relevant result is \cref{thm:composite} which we use to compute the trace of 
\[\id_{\coo{A}}\odot (f\otimes g)\odot \id_{\Gamma}\odot \id_{\evo{B}}\] in terms of the trace of each piece.  This is the bottom region in the 
 expansion of  the diagram in \eqref{fig:main_pairing} to the diagram in \cref{fig:main_pairing_2}.  The next result we use is \cref{lem:trace_otimes_compatible} which relates the trace of $f\otimes g$ and the trace of $f$ and $g$.  This is the top left region in \cref{fig:main_pairing_2}.  This is closely related to the oval region of \cref{fig:main_pairing_2} that recognizes the Euler characteristic of $\Gamma$ as a symmetry map (\cref{lem:euler_char_gamma}).  Note that the region here differs from the diagram in \cref{lem:euler_char_gamma} since $\coo{\zdual{A}}$ and $\evo{\zdual{A}}$ are replaced by $\rdual{\evo{A}}$ and $\rdual{\coo{A}}$. 
  This is permissible since  \cref{lem:compatiblity_1-witnesses_if_2-dual} shows $\rdual{\evo{A}}$ is a choice of $\coo{\zdual{A}}$.

The left and right regions are the conclusions of \cref{lem:ec_e_and_c} that verifies the compatibility between the Euler characteristic of $\coo{A}$ and the copairing and the Euler characteristic of $\evo{A}$ and the pairing.
The remaining region, labeled with (1), commutes by \cite[Thm.~1.25]{gurski_osorno}.
\end{proof}

\section{Serre duality and Shklyarov's theorem}\label{sec:serre_duality}

It should now be clear that there there are a considerable number of categorical operations in symmetric monoidal bicategories given by geometric  operations. Thinking in a ``cobordism hypothesis'' way, this makes perfect sense. One more geometric move we can make in the presence of 1-dualizability is ``putting a kink'' in a bimodule, or turning a $(A, B)$-bimodule into a $(B^\vee, A^\vee)$-bimodule.  More precisely, suppose $A$ and $B$ are 1-dualizable 0-cells. 
For a 1-cell $M\in \sB(A,B)$, the Serre dual of $M$, denoted $\serredual{M}{A}{B}\in \sB(\zdual{B}, \zdual{A})$, is the composite in \cref{fig:serre_dual_one_cell}.  For a 2-cell $f\colon M\to N$, let $\serredual{f}{A}{B}$ be the similar composite using the identity maps of ${\coo{A}}$ and $\evo{B}$. 

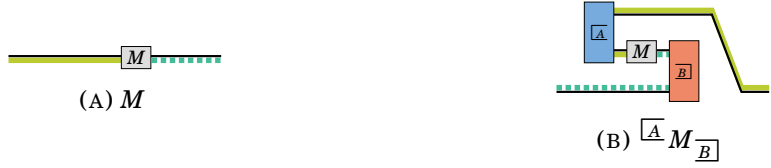
\begin{figure}[h!]
\hspace{2cm}
\begin{subfigure}{.2\textwidth}
\resizebox{\textwidth}{!}{
\begin{tikzpicture}
\oc{m}{M}{}{-8}{0}{0}{\en}
\begin{scope}[on background layer]
\alp{(-11*\d,0)--({m}t)}
\blp{({m}t)--(-6*\d,0*\h)}
\end{scope}
\end{tikzpicture}}
\caption{$M$}\label{fig:serre_dual_one_cell_input}
\end{subfigure}
\hfill
\begin{subfigure}{.2\textwidth}

\resizebox{\textwidth}{!}{
\begin{tikzpicture}
\oc{t8}{\evo{B}}{_8}{-8}{1}{0}{\ep}
\oc{m}{M}{}{-9}{1}{1}{\en}

\oc{e3}{{\coo{A}}}{_3}{-10}{2}{1}{\et}
\gc{g1}{-9}{2}{1}
\gc{g2}{-8.5}{1}{0}
\gc{g3}{-8}{0}{-1}

\begin{scope}[on background layer]
\dlp{({e3}t)--(-7*\d-\w,2*\h)--(-7*\d+\w,0*\h)--(-6*\d,0*\h)}
\alp{({e3}b)--({m}t)}
\blp{({m}t)--
({t8}t)}

\dblp{(-11*\d,0)--
({t8}b)}
\end{scope}
\end{tikzpicture}}
\caption{$\serredual{M}{A}{B}$}\label{fig:serre_dual_one_cell}
\end{subfigure}
\hspace{2cm}
\caption{Serre dual}\label{fig:define_serre_dual}
\end{figure}

This Serre dual gives us a way of realizing one of our original goals: 

\begin{lem}[\cref{goal:dualizable_functor}]\label{lem:serre_dual_dualizable}
If $A$ and $B$ are 2-dualizable and $M$ is right dualizable then 
$\serredual{M}{A}{B}$ is right dualizable.
\end{lem}

\begin{proof}
By assumption $M$ is right dualizable.  The 1-cells  $\coo{A}$ and $\coo{B}$ are right dualizable since $A$ and $B$ are 2-dualizable.  Then \cref{lem:duals_compose} implies $\serredual{M}{A}{B}$ is right dualizable.
\end{proof}

The following result is only used in the motivating result in \cref{sec:motivation:Shklyarov}, but it is an expected question following the lemma above so we address it here.

\begin{cor}\label{lem:serre_dual_trace}
If $M\in \sB(A,B)$ is right dualizable, 
the trace of the Serre dual of $f\colon M\to M$ is 
the  composite 
\[\sh{U_{\zdual{B}}} \xto{\id\otimes \mathrm{copairing}} \sh{U_{\zdual{B}}}\otimes \sh{U_A}\otimes \sh{U_{\zdual{A}}}  \xto{\id\otimes \tr(f)\otimes \id} \sh{U_{\zdual{B}}}\otimes \sh{U_B}\otimes \sh{U_{\zdual{A}}} 
\xto{\mathrm{pairing}\otimes \id}  \sh{U_{\zdual{A}}} 
\]
\end{cor}

\begin{figure}
\begin{subfigure}{.58\textwidth}
\resizebox{\textwidth}{!}
{
\begin{tikzpicture}[
    mystyle/.style={%
    },
   my style/.style={%
   },
  ]

\node[draw,my style] (A3) at (0, 8){
\begin{tikzpicture}
\oc{e2}{\coo{B}}{_2}{-10}{0}{-1}{\et}

\oc{m}{M}{}{-9}{-1}{-1}{\en}
\oc{e3}{\evo{A}}{_3}{-8}{-1}{-2}{\ep}

\begin{scope}[on background layer]

\alp{({e2}b)--({m}b)}
\blp{({m}b)--({e3}t)}

\dlp{({e2}t)--(-7*\d,0*\h)}
\dblp{(-11*\d,-2*\h)--({e3}b)}

\end{scope}
\end{tikzpicture}
};

\node[draw,my style] (A4) at (6, 8){
\begin{tikzpicture}
\oc{e2}{\coo{B}}{_2}{-10}{0}{-1}{\et}

\oc{m}{M}{}{-9}{-1}{-1}{\en}
\oc{e3}{\evo{A}}{_3}{-8}{-1}{-2}{\ep}

\begin{scope}[on background layer]

\alp{({e2}b)--({m}b)}
\blp{({m}b)--({e3}t)}

\dlp{({e2}t)--(-7*\d,0*\h)}
\dblp{(-11*\d,-2*\h)--({e3}b)}

\end{scope}
\end{tikzpicture}
};
\draw[->](A3)--(A4);
\end{tikzpicture}

}
\caption{}\label{fig:serre_dual_fix_notation}
\end{subfigure}
\hfill
\begin{subfigure}{.22\textwidth}
\resizebox{\textwidth}{!}
{
\begin{tikzpicture}

\node[draw] (A3) at (0, 8){
\begin{tikzpicture}
\oc{e2}{\coo{B}}{_2}{-10}{0}{-1}{\et}

\begin{scope}[on background layer]

\alp{({e2}b)--(-7*\d,-1*\h)}
\dlp{({e2}t)--(-7*\d,0*\h)}
\dblp{(-11*\d,-2*\h)--(-7*\d,-2*\h)}

\end{scope}
\end{tikzpicture}
};

\end{tikzpicture}

},
\caption{}\label{fig:serre_dual_map_1}
\end{subfigure}

\begin{subfigure}{.58\textwidth}
\resizebox{\textwidth}{!}
{
\begin{tikzpicture}[
    mystyle/.style={%
      5label={right:\pgfkeysvalueof{/pgf/minimum width}},
    },
   my style/.style={%
   },
  ]

\node[draw,my style] (A3) at (0, 8){
\begin{tikzpicture}
\oc{m}{M}{}{-9}{-1}{-1}{\en}

\begin{scope}[on background layer]

\alp{(-11*\d,-1*\h)--({m}b)}
\blp{({m}b)--(-7*\d,-1*\h)}

\dlp{(-11*\d,0*\h)--(-7*\d,0*\h)}
\dblp{(-11*\d,-2*\h)--(-7*\d,-2*\h)}

\end{scope}
\end{tikzpicture}
};

\node[draw,my style] (A4) at (6, 8){
\begin{tikzpicture}
\oc{m}{M}{}{-9}{-1}{-1}{\en}

\begin{scope}[on background layer]

\alp{(-11*\d,-1*\h)--({m}b)}
\blp{({m}b)--(-7*\d,-1*\h)}

\dlp{(-11*\d,0*\h)--(-7*\d,0*\h)}
\dblp{(-11*\d,-2*\h)--(-7*\d,-2*\h)}

\end{scope}
\end{tikzpicture}
};
\draw[->](A3)--(A4)node [midway , fill=white] {$\id\otimes f\otimes \id$};
\end{tikzpicture}

}
\caption{}\label{fig:serre_dual_map_2}
\end{subfigure}
\hfill
\begin{subfigure}{.22\textwidth}
\resizebox{\textwidth}{!}
{
\begin{tikzpicture}
  ]

\node[draw] (A3) at (0, 8){
\begin{tikzpicture}
\oc{e3}{\evo{A}}{_3}{-8}{-1}{-2}{\ep}

\begin{scope}[on background layer]

\blp{(-11*\d,-1*\h)--({e3}t)}

\dlp{(-11*\d,0*\h)--(-7*\d,0*\h)}
\dblp{(-11*\d,-2*\h)--({e3}b)}

\end{scope}
\end{tikzpicture}
};

\end{tikzpicture}

}
\caption{}\label{fig:serre_dual_map_3}
\end{subfigure}
\caption{Serre dual of a map $f\colon M\to M$}
\end{figure}

\begin{proof}
The Serre dual of $f$ is a map of the  form in \cref{fig:serre_dual_fix_notation}.  Using \cref{thm:composite}, its trace is the composite of three traces.  The first is the trace of identity map of the 
1-cell in \cref{fig:serre_dual_map_1}.  The second is the trace of the  map $f\colon M\to M$ and identity maps as in \cref{fig:serre_dual_map_2}.  The third is the trace of the identity map of the 1-cell in  \cref{fig:serre_dual_map_3}.  This identifies the trace of the Serre dual of $f$ with the top composite in  \ref{fig:serre_dual_maps_big}.  

\cref{lem:trace_otimes_compatible} allows us to expand each of these traces into a monoidal product of two traces.  The trace of identity map of the 
1-cell in \cref{fig:serre_dual_map_1} becomes the traces of the identity maps of each of the 1-cells in the digram in   \cref{fig:serre_dual_map_1}.  The same simplification holds for the 1-cells in  \cref{fig:serre_dual_map_3}.  For the trace of the map in \cref{fig:serre_dual_map_2} we divide this into two traces, one for the identity map and $f$ and one for an identity map of the unit 1-cell.  Note that the trace of the identity map of the unit 1-cell is the identity.  We repeat this simplification for the bottom center region in \ref{fig:serre_dual_maps_big} to obtain the bottom arrow of the diagram.  Both regions labeled by \cref{lem:ec_e_and_c} are immediate consequences of that result and the observation that the Euler characteristic of a unit 1-cell is an identity map.

Finally, the region labeled (1) is another application of \cite[Thm.~1.25]{gurski_osorno}.
\end{proof}
\begin{figure}
\resizebox{\textwidth}{!}
{
\begin{tikzpicture}[
    mystyle/.style={%
    },
   my style/.style={%
   },
  ]

\node[draw,my style] (C0) at (0, 0){
\begin{tikzpicture}

\oc{e5}{\rdual{\coo{{A}}}}{_5}{-3}{-1}{-2}{\etr\lt}

\oc{e6}{\rdual{\evo{{A}}}}{_6}{-5}{-1}{-2}{\epr\lt}
\gc{g1}{-4}{-1}{-2}

\begin{scope}[on background layer]

\dblp{({e6}b)\gpb{g1}({e5}t)}
\blp{({e6}t)\gpt{g1}({e5}b)}

\end{scope}
\end{tikzpicture}
};

\node[draw,my style] (C4) at (32, 0){
\begin{tikzpicture}

\oc{e2}{\rdual{\coo{{A}}}}{_2}{-3}{-1}{-2}{\etr\lt}

\oc{e4}{\rdual{\evo{{A}}}}{_4}{-5}{-1}{-2}{\epr\lt}
\gc{g1}{-4}{-1}{-2}

\begin{scope}[on background layer]

\dblp{({e4}b)\gpb{g1}({e2}t)}
\blp{({e4}t)\gpt{g1}({e2}b)}

\end{scope}
\end{tikzpicture}
};

\node[draw,my style] (C1) at (7, 0){
\begin{tikzpicture}

\oc{e1}{\coo{{B}}}{_1}{-3}{1}{-2}{\et\lt}

\oc{e2}{\rdual{\coo{{B}}}}{_2}{2}{2}{-3}{\etr\lt}

\oc{e3}{\evo{{B}}}{_3}{1}{1}{-2}{\ep\lt}

\oc{e4}{\rdual{\evo{{B}}}}{_4}{-2}{0}{-1}{\epr\lt}

\oc{e5}{\rdual{\coo{{A}}}}{_5}{0}{0}{-1}{\etr\lt}

\oc{e6}{\rdual{\evo{{A}}}}{_6}{-4}{2}{-3}{\epr\lt}
\gc{g1}{-4}{-1}{-2}

\begin{scope}[on background layer]

\alp{({e1}b)--(-1*\d-\w,-2*\h)--(-1*\d+\w,1*\h)--({e3}t)}
\dlp{({e4}b)--(-1*\d-\w,-1*\h)--(-1*\d+\w,2*\h)--({e2}t)}

\dblp{({e6}b)--(-1*\d-\w,-3*\h)--(-1*\d+\w,0*\h)--({e5}t)}
\blp{({e6}t)--(-1*\d-\w,2*\h)--(-1*\d+\w,-1*\h)--({e5}b)}
\alp{({e4}t)--(-1*\d-\w,0*\h)--(-1*\d+\w,-3*\h)--({e2}b)}
\dlp{({e1}t)--(-1*\d-\w,1*\h)--(-1*\d+\w,-2*\h)--({e3}b)}

\end{scope}
\end{tikzpicture}
};

\node[draw,my style] (C2) at (25, 0){
\begin{tikzpicture}

\oc{e1}{\coo{{A}}}{_1}{-3}{1}{-2}{\et\lt}

\oc{e2}{\rdual{\coo{{B}}}}{_2}{2}{2}{-3}{\etr\lt}

\oc{e3}{\evo{{A}}}{_3}{1}{1}{-2}{\ep\lt}

\oc{e4}{\rdual{\evo{{B}}}}{_4}{-2}{0}{-1}{\epr\lt}

\oc{e5}{\rdual{\coo{{A}}}}{_5}{0}{0}{-1}{\etr\lt}

\oc{e6}{\rdual{\evo{{A}}}}{_6}{-4}{2}{-3}{\epr\lt}
\gc{g1}{-4}{-1}{-2}

\begin{scope}[on background layer]

\blp{({e1}b)--(-1*\d-\w,-2*\h)--(-1*\d+\w,1*\h)--({e3}t)}
\dlp{({e4}b)--(-1*\d-\w,-1*\h)--(-1*\d+\w,2*\h)--({e2}t)}

\dblp{({e6}b)--(-1*\d-\w,-3*\h)--(-1*\d+\w,0*\h)--({e5}t)}
\blp{({e6}t)--(-1*\d-\w,2*\h)--(-1*\d+\w,-1*\h)--({e5}b)}
\alp{({e4}t)--(-1*\d-\w,0*\h)--(-1*\d+\w,-3*\h)--({e2}b)}
\dblp{({e1}t)--(-1*\d-\w,1*\h)--(-1*\d+\w,-2*\h)--({e3}b)}

\end{scope}
\end{tikzpicture}
};

\node[draw,my style] (D1a) at (15, -6){
\begin{tikzpicture}

\oc{e1}{\coo{{B}}}{_1}{-3}{0}{-1}{\et\lt}

\oc{e2}{\rdual{\coo{{B}}}}{_2}{3}{0}{-1}{\etr\lt}

\oc{e3}{\evo{{B}}}{_3}{0}{1}{-2}{\ep\lt}

\oc{e4}{\rdual{\evo{{B}}}}{_4}{1}{0}{-1}{\epr\lt}

\oc{e5}{\rdual{\coo{{A}}}}{_5}{-1}{0}{-1}{\etr\lt}

\oc{e6}{\rdual{\evo{{A}}}}{_6}{-4}{1}{-2}{\epr\lt}
\gc{g1}{-4}{-1}{-2}

\begin{scope}[on background layer]

\alp{({e1}b)--(-2*\d-\w,-1*\h)--(-2*\d+\w,1*\h)--({e3}t)}
\dlp{({e4}b)--(2*\d-\w,-1*\h)--(2*\d+\w,0*\h)--({e2}t)}

\dblp{({e6}b)--(-2*\d-\w,-2*\h)--(-2*\d+\w,0*\h)--({e5}t)}
\blp{({e6}t)--(-2*\d-\w,1*\h)--(-2*\d+\w,-1*\h)--({e5}b)}
\alp{({e4}t)--(2*\d-\w,0*\h)--(2*\d+\w,-1*\h)--({e2}b)}
\dlp{({e1}t)--(-2*\d-\w,0*\h)--(-2*\d+\w,-2*\h)--({e3}b)}

\end{scope}
\end{tikzpicture}
};
\node[draw,my style] (D2) at (25, -6){
\begin{tikzpicture}

\oc{e1}{\coo{{A}}}{_1}{-3}{0}{-1}{\et\lt}

\oc{e2}{\rdual{\coo{{B}}}}{_2}{3}{0}{-1}{\etr\lt}

\oc{e3}{\evo{{A}}}{_3}{0}{1}{-2}{\ep\lt}

\oc{e4}{\rdual{\evo{{B}}}}{_4}{1}{0}{-1}{\epr\lt}

\oc{e5}{\rdual{\coo{{A}}}}{_5}{-1}{0}{-1}{\etr\lt}

\oc{e6}{\rdual{\evo{{A}}}}{_6}{-4}{1}{-2}{\epr\lt}
\gc{g1}{-4}{-1}{-2}

\begin{scope}[on background layer]

\blp{({e1}b)--(-2*\d-\w,-1*\h)--(-2*\d+\w,1*\h)--({e3}t)}
\dlp{({e4}b)--(2*\d-\w,-1*\h)--(2*\d+\w,0*\h)--({e2}t)}

\dblp{({e6}b)--(-2*\d-\w,-2*\h)--(-2*\d+\w,0*\h)--({e5}t)}
\blp{({e6}t)--(-2*\d-\w,1*\h)--(-2*\d+\w,-1*\h)--({e5}b)}
\alp{({e4}t)--(2*\d-\w,0*\h)--(2*\d+\w,-1*\h)--({e2}b)}
\dblp{({e1}t)--(-2*\d-\w,0*\h)--(-2*\d+\w,-2*\h)--({e3}b)}

\end{scope}
\end{tikzpicture}
};

\node[draw,my style] (D1) at (7, -6){
\begin{tikzpicture}

\oc{e1}{\coo{{B}}}{_1}{-2}{0}{-3}{\et\lt}

\oc{e2}{\rdual{\coo{{B}}}}{_2}{2}{0}{-3}{\etr\lt}

\oc{e3}{\evo{{B}}}{_3}{1}{-1}{-2}{\ep\lt}

\oc{e4}{\rdual{\evo{{B}}}}{_4}{-1}{-1}{-2}{\epr\lt}

\oc{e5}{\rdual{\coo{{A}}}}{_5}{-3}{-1}{-2}{\etr\lt}

\oc{e6}{\rdual{\evo{{A}}}}{_6}{-5}{-1}{-2}{\epr\lt}
\gc{g1}{-4}{-1}{-2}

\begin{scope}[on background layer]

\alp{({e1}b)--(0*\d-\w,-3*\h)--(0*\d+\w,-1*\h)--({e3}t)}
\dlp{({e4}b)--(0*\d-\w,-2*\h)--(0*\d+\w,0*\h)--({e2}t)}

\dblp{({e6}b)\gpb{g1}({e5}t)}
\blp{({e6}t)\gpt{g1}({e5}b)}
\alp{({e4}t)--(0*\d-\w,-1*\h)--(0*\d+\w,-3*\h)--({e2}b)}
\dlp{({e1}t)--(0*\d-\w,0*\h)--(0*\d+\w,-2*\h)--({e3}b)}

\end{scope}
\end{tikzpicture}
};

\node[draw,my style] (E1) at (7, -11){
\begin{tikzpicture}

\oc{e1}{\coo{{B}}}{_1}{-2}{-1}{-2}{\et\lt}

\oc{e2}{\rdual{\coo{{B}}}}{_2}{3}{-1}{-2}{\etr\lt}

\oc{e3}{\evo{{B}}}{_3}{0}{-1}{-2}{\ep\lt}

\oc{e4}{\rdual{\evo{{B}}}}{_4}{1}{-1}{-2}{\epr\lt}

\oc{e5}{\rdual{\coo{{A}}}}{_5}{-3}{-1}{-2}{\etr\lt}

\oc{e6}{\rdual{\evo{{A}}}}{_6}{-5}{-1}{-2}{\epr\lt}
\gc{g1}{-4}{-1}{-2}

\begin{scope}[on background layer]

\alp{({e1}b)--(-1*\d-\w,-2*\h)--(-1*\d+\w,-1*\h)--({e3}t)}
\dlp{({e4}b)--(2*\d-\w,-2*\h)--(2*\d+\w,-1*\h)--({e2}t)}

\dblp{({e6}b)\gpb{g1}({e5}t)}
\blp{({e6}t)\gpt{g1}({e5}b)}
\alp{({e4}t)--(2*\d-\w,-1*\h)--(2*\d+\w,-2*\h)--({e2}b)}
\dlp{({e1}t)--(-1*\d-\w,-1*\h)--(-1*\d+\w,-2*\h)--({e3}b)}

\end{scope}
\end{tikzpicture}
};

\node[draw,my style] (E2) at (25, -11){
\begin{tikzpicture}

\oc{e1}{\coo{{A}}}{_1}{-2}{-1}{-2}{\et\lt}

\oc{e2}{\rdual{\coo{{B}}}}{_2}{3}{-1}{-2}{\etr\lt}

\oc{e3}{\evo{{A}}}{_3}{0}{-1}{-2}{\ep\lt}

\oc{e4}{\rdual{\evo{{B}}}}{_4}{1}{-1}{-2}{\epr\lt}

\oc{e5}{\rdual{\coo{{A}}}}{_5}{-3}{-1}{-2}{\etr\lt}

\oc{e6}{\rdual{\evo{{A}}}}{_6}{-5}{-1}{-2}{\epr\lt}
\gc{g1}{-4}{-1}{-2}

\begin{scope}[on background layer]

\blp{({e1}b)--(-1*\d-\w,-2*\h)--(-1*\d+\w,-1*\h)--({e3}t)}
\dlp{({e4}b)--(2*\d-\w,-2*\h)--(2*\d+\w,-1*\h)--({e2}t)}

\dblp{({e6}b)\gpb{g1}({e5}t)}
\blp{({e6}t)\gpt{g1}({e5}b)}
\alp{({e4}t)--(2*\d-\w,-1*\h)--(2*\d+\w,-2*\h)--({e2}b)}
\dblp{({e1}t)--(-1*\d-\w,-1*\h)--(-1*\d+\w,-2*\h)--({e3}b)}

\end{scope}
\end{tikzpicture}
};

\draw[->](C0)to [out = 50, in =130,looseness=0.5]node [midway , fill=white] {$\tr(\serredual{f}{A}{B} )$}(C4);
\draw[->](C0)--(C1)node [midway , fill=white] {$\chi\eqref{fig:serre_dual_map_1}$};
\draw[->](C0)to [out = -70, in =150] node [midway , fill=white] {$\chi(U_A)\otimes \chi(\coo{B})$}(D1);
\draw[->](C0)to [out = -90, in =150]node [midway , fill=white] {$\id\otimes \mathrm{copairing}$}(E1);

\draw[->](C2)--(C4)node [midway , fill=white] {$\chi\eqref{fig:serre_dual_map_3}$};
\draw[->](D2)to [out = 20, in =-110] node [midway , fill=white] {$\chi(\evo{A})\otimes \chi(U_B)$}(C4);
\draw[->](E2)to [out = 40, in =-90]node [midway , fill=white] {$ \mathrm{pairing}\otimes\id$}(C4);

\draw[->](C1)--(D1)node [midway , fill=white] {$\gammat$};
\draw[->](C2)--(D2)node [midway , fill=white] {$\gammat$};
\draw[->](C1)--(D1a)node [midway , fill=white] {$\gammat$};
\draw[->](D1)--(E1)node [midway , fill=white] {$\gammat$};
\draw[->](D2)--(E2)node [midway , fill=white] {$\gammat$};
\draw[->](D1a)--(E1)node [midway , fill=white] {$\gammat$};

\draw[->](C1)--(C2)node [midway , fill=white] {$\tr\eqref{fig:serre_dual_map_1}$};
\draw[->](D1a)--(D2)node [midway , fill=white] {$\tr(\id\otimes f)\otimes \id$};
\draw[->](E1)--(E2)node [midway , fill=white] {$\id\otimes \tr(f)\otimes \id$};

\node at  (barycentric cs:E1=1,C0=.3){\cref{lem:ec_e_and_c}};
\node at  (barycentric cs:E2=1,C4=.3){\cref{lem:ec_e_and_c}};

\node at  (barycentric cs:C0=.6,D1=1,C1=.2){\cref{lem:trace_otimes_compatible}};
\node at  (barycentric cs:C2=.2,C4=.6,D2=1){\cref{lem:trace_otimes_compatible}};
\node at  (barycentric cs:C1=1,C2=1,D1a=1,D2=1){\cref{lem:trace_otimes_compatible}};
\node at  (barycentric cs:E1=1,E2=1,D1a=1,D2=1){\cref{lem:trace_otimes_compatible}};
\node at  (barycentric cs:D1a=1,D1=1){(1)};
\node (Cshift) at (16, 6){};
\node at  (barycentric cs:C1=1,C2=1,Cshift=1){\cref{thm:composite}};

\end{tikzpicture}

}
\caption{Trace of Serre dual}\label{fig:serre_dual_maps_big}
\end{figure}
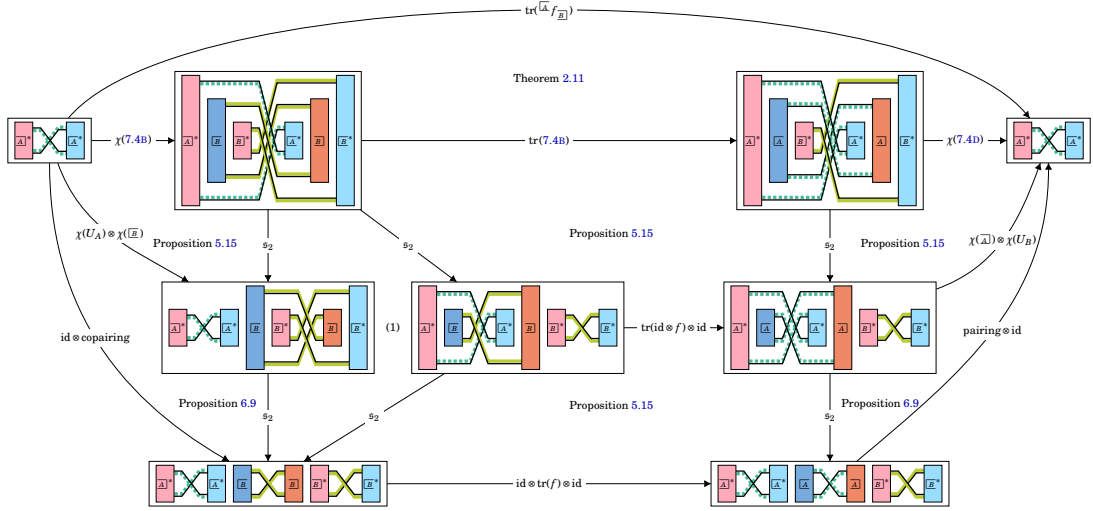

In this form, Serre duals arise naturally within the shadow isomorphisms. 

\begin{lem}\label{prop:serre_dual_odot_shadow}Suppose $A$ and $B$ are 1-dualizable.  \hfill
\begin{enumerate}
\item If $M$ is a 1-cell in $\sB(A,B)$ and $N$ is a 1-cell in $\sB(B,A)$, 
there is a natural isomorphism 
\begin{equation}\label{eq:serre_dual_odot_shadow}\sh{M\odot N}
	\cong \coo{A}\odot \left(M\otimes  \serredual{N}{B}{A} \right)\odot\Gamma\odot  \evo{B}.
\end{equation}

\item  If $M$  is a 1-cell in $\sB(A,B)$ and $N$ is a 1-cell in $\sB(\zdual{A},\zdual{B})$, 
there is a natural isomorphism
\begin{equation}\label{eq:serre_dual_odot_shadow_2}
\sh{M\odot\serredual{N}{\zdual{A}}{\zdual{B}}}
\cong
\coo{\zdual{A}}\odot \left(N\otimes M\right)\odot \Gamma\odot \evo{\zdual{B}}.
\end{equation}
\end{enumerate}
\end{lem}

\begin{proof}
The  composite in \cref{fig:serre_dual_odot_shadow} of the map $\trietop$ in \eqref{eq:trietop} and one of the pseduonaturality maps for $\Gamma$ defines the maps in  \eqref{eq:serre_dual_odot_shadow}.  
Since both maps in the composite are natural by assumption, the composite is natural.

The map in \eqref{eq:serre_dual_odot_shadow_2} is the composite in \cref{fig:serre_dual_odot_shadow_2} of $\trietop$ in \eqref{eq:trietop} and one of the  pseduonaturality maps for $\Gamma$.  It is also natural by hypothesis.
\end{proof}
\begin{figure}[h!]

\begin{subfigure}{.8\textwidth}
\resizebox{\textwidth}{!}{
\resizebox{.9\textwidth}{!}
{
\begin{tikzpicture}[
    mystyle/.style={%
    },
   my style/.style={%
   },
  ]
 
\def\scale{1.15}

\node[draw,my style] (D2) at (\scale*-5, 0){
\begin{tikzpicture}
\oc{t8}{\evo{A}}{_8}{-7}{0}{-3}{\ep}
\oc{m}{N}{}{-9}{-3}{-3}{\en}

\oc{e3}{\coo{B}}{_3}{-10}{-2}{-3}{\et}

\oc{t1}{\evo{B}}{_1}{-9}{-1}{-2}{\ep\dk}
\oc{n}{M}{}{-10}{-1}{-1}{\en}

\oc{e2}{\coo{A}}{_2}{-11}{0}{-1}{\et\dk}

\begin{scope}[on background layer]
\dlp{({e3}t)--
({t1}b)}
\alp{({n}t)--
({t1}t)}
\clp{({e2}b)--({n}t)}

\alp{({e3}b)--({m}t)}
\clp{({m}t)--(-8*\d-\w,-3*\h)--(-8*\d+\w,0*\h)--
({t8}t)}

\dclp{({e2}t)--(-8*\d-\w,0*\h)--(-8*\d+\w,-3*\h)--
({t8}b)}
\end{scope}
\end{tikzpicture}
};

\node[draw,my style] (D3) at (\scale*-10, 0){
\begin{tikzpicture}
\oc{t8}{\evo{A}}{_8}{-7}{0}{-1}{\ep}
\oc{m}{N}{}{-9}{-1}{-1}{\en}

\oc{n}{M}{}{-10}{-1}{-1}{\en}

\oc{e2}{\coo{A}}{_2}{-11}{0}{-1}{\et\dk}

\begin{scope}[on background layer]
\alp{({n}t)--({m}t)}
\clp{({e2}b)--({n}t)}

\clp{({m}t)--(-8*\d-\w,-1*\h)--(-8*\d+\w,0*\h)--
({t8}t)}

\dclp{({e2}t)--(-8*\d-\w,0*\h)--(-8*\d+\w,-1*\h)--
({t8}b)}
\end{scope}
\end{tikzpicture}
};

\node[draw,my style] (D1) at (1, 0){
\begin{tikzpicture}
\oc{t8}{\evo{A}}{_8}{-8}{1}{0}{\ep}
\oc{m}{N}{}{-9}{1}{1}{\en}

\oc{e3}{\coo{B}}{_3}{-10}{2}{1}{\et}

\oc{t1}{\evo{B}}{_1}{-6}{2}{-1}{\ep\dk}
\oc{n}{M}{}{-9}{-1}{-1}{\en}

\oc{e2}{\coo{A}}{_2}{-11}{0}{-1}{\et\dk}

\begin{scope}[on background layer]
\alp{({n}t)--(-7*\d-\w,-1*\h)--(-7*\d+\w,2*\h)--({t1}t)}
\clp{({e2}b)--({n}t)}

\alp{({e3}b)--({m}t)}
\clp{({m}t)--
({t8}t)}

\dclp{({e2}t)--
({t8}b)}
\dlp{({e3}t)--(-7*\d-\w,2*\h)--(-7*\d+\w,-1*\h)--({t1}b)}
\end{scope}
\end{tikzpicture}
};

\draw[<-](D1)--(D2)node [midway , fill=white] {$\gammao$};
\draw[<-](D2)--(D3)node [midway , fill=white] {$\trietop$};
\end{tikzpicture}
}}
\caption{The isomorphism in \eqref{eq:serre_dual_odot_shadow}}\label{fig:serre_dual_odot_shadow}
\end{subfigure}

\begin{subfigure}{.8\textwidth}
\resizebox{\textwidth}{!}{
\resizebox{.9\textwidth}{!}
{
\begin{tikzpicture}[
    mystyle/.style={%
    },
   my style/.style={%
   },
  ]
 
\def\scale{1.15}

\node[draw,my style] (D2) at (\scale*-4, 0){
\begin{tikzpicture}
\oc{t8}{\evo{A}}{_8}{-7}{1}{0}{\ep}
\oc{n}{N}{}{-9}{-1}{-1}{\en}

\oc{e3}{\coo{A}}{_3}{-12}{1}{-2}{\et}

\oc{t1}{\evo{\zdual{B}}}{_1}{-8}{-1}{-2}{\ep\dk}
\oc{m}{M}{}{-11}{-2}{-2}{\en}

\oc{e2}{\coo{\zdual{A}}}{_2}{-10}{0}{-1}{\et\dk}

\begin{scope}[on background layer]
\dblp{({n}t)--
({t1}t)}
\alp{({e2}t)--(-8*\d-\w,0*\h)--(-8*\d+\w,1*\h)--
({t8}t)}

\dlp{({e2}b)--({n}t)}

\blp{({m}t)--({t1}b)}

\alp{({e3}b)--({m}b)}
\dlp{({e3}t)--(-8*\d-\w,1*\h)--(-8*\d+\w,0*\h)--({t8}b)}
\end{scope}
\end{tikzpicture}
};

\node[draw,my style] (D3) at (\scale*-10, 0){
\begin{tikzpicture}
\oc{t8}{\evo{A}}{_8}{-10}{1}{0}{\ep}
\oc{n}{N}{}{-10}{-2}{-2}{\en}

\oc{e3}{\coo{A}}{_3}{-11}{0}{-1}{\et}

\oc{t1}{\evo{\zdual{B}}}{_1}{-8}{-1}{-2}{\ep\dk}
\oc{m}{M}{}{-10}{-1}{-1}{\en}

\oc{e2}{\coo{\zdual{A}}}{_2}{-12}{1}{-2}{\et\dk}

\begin{scope}[on background layer]
\dblp{({n}t)--(-9*\d-\w,-2*\h)--(-9*\d+\w,-1*\h)--
({t1}t)}
\alp{({e2}t)--%
({t8}t)}

\dlp{({e2}b)--({n}t)}

\blp{({m}t)--(-9*\d-\w,-1*\h)--(-9*\d+\w,-2*\h)--({t1}b)}

\alp{({e3}b)--({m}b)}
\dlp{({e3}t)--%
({t8}b)}
\end{scope}
\end{tikzpicture}
};

\node[draw,my style] (D4) at (\scale*-15, 0){
\begin{tikzpicture}
\oc{n}{N}{}{-10}{-2}{-2}{\en}

\oc{t1}{\evo{\zdual{B}}}{_1}{-8}{-1}{-2}{\ep\dk}
\oc{m}{M}{}{-10}{-1}{-1}{\en}

\oc{e2}{\coo{\zdual{A}}}{_2}{-11}{-1}{-2}{\et\dk}

\begin{scope}[on background layer]
\dblp{({n}t)--(-9*\d-\w,-2*\h)--(-9*\d+\w,-1*\h)--
({t1}t)}
\alp{({e2}t)--({m}b)}

\dlp{({e2}b)--({n}t)}

\blp{({m}t)--(-9*\d-\w,-1*\h)--(-9*\d+\w,-2*\h)--({t1}b)}

\end{scope}
\end{tikzpicture}
};

\draw[<-](D3)--(D4)node [midway , fill=white] {$\trietop$};
\draw[->](D3)--(D2)node [midway , fill=white] {$\gammael$};
\end{tikzpicture}
}}
\caption{The isomorphism in \eqref{eq:serre_dual_odot_shadow_2}}\label{fig:serre_dual_odot_shadow_2}
\end{subfigure}
\caption{The isomorphisms relating compositions in \cref{prop:serre_dual_odot_shadow}. }\label{fig:serredual_shadow}
\end{figure}
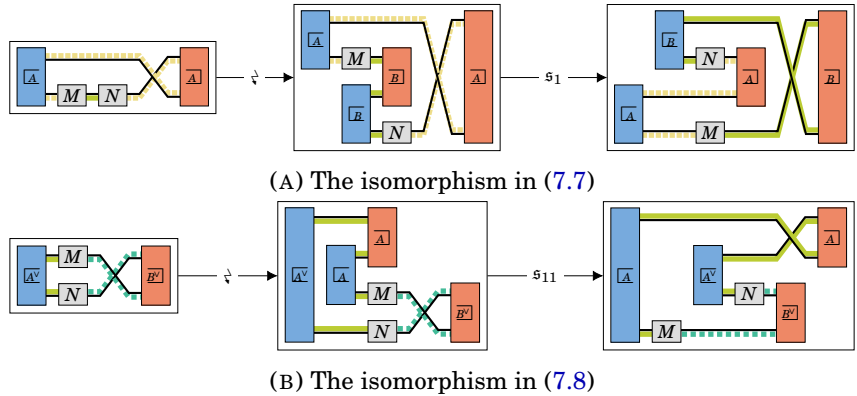

Making some of the naturality statements in \cref{prop:serre_dual_odot_shadow} explicit, if $M\in \sB(A,B)$ and $N\in \sB(B,A)$  and $f\colon M\to M$ and $g\colon N\to N$
the following diagram commutes.
\begin{equation}\label{eq:serre_duality_shadow_nat}
	\xymatrix{
	\sh{M\odot N}\ar[r]^-{\eqref{eq:serre_dual_odot_shadow}}\ar[d]^{\sh{f\odot g}}
	&{\coo{A}\odot \left(M\otimes  \serredual{N}{B}{A} \right)\odot\Gamma\odot  \evo{B}}
		\ar[d]^{\id_{\coo{A}}\odot \left(f\otimes \serredual{g}{B}{A} \right)\odot \id_\Gamma \odot \id_{\evo{B}}}
\\
\sh{M\odot N}\ar[r]^-{\eqref{eq:serre_dual_odot_shadow}}
	&{\coo{A}\odot \left(M\otimes  \serredual{N}{B}{A} \right)\odot\Gamma\odot  \evo{B}}
}\end{equation}
If $M\in \sB(A,B)$,  $N\in \sB(\zdual{A},\zdual{B})$,
$f\colon M\to M$ and $g\colon N\to N$ the following diagram commutes.
\begin{equation}\label{eq:serre_duality_shadow_nat_2}
\xymatrix{\sh{M\odot \serredual{N}{\zdual{A}}{\zdual{B}} }
		\ar[d]^{\sh{f\odot \serredual{g}{\zdual{A}}{\zdual{B}} }}
			\ar[r]^-{\eqref{eq:serre_dual_odot_shadow_2}}
	&
	{\coo{\zdual{A}}}\odot (N\otimes M)\odot {\Gamma}\odot{\evo{\zdual{B}}}
		\ar[d]^-{\id_{\coo{\zdual{A}}}\odot (g\otimes f)\odot\id_{\Gamma}\odot\id_{\evo{\zdual{B}}}}
\\
	\sh{ M\odot \serredual{N}{\zdual{A}}{\zdual{B}}}
		\ar[r]^-{\eqref{eq:serre_dual_odot_shadow_2}}
	&
{\coo{\zdual{A}}}\odot (N\otimes M)\odot{\Gamma}\odot{\evo{\zdual{B}}}}
\end{equation}
These two commuting squares allow us to compare the traces of the vertical maps.

\begin{cor}\label{prop:serre_dual_odot_trace}Suppose  $A$ and $B$ are 2-dualizable 0-cells.\hfill
\begin{enumerate}
\item\label{item:serre_dual_odot_trace} Let   $M \in \sB(A,B)$ and $N\in \sB(B,A)$ be right dualizable 1-cells.   
If $f\colon M\to M$ and $g\colon N\to N$, then  
\begin{align*}\tr(\sh{f\odot g})
	&=\tr\left(\id_{\coo{A}}\odot \left(f\otimes \serredual{g}{B}{A} \right)\odot \id_\Gamma \odot \id_{\evo{B}}\right)
\end{align*}
\item\label{item:serre_dual_odot_trace_2}  Let  $M\in \sB(A,B)$ and  $N\in \sB(\zdual{A},\zdual{B})$ be right dualizable 1-cells. 
If $f\colon M\to M$ and $g\colon N\to N$, then 
\[\tr(\sh{f\odot \serredual{g}{\zdual{A}}{\zdual{B}} })
=\tr\left(\id_{\coo{\zdual{A}}}\odot (g\otimes f)\odot\id_{\Gamma}\odot\id_{\evo{\zdual{B}}}\right)\]
\end{enumerate}
\end{cor}

\begin{proof}Note that the traces in the statement of this result are symmetric monoidal traces rather than bicategorical.  

We first make an elementary observation about symmetric monoidal traces.
Suppose $X$ and $Y$ are dualizable objects in a symmetric monoidal category, $h\colon X\to Y$ is an isomorphism, and the following diagram commutes.
\[\xymatrix{X\ar[r]^h\ar[d]^f&Y\ar[d]^g
\\X\ar[r]^h&Y}\] 
Then 
\[\tr(f)=\tr(h^{-1}hf)=\tr(h^{-1}gh)=\tr(ghh^{-1})=\tr(g)\]
The first and last equality replace $h^{-1}h$ and $hh^{-1}$ with identity maps.  The second equality from the left uses the commutativity of the square.  The remaining equality is the invariance of trace under cyclic permutation.

Then  
the commutative diagram in \eqref{eq:serre_duality_shadow_nat}  implies the traces of the vertical maps in \eqref{eq:serre_duality_shadow_nat}  are the same and proves \ref{item:serre_dual_odot_trace}.  The identifications in \ref{item:serre_dual_odot_trace_2}  follow from \eqref{eq:serre_duality_shadow_nat_2} in the same way.
\end{proof}

We have now reached one of the main results of this paper.  It  is the generalization of \cref{intro:shklyarov} and the manifestation of \cref{thm:goal}. 

\begin{thm}\label{thm:pairing_var_1} \label{thm:pairing_var_2}
 Suppose $A$ and $B$ are 2-dualizable 0-cells.
\begin{enumerate}
\item\label{item:pairing_var_1} If $M\in \sB(A,B)$ and $N\in \sB(B,A)$ are right dualizable  1-cells  and $f\colon M\to M$ and $g\colon N\to N$ then 
the diagram in \cref{fig:pairing_var_1} commutes.
\item \label{item:pairing_var_2}  If $M\in \sB(A,B)$ and $N\in \sB(\zdual{A},\zdual{B})$ are right dualizable 1-cells and $f\colon M\to M$ and $g\colon N\to N$  then the diagram in \cref{fig:pairing_var_2} commutes.
\end{enumerate}
\end{thm}

\begin{figure}[ht]
\begin{subfigure}{.48\textwidth}
\resizebox{\textwidth}{!}
{
\begin{tikzpicture}[
    mystyle/.style={%
    },
   my style/.style={%
   },
  ]
 
\def\scale{1.15}

\node[draw,my style] (A0) at (\scale*5,\scale* -4){\begin{tikzpicture}
		\node (ta) at (2,-1){};
		\node (ta) at (4,-2){};
		\end{tikzpicture}};

\node[draw,my style] (A4) at (\scale*13,\scale* -4){\begin{tikzpicture}
		\node (ta) at (2,-1){};
		\node (ta) at (4,-2){};
		\end{tikzpicture}};

\node[draw,my style] (B1) at (\scale*5,\scale*1){\begin{tikzpicture}
\oc{e3}{\rdual{\evo{A}}}{_3}{1}{2}{1}{\epr\lt}
\oc{t4}{\coo{A}}{_4}{-2}{2}{1}{\et\lt}

\oc{t5}{\rdual{\coo{A}}}{_5}{3}{2}{1}{\etr\lt}
\oc{e6}{\evo{A}}{_6}{0}{2}{1}{\ep\lt}
\begin{scope}[on background layer]

\dlp{({e3}b)--(2*\d-\w,\h)--(2*\d+\w,2*\h)--
({t5}t)}
\alp{({t4}b)--(-1*\d-\w,1*\h)--(-1*\d+\w,2*\h)--
({e6}t)}

\dlp{({t4}t)--(-1*\d-\w,2*\h)--(-1*\d+\w,1*\h)--
({e6}b)}
\alp{({e3}t)--(2*\d-\w,2*\h)--(2*\d+\w,1*\h)--
({t5}b)}
\end{scope}
\end{tikzpicture}
};

\node[draw,my style] (B3) at (\scale*13, \scale*1){\begin{tikzpicture}
\oc{e1}{\rdual{\evo{B}}}{_1}{-4}{0}{-1}{\epr}
\oc{t2}{\coo{B}}{_2}{-7}{0}{-1}{\et}

\oc{t7}{\rdual{\coo{B}}}{_7}{-2}{0}{-1}{\etr}
\oc{e8}{\evo{B}}{_8}{-5}{0}{-1}{\ep}

\begin{scope}[on background layer]

\dblp{({e1}b)--(-3*\d-\w,-1*\h)--(-3*\d+\w,0*\h)--({t7}t)}
\blp{({e1}t)--(-3*\d-\w,0*\h)--(-3*\d+\w,-1*\h)--({t7}b)}
\blp{({t2}b)--(-6*\d-\w,-1*\h)--(-6*\d+\w,0*\h)--({e8}t)}

\dblp{({t2}t)--(-6*\d-\w,0*\h)--(-6*\d+\w,-1*\h)--({e8}b)}

\end{scope}
\end{tikzpicture}
};

\draw[->](B1)edge node [midway , fill=white] {$\tr
(f)\otimes \tr
\left(\serredual{g}{B}{A}\right)$}(B3);

\draw[->](A0)edge node [midway , fill=white] {copairing}(B1) ;
\draw[->](B3) edge node [midway , fill=white] {pairing} (A4);

\draw[->](A0) edge node [midway , fill=white] {$\tr\left(\sh{f\odot g}
	\right)$} (A4);

\node at  (barycentric cs:A0=.5,B1=1,B3=1,A4=.5){\cref{thm:main_pairing}};
\node at  (barycentric cs:A0=1,B1=.25,B3=.25,A4=1){\cref{prop:serre_dual_odot_trace}};
\draw[->](A0) edge[bend left=50]node [midway , fill=white] {$\tr\left(\id_{\coo{A}}\odot ({f\otimes \serredual{g}{B}{A}})\odot \id_{\Gamma}\odot \id_{\evo{B}}
	\right)$} (A4);
\end{tikzpicture}
}
\caption{}\label{fig:pairing_var_1}
\end{subfigure}
\begin{subfigure}{.48\textwidth}
\resizebox{\textwidth}{!}
{
\begin{tikzpicture}[
    mystyle/.style={%
    },
   my style/.style={%
   },
  ]

\def\scale{1.15}

\node[draw,my style] (A0) at (\scale*5, \scale*-4){\begin{tikzpicture}
		\node (ta) at (2,-1){};
		\node (ta) at (4,-2){};
		\end{tikzpicture}};

\node[draw,my style] (A4) at (\scale*13,\scale* -4){\begin{tikzpicture}
		\node (ta) at (2,-1){};
		\node (ta) at (4,-2){};
		\end{tikzpicture}};

\node[draw,my style] (B1) at (\scale*5,\scale*1){\begin{tikzpicture}
\oc{e3}{\rdual{\evo{\zdual{A}}}}{_3}{1}{2}{1}{\epr\lt}
\oc{t4}{\coo{\zdual{A}}}{_4}{-2}{2}{1}{\et\lt}

\oc{t5}{\rdual{\coo{\zdual{A}}}}{_5}{3}{2}{1}{\etr\lt}
\oc{e6}{\evo{\zdual{A}}}{_6}{0}{2}{1}{\ep\lt}
\begin{scope}[on background layer]

\dlp{({e3}b)--(2*\d-\w,\h)--(2*\d+\w,2*\h)--
({t5}t)}
\alp{({t4}b)--(-1*\d-\w,1*\h)--(-1*\d+\w,2*\h)--
({e6}t)}

\dlp{({t4}t)--(-1*\d-\w,2*\h)--(-1*\d+\w,1*\h)--
({e6}b)}
\alp{({e3}t)--(2*\d-\w,2*\h)--(2*\d+\w,1*\h)--
({t5}b)}
\end{scope}
\end{tikzpicture}
};

\node[draw,my style] (B3) at (\scale*13, \scale*1){\begin{tikzpicture}
\oc{e1}{\rdual{\evo{\zdual{B}}}}{_1}{-4}{0}{-1}{\epr}
\oc{t2}{\coo{\zdual{B}}}{_2}{-7}{0}{-1}{\et}

\oc{t7}{\rdual{\coo{\zdual{B}}}}{_7}{-2}{0}{-1}{\etr}
\oc{e8}{\evo{\zdual{B}}}{_8}{-5}{0}{-1}{\ep}

\begin{scope}[on background layer]

\dblp{({e1}b)--(-3*\d-\w,-1*\h)--(-3*\d+\w,0*\h)--({t7}t)}
\blp{({e1}t)--(-3*\d-\w,0*\h)--(-3*\d+\w,-1*\h)--({t7}b)}
\blp{({t2}b)--(-6*\d-\w,-1*\h)--(-6*\d+\w,0*\h)--({e8}t)}

\dblp{({t2}t)--(-6*\d-\w,0*\h)--(-6*\d+\w,-1*\h)--({e8}b)}

\end{scope}
\end{tikzpicture}
};

\draw[->](B1)--(B3)node [midway , fill=white] {$\tr
(g)\otimes \tr
(f)$};

\draw[->](A0)edge node [midway , fill=white] {copairing}(B1) ;
\draw[->](B3) edge node [midway , fill=white] {pairing} (A4);

\draw[->](A0) edge node [midway , fill=white] {$\tr\left(\sh{f\odot  \serredual{g}{\zdual{A}}{\zdual{B}}}
	\right)$} (A4);

\node at  (barycentric cs:A0=.5,B1=1,B3=1,A4=.5){\cref{thm:main_pairing}};
\node at  (barycentric cs:A0=1,B1=.25,B3=.25,A4=1){\cref{prop:serre_dual_odot_trace}};

\draw[->](A0) edge[bend left=50]node [midway , fill=white] {$\tr\left(\id_{\coo{\zdual{A}}}\odot ({g\otimes f})\odot \id_{\Gamma}\odot \id_{\evo{\zdual{B}}}
	\right)$} (A4);

\end{tikzpicture}
}
\caption{}\label{fig:pairing_var_2}
\end{subfigure}
\caption{Compatibilities between trace, pairing, and copairing (\cref{thm:pairing_var_1})}
\end{figure}

\begin{proof}
This theorem  is an immediate consequence of \cref{prop:serre_dual_odot_trace,thm:main_pairing}.  \cref{prop:serre_dual_odot_trace}  provides the equality between 
\[\tr(\sh{f\odot g})\quad\text{ and }\tr\left(\id_{\coo{A}}\odot (f\otimes \serredual{g}{B}{A})\odot \id_\Gamma\odot \id_{\evo{B}}\right).\]
\cref{thm:main_pairing} replaces this second trace by the composite of $\tr(f)$ and $\tr\left(\serredual{g}{B}{A}\right)$ with the copairing and pairing maps.  The proof for \ref{item:pairing_var_2}  is similar.
\end{proof}

\cref{thm:pairing_var_1} deserves to be amplified and stated in more specific terms. 
\begin{example}
We work in the familiar bicategory of rings and bimodules. 
In this category the shadow is Hochschild homology. 
For an $(B, A)$-bimodule $N$, the Serre dual in $(A^\vee, B^\vee)$-bimodules is denoted $\serredual{N}{B}{A}$ as above  (assuming $A, B$ are suitably dualizable). 

Let $A$ be a $k$-algebra for some base ring $k$. The pairings defined in \cref{fig:copairing_pairing_defn_fig} become maps
\[
\hh(A) \otimes \hh(A^\op) \to \hh(k)  \qquad \hh(B) \otimes \hh(B^\op) \to \hh(k) 
\]
If we are given $M \in \mc{B}(A, B)$ and $N \in \mc{B}(B, A)$, we furthermore have traces induced by the identity maps of $M$ and $\serredual{N}{B}{A}$:
\[
\hh(A) \xto{\chi(M)} \hh(B) \qquad \hh(A^\vee) \xto{\chi\left(\serredual{N}{B}{A}\right)} \hh(B^\vee) 
\]
which we can tensor together to a map
\[
\hh(A) \otimes \hh(A^\vee) \xto{\chi(M)\otimes \chi\left(\serredual{N}{B}{A}\right)} \hh(B) \otimes \hh(B^\vee).
\]
In the presence of 2-duality, both $\hh(A)$ and $\hh(B)$ themselves become dualizable and we can compose with the copairing and pairing maps to obtain
\[
\hh(k) \to \hh(A) \otimes \hh(A^\vee) \xto{\chi(M)\otimes \chi\left(\serredual{N}{B}{A}\right)} \hh(B) \otimes \hh(B^\vee) \to \hh(k). 
\]
The theorem says that this can be computed in one step: consider $M \odot_B N$ and take $\hh(A; M \odot_B N)$. This, itself, is dualizable as a $k$-algebra and so we can compute
\[
\hh(k) \to \hh(A; M \odot_B N) \otimes \hh(A; M \odot_B N)^\vee \to \hh(k). 
\]
This corresponds to the toroidal trace diagrams in the introduction.
\end{example}

\begin{rmk}
This also illustrates quite clearly why Serre duality needs to rear its head: if we have $M$ an $(A, B)$-bimodule and $N$ a $(B, A)$-bimodule, we can combine them by either looking at $M \odot_B N$ or $N \odot_A M$. In either case, in order to obtain pairings from their respective traces, we have to obtain $(A^\vee, B^\vee)$-bimodules. 
\end{rmk}

\section{The Motivating Results}\label{sec:motivation}

In this section we use our abstract categorical results to prove the results that motivated this work originally, Shklyarov's Hirzebruch-Riemann-Roch theorem for DG-algebras \cite{shklyarov} and Petit's Riemann-Roch theorem \cite{petit}. Both theorems essentially follow from \cref{thm:pairing_var_1}, however some work is needed to identify pairings defined in each paper with our pairing.

\subsection{Shklyarov's Hirzebruch-Riemann-Roch theorem for DG-algebras}
\label{sec:motivation:Shklyarov}

In this subsection we prove 
Shklyarov's Hirzebruch-Riemann-Roch theorem. It is an easy consequence of our theorems, but in order to state the theorem in his form, we  restate his  pairing and invariants. All of these are defined in terms of familiar  properties of Hochschild homology and DG-categories: the shuffle isomorphism and the Morita invariance of Hochschild homology.

To define a pairing as Shklyarov does, we recall the following basic, but powerful fact about Hochschild homology: If $A$ is a dg-$k$-algebra there is a canonical isomorphism
\begin{equation}\label{shk:perf_to_A}
\HH(\Perf A) \simeq\HH(A) .
\end{equation}
given by Morita equivalence. The map $\HH(A) \to \HH(\Perf A)$ is induced by the inclusion. See \cite{campbell_ponto} for (much) more detail. 

Shklyarov defines a pairing $\cup\colon \HH(A)\otimes \HH(A^\op)\to \HH(k)$ as follows.

\begin{defn}\cite[(3.2)]{shklyarov} \label{pairing:shklyarov}
  We define
   \[
  \cup \colon \HH(A) \otimes \HH(A^\op) \to \HH(k)
  \]
  as the composite 
  \[
\HH(A) \otimes \HH(A^\op) \xrightarrow{\mathrm{sh}} \hh(A \otimes A^\op) \xrightarrow{\cong} \hh(\Perf_{A\otimes A^\op}) \xrightarrow{-\otimes_{A\otimes A^\op} A_k} \HH(\Perf k) \xrightarrow{\cong} \HH(k)
  \]
where the first map is the shuffle map.  The maps denoted with $\cong$ are the isomorphism (or its inverse) in \eqref{shk:perf_to_A}.  The map denoted $-\otimes_{A\otimes A^\op} A_k$ tensors a module or homomorphism in $\Perf_{A\otimes A^\op}$ with $_{A\otimes A^\op} A_k$ over ${A\otimes A^\op}$.
\end{defn}

Our work in \cite{campbell_ponto} (in particular the Hochschild homology version of \cite[Ex. 5.12]{campbell_ponto}) immediately implies that our pairing agrees with Shklyarov's.

\begin{lem}\label{lem:shk_pairing}
The pairing from  \cref{fig:copairing_pairing_defn_fig} agrees with the pairing in \cref{pairing:shklyarov}.
\end{lem}
\begin{proof}
  The proof is a diagram chase:
  \[
  \xymatrix@C=70pt{
     & \HH(\Perf_{A \otimes A^\op})\ar[r]^{\otimes_{A \otimes A^\op} A}  & \HH(\Perf_k) \\
    \HH(A) \otimes \HH(A^\op) \ar[r]^{\mathrm{sh}}  & \HH(A \otimes A^\op) \ar[r]^{\chi\left(\evo{A}\right)} \ar[u]^{\cong} & \HH(k) \ar[u]^{\cong} 
    }
  \]
  The commutativity of the square is a special case of \cite[Lemma 5.8]{campbell_ponto}. The fact that the composite along the bottom is the pairing from \cref{fig:copairing_pairing_defn_fig} is \cref{lem:ec_e_and_c}. 
\end{proof}

Shklyarov defines invariants $\mathrm{eu}(L)$ and $\mathrm{Eu}(L)$ closely related to our trace invariants, but not \textit{exactly} the same. Here are his definitions. 

\begin{defn}
  For a perfect dg $A$ module $L$ \cite[\S1.3]{shklyarov} defines  $\mathrm{Eu}(L) \in \HH_0(\Perf A)$ to be the image of $L$ in $\HH_0(\Perf A)$ and $\mathrm{eu}(L)$ to be the image of $\mathrm{Eu}(L)$ under the isomorphism in \eqref{shk:perf_to_A}.
  \end{defn}

We show agreement between these invariants and the Euler characteristics above.

\begin{lem}\label{lem:shk_euler}
The image of $1\in k$ under the composite 
\[k\to \HH_0(k)\xto{\chi(L)}\HH(A)\]
is $\eu(L)$.
\end{lem}

The map $k\to \HH_0(k)$ is the isomorphism that recognizes the quotient on $\HH_0(k)$ is by the zero submodule.

\begin{proof}
Consider the diagram 
\begin{equation}\label{eq:shk_comparisons_euler}\xymatrix{k\ar[r]\ar[d]& \HH_0(\Perf_A)
\\
\HH_0(k)\ar[r]^-{\chi(L)}&\HH_0(A)\ar[u]^\cong 
}\end{equation}
where the top arrow sends $1$ to the class of $L$ in $\HH_0(\Perf_A)$.
The diagram in \eqref{eq:shk_comparisons_euler} commutes by  \cite[Lemma 5.8]{campbell_ponto}.
\end{proof}

Assembling the results establishing the agreement between the various pairings and Euler characteristics, Shklyarov's HRR theorem follows easily from 
\cref{thm:pairing_var_1}, but 
we need to say something about a subtle point.
  Following Shklyarov, we use the functor
  \begin{equation}\label{eq:first_d}
	\Perf_A \to \Perf_{A^{\op}} 
  \end{equation}
  given by $M \mapsto \hom_A (M, A)$. This functor hides some very common identifications. 
\begin{itemize}
\item The functor $M \mapsto \hom_A (M, A)$ takes an $(A, k)$-bimodule $M$ to a $(k, A)$-bimodule: indeed, $A$ acts on the right of $\hom_A (M, A)$. 
\end{itemize}
If one prefers, one could also note that $\hom_A (M, A)$ being a $(k, A)$-bimodule is correct from the viewpoint of bicategorical duality. 
\begin{itemize}
\item 
We then view an $(k, A)$-bimodule as a $(A^\op, k)$-bimodule by declaring that an $A^\op$-action on the left is the same as an $A$-action on the right (in this situation, we ignore $k$ since it is commutative). This gives a functor 
\[\operatorname{Mod}_{(k, A)} \to \operatorname{Mod}_{(A^\op, k)}.\] 
\end{itemize}
This transition is accomplished with a base-change object. Indeed, switching from a $(k, A)$-bimodule to an $(A^\op, k)$-module amounts to tensoring with the base-change object $\ _{A\otimes A^\op} A_k$. So, the switch from $(k, A)$-bimodules to $(A^\op, k)$-bimodules amounts to taking the Serre dual $\serredual{(\hom_A (M, A))}{k}{A}$. 

When we very carefully keep track of sidedness, the functor \eqref{eq:first_d}, which we denote $D$,  is  the composite
  \[
    \Perf_{(A, k)}\xto{\hom_A (-,A)}   \Perf_{(k, A)} \xto{\serredual{(-)}{k}{A}} \Perf_{(A^\op, k)}
  \]
\begin{thm}\label{thm:shklyarov}\cite[Thm. 3]{shklyarov} For any perfect DG $A$-modules $M$ and $N$ define 
\[\chi(M, N) \coloneqq \chi(\hom_{\Perf A} (M, N)).\] 
Then
\[\chi (M,N)  =  \eu(N) \cup \eu(DM) \]
\end{thm}

\begin{proof}
The follows from 
\cref{thm:pairing_var_1}  with the identifications in \cref{lem:shk_pairing,lem:shk_euler} but we give more detail because of the subtleties noted above.

If $M$ and $N$ are $(A, k)$-bimodules and perfect as $A$-modules, the module $\hom_A (M, A)$ is a $(k, A)$-bimodule and we are 
in the situation of \cref{thm:pairing_var_1}.  That theorem 
gives the commutativity of the following (recalling that $DM$ is the notation for the Serre dual of $\hom_A (M, A)$) 
\[
\xymatrix@C=70pt{
 \HH(k) \otimes \HH(k) \ar[rr]^{\chi(N) \otimes \chi(DM)} & & \HH(A) \otimes \HH(A^\op)  \ar[d]\\
  \hh(k) \ar[u]\ar[rr]_{\chi(\hom_A (M, A) \odot N)} & & \hh(k) 
  }
\]
In our current situation, $\hom_A(M, N) \odot N = \hom_A(M, A) \otimes_A N$ which is $\hom_A (M, N)$ as a $(k, k)$-bimodule. We have already identified the long way around the diagram as the pairing of the Euler classes. 
\end{proof}

\begin{rmk}
In fact, the above only requires that $A$ be proper --- only the right dualizability of $\evo{A}$ is used. The pairing on Hochschild homology may not be non-degenerate in this situation, but the theorem remains agnostic about that. 
\end{rmk}

\begin{rmk}
Shklyarov also identifies $\eu(DM)$ with $\eu(M)^\vee$.  This is the identification in \cref{lem:serre_dual_trace}.
\end{rmk}

\begin{rmk}
In Shklyarov's formula an integral sign also appears.
We viewed this as being rolled into the characters computed by Morita equivalence and so we omit it. 
\end{rmk}

\subsection{Petit's Riemann-Roch theorem for DG algebras}
Like \cref{intro:shklyarov}, the main result in \cite{petit} describes the interaction between a Hochschild class and a pairing map. In \cite[Rmk. 4.11]{petit},  Petit  observes that his Hochschild class $hh_A(M,f)$ agrees with the bicategorical trace of the identity map in \cite{ps:sym,ps:bicat}  so no further identifications will be necessary for the comparison of invariants.  

 Petit gives several descriptions for his pairing.  We will use the description in \cite[5.2]{petit} since it allows the most direct comparison to \cref{fig:pairing_defn_fig}.
Unless explicitly stated,  $\hom$ and $\otimes$ are derived so we will not separately notate this.
Petit's pairing $\cup$ is the following composite:
\begin{align}\label{petit:pairing}
\hom_{(B^\op)^e}&(\omega_{B^\op}^{-1},B^\op) \otimes \hom_{B^e}(\omega^{-1}_{B}, B )
\\
&\xto{(-\otimes \id_{\omega_{B^\op}})\otimes \id}
\hom_{(B^\op)^e}(\omega_{B^\op}^{-1}\otimes_{B^\op}\omega_{B^\op},B^\op\otimes_{B^\op}\omega_{B^\op}) \otimes \hom_{B^e}(\omega^{-1}_{B}, B )\nonumber
\\
&\xto{\hom(f,g)\otimes \id}
\hom_{^eB}(B^\op, \omega_{B^\op}) \otimes \hom_{B^e}(\omega^{-1}_{B}, B )\nonumber
\\
&\xto{\otimes}
\hom_{k}(B^\op\otimes_{B^e}\omega^{-1}_{B}, \omega_{B^\op}\otimes_{B^e}B )\xto{\hom(c,e)}
\hom_{k}(k, k)\cong k\nonumber
\end{align} 
where $B$ is a DGA, $B^e\coloneqq B \otimes B^\op$ and $^eB\coloneqq B^\op \otimes B$ and  the other objects and maps are defined in the left most column of \cref{notation_translation}.  The translations of $\omega^{-1}_B$ and $\omega^{-1}_{B^\op}$ look different in  \cref{notation_translation} because of sidedness considerations.  After converting to actions on the same side (as Petit does) there is no inconsistency.

\begin{table}
\begin{tabular}{c|c|c}
Petit \cite{petit}&This paper&\cref{lem:simp_petit} 
\\\hline
\hline

 $\omega^{-1}_B\coloneqq \hom_{\,^eB}(B^\op,\,^eB)$
&\resizebox{!}{.75cm}{\begin{tikzpicture}
\oc{t1}{\rdual{\coo{A}}}{_1}{-1}{1}{0}{\etr}
\gc{g1}{-2}{1}{0}
\begin{scope}[on background layer]
\dlp{(-3*\d,0*\h)\gpb{g1}({t1}t)}
\alp{ (-3*\d,\h)\gpt{g1}({t1}b)}
\end{scope}
\end{tikzpicture}} 
& $\rdual{J}$
\\\hdashline
&
\resizebox{!}{.75cm}{\begin{tikzpicture}
\oc{t1}{\coo{A}}{_1}{1}{1}{0}{\et}
\gc{g1}{2}{1}{0}
\begin{scope}[on background layer]
\alp{ ({t1}b)\gpb{g1}(3*\d,\h)}
\dlp{({t1}t)\gpt{g1}(3*\d,0*\h)}
\end{scope}
\end{tikzpicture}}&$J$
\\\hline
&\resizebox{!}{.75cm}{\begin{tikzpicture}
\oc{e2}{\evo{A}}{_2}{3}{1}{0}{\ep}
\gc{g1}{2}{1}{0}
\begin{scope}[on background layer]
\alp{ (2*\d,1*\h)--({e2}t)}
\dlp{(2*\d,0*\h)--({e2}b)}
\end{scope}

\end{tikzpicture}}
&$K$
\\\hdashline
&\resizebox{!}{.75cm}{\begin{tikzpicture}
\oc{e2}{\rdual{\evo{A}}}{_2}{-3}{1}{0}{\epr}
\gc{g1}{2}{1}{0}
\begin{scope}[on background layer]
\alp{ (-2*\d,1*\h)--({e2}t)}
\dlp{(-2*\d,0*\h)--({e2}b)}
\end{scope}

\end{tikzpicture}}
&$\rdual{K}$
\\\hline
 $\omega^{-1}_{B^\op}\coloneqq \hom_{\,^eB}(B,\,^eB)$&
\resizebox{!}{.75cm}{\begin{tikzpicture}
\oc{e3}{\rdual{\evo{\zdual{A}}}}{_3}{-2}{1}{0}{\epr}
\gc{g2}{-1}{1}{0}
\begin{scope}[on background layer]
\dlp{({e3}t)--(-1*\d,1*\h)}
\alp{ ({e3}b)--(-1*\d,0*\h)}
\end{scope}
\end{tikzpicture}}
&$N$
\\\hdashline
&
\resizebox{!}{.75cm}{\begin{tikzpicture}
\oc{t4}{{\evo{\zdual{A}}}}{_4}{0}{1}{0}{\ep\lt}
\gc{g2}{-1}{1}{0}
\begin{scope}[on background layer]
\dlp{(-1*\d,\h)--({t4}t)}
\alp{ (-1*\d,0)--({t4}b)}
\end{scope}
\end{tikzpicture}}
&$\rdual{N}$
\\\hline
&\resizebox{!}{.75cm}{\begin{tikzpicture}
\oc{e3}{{\coo{\zdual{A}}}}{_3}{-2}{1}{0}{\et\lt}
\gc{g2}{-1}{1}{0}
\begin{scope}[on background layer]
\dlp{({e3}b)\gpb{g2}(0,\h)}
\alp{ ({e3}t)\gpt{g2}(0,0)}
\end{scope}
\end{tikzpicture}}
&$M$
\\\hline
$\omega_{B^\op}\coloneqq \hom_k(B,k)$&
\resizebox{!}{1.5cm}{\begin{tikzpicture}
\gc{g1}{2}{-1}{-2}

\oc{t4}{{\evo{\zdual{A}}}}{_4}{-1}{0}{-1}{\ep\lt}
\gc{g2}{-4}{0}{-3}

\oc{e5}{\rdual{\evo{A}}}{_5}{-3}{-1}{-2}{\epr}
\gc{g3}{-2}{-1}{-2}
\gc{g4}{-1}{-2}{-3}

\begin{scope}[on background layer]
\alp{(-4*\d,0*\h)--({t4}t)}
\dlp{ (-4*\d,-3*\h)--(0*\d,-3*\h)}

\dlp{ ({e5}b)\gpb{g3}({t4}b)}
\alp{({e5}t)\gpt{g3}(0*\d,-2*\h)}
\end{scope}
\end{tikzpicture}}
&$L$
\\\hline\hline
$c\colon k \to \hom_{\,^eB}(B^\op, B^\op)\cong B^\op\otimes_{B^e}\omega^{-1}_{B}$
&$P_{\coo{A}}$
&$\eta_J$
\\
$e\colon \omega_{B^\op}\otimes_{B^e}B \to k$
&$L_{\evo{A}}$ 
&$\epsilon_K$
\\\hline
$B^\op\otimes_{B^\op} \omega_{B^\op}\to \omega_{B^\op}$
&\cref{petit_g} 
&$g\colon M\otimes_R L\to \rdual{K}$
\\
&\cref{petit_f_prime}
&
$f'\colon \rdual{N}\otimes_kJ\to L$,
\\
$B^\op\to \omega_{B^\op}^{-1}\otimes_{B^\op} \omega_{B^\op}$
&\cref{petit_f}
&$f\colon J\to N\otimes_RL$
\\
\hline
\end{tabular}
\caption{Notation translation}\label{notation_translation}
\end{table}

The main theorem of  \cite{petit} follows from \cref{thm:pairing_var_2}\ref{item:pairing_var_2} after we identify Petit's pairing \eqref{petit:pairing} with ours.  For the first step in the identification of the pairings the particular modules are not relevant, and in fact can be a distraction, so we replace the particular modules in \eqref{petit:pairing}  by the choices in \cref{notation_translation} that have the minimum required compatibility for \eqref{petit:pairing}  to be defined.  

The primary difference between \eqref{petit:pairing}  and our pairing is the presence of the internal hom. The goal of the following result is to replace those internal homs by duals.  In the process we obtain a significant simplification of \eqref{petit:pairing}.  To see this, compare the right and left hand sides of the displayed equation in \cref{lem:simp_petit}.  The right side is Petit's pairing and the left is the composite we will use in \cref{thm:petit}.

\begin{lem}\label{lem:simp_petit}
Let $R$ be a $k$-algebra, $J$, $M$, and $N$ be right $R$-modules, $K$ be a left $R$-module, and $L$ be an $(R,R)$-bimodule.  Suppose $J$, $N$, and $K$ are right dualizable.
For module homomorphisms 
$g\colon M\otimes_R L\to \rdual{K}$ and 
$f'\colon \rdual{N}\otimes_kJ\to L$,
the following diagram commutes.
\begin{equation}\label{eq:simp_petit}\xymatrix{
M\otimes_R\rdual{N}\otimes_kJ\otimes_RK\ar[r]^-\sim\ar[d]^{\id\otimes f'\otimes \id}
&\hom_R(N,M)\otimes \hom_R(\rdual{J},K)\ar[d]
\\
M\otimes_RL\otimes_RK\ar[d]^{g\otimes \id}
&\hom_R(N\otimes _RL,M\otimes_RL)\otimes \hom_R(\rdual{J},K)\ar[d]^{\hom((\id\otimes f')(\eta_N\otimes\id),g)}
\\
\rdual{K}\otimes_RK\ar[dd]^{\epsilon_K}
&\hom_R(J,\rdual{K})\otimes \hom_R(\rdual{J},K)\ar[d]^\otimes
\\
&\hom_k(J\otimes_R \rdual{J},\rdual{K}\otimes_R K)\ar[d]^{\hom(\eta_J,\epsilon_K)}
\\
k\ar[r]
&\hom(k,k)
}\end{equation}
\end{lem}

Note that $\rdual{K}=\hom_k(K,k)$,  $\rdual{J}=\hom_R(J,R)$ and $\rdual{N}=\hom_R(N,R)$.  The map $\eta_N$ is $k\to N\otimes_R \rdual{N}$ is the coevaluation for $N$.  Similarly, $\eta_J$ is the coevaluation for $J$ and $\epsilon_K$ is the evaluation for $K$.  The map $f$ is $(\id\otimes f')(\eta_N\otimes\id)$.

\begin{figure}
\begin{subfigure}[b]{0.6\textwidth}

\resizebox{\textwidth}{!}{
\begin{tikzpicture}[
    mystyle/.style={%
    },
   my style/.style={%
   },
  ]

\node[draw,my style] (A1a) at (-5, 1){\begin{tikzpicture}
\oc{t1}{\coo{A}}{_1}{1}{1}{0}{\et}
\gc{g1}{2}{1}{0}
\begin{scope}[on background layer]
\alp{ ({t1}b)\gpb{g1}(3*\d, 1*\h)}
\dlp{({t1}t)\gpt{g1}(3*\d, 0*\h)}
\end{scope}
\oc{t4}{{\evo{\zdual{A}}}}{_4}{0}{1}{0}{\ep\lt}
\gc{g2}{-1}{1}{0}
\begin{scope}[on background layer]
\dlp{(-1*\d, 1*\h)--({t4}t)}
\alp{ (-1*\d, 0*\h)--({t4}b)}
\end{scope}
\end{tikzpicture}
};

\node[draw,my style] (B1) at (0, 1){\begin{tikzpicture}
\oc{t1}{\coo{A}}{_1}{1}{-2}{-3}{\et}
\gc{g1}{2}{-2}{-3}
\begin{scope}[on background layer]
\alp{ ({t1}b)\gpb{g1}(3*\d, -2*\h)}
\dlp{({t1}t)\gpt{g1}(3*\d, -3*\h)}
\end{scope}
\oc{t4}{{\evo{\zdual{A}}}}{_4}{0}{0}{-1}{\ep\lt}
\gc{g2}{-4}{0}{-3}

\oc{e5}{\rdual{\evo{A}}}{_5}{-1}{-1}{-2}{\epr}
\oc{t6}{\rdual{\coo{A}}}{_6}{0}{-2}{-3}{\etr}
\gc{g3}{-2}{-1}{-2}
\gc{g4}{-1}{-2}{-3}

\begin{scope}[on background layer]
\dlp{(-2*\d, 0*\h)--({t4}t)}
\alp{ (-2*\d, -3*\h)--({t6}b)}

\dlp{ ({e5}b)--({t6}t)}
\alp{({e5}t)--({t4}b)}
\end{scope}
\end{tikzpicture}
};

\node[draw,my style] (C1) at (5, 1){\begin{tikzpicture}
\gc{g1}{2}{-2}{-3}
\oc{t4}{{\evo{\zdual{A}}}}{_4}{2}{0}{-1}{\ep\lt}
\gc{g2}{-4}{0}{-3}

\oc{e5}{\rdual{\evo{A}}}{_5}{1}{-1}{-2}{\epr}
\gc{g3}{-2}{-1}{-2}
\gc{g4}{-1}{-2}{-3}

\begin{scope}[on background layer]
\dlp{(0*\d, 0*\h)--({t4}t)}
\alp{ (0*\d, -3*\h)\gpb{g1}(3*\d, -2*\h)}

\dlp{ ({e5}b)\gpt{g1}(3*\d, -3*\h)}
\alp{({e5}t)--({t4}b)}

\end{scope}
\end{tikzpicture}
};

\draw[->](A1a)--(B1)node [midway , fill=white] {$\tridetop$};

\draw[->](B1)--(C1)node [midway , fill=white]{$P_{\coo{A}}$};

\end{tikzpicture}
}
\caption{$f'$}\label{petit_f_prime}
\end{subfigure}
\begin{subfigure}[b]{0.8\textwidth}
\resizebox{\textwidth}{!}{
\begin{tikzpicture}[
    mystyle/.style={%
    },
   my style/.style={%
   },
  ]

\node[draw,my style] (A1) at (-11, 1){\begin{tikzpicture}
\oc{t1}{\coo{A}}{_1}{1}{1}{0}{\et}
\gc{g1}{2}{1}{0}
\begin{scope}[on background layer]
\alp{ ({t1}b)\gpb{g1}(3*\d, 1*\h)}
\dlp{({t1}t)\gpt{g1}(3*\d, 0*\h)}
\end{scope}

\end{tikzpicture}
};

\node[draw,my style] (A1a) at (-5.5, 1){\begin{tikzpicture}
\oc{t1}{\coo{A}}{_1}{1}{1}{0}{\et}
\gc{g1}{2}{1}{0}
\begin{scope}[on background layer]
\alp{ ({t1}b)\gpb{g1}(3*\d, 1*\h)}
\dlp{({t1}t)\gpt{g1}(3*\d, 0*\h)}
\end{scope}
\oc{e3}{\rdual{\evo{\zdual{A}}}}{_3}{-2}{1}{0}{\epr\lt}
\oc{t4}{{\evo{\zdual{A}}}}{_4}{0}{1}{0}{\ep\lt}
\gc{g2}{-1}{1}{0}
\begin{scope}[on background layer]
\dlp{({e3}t)--({t4}t)}
\alp{ ({e3}b)--({t4}b)}
\end{scope}
\end{tikzpicture}
};

\node[draw,my style] (B1) at (.5, 1){\begin{tikzpicture}
\oc{t1}{\coo{A}}{_1}{1}{-2}{-3}{\et}
\gc{g1}{2}{-2}{-3}
\begin{scope}[on background layer]
\alp{({t1}b)\gpb{g1}(3*\d, -2*\h)}
\dlp{({t1}t)\gpt{g1}(3*\d, -3*\h)}
\end{scope}
\oc{e3}{\rdual{\evo{\zdual{A}}}}{_3}{-2}{0}{-3}{\epr\lt}
\oc{t4}{{\evo{\zdual{A}}}}{_4}{0}{0}{-1}{\ep\lt}
\gc{g2}{-4}{0}{-3}

\oc{e5}{\rdual{\evo{A}}}{_5}{-1}{-1}{-2}{\epr}
\oc{t6}{\rdual{\coo{A}}}{_6}{0}{-2}{-3}{\etr}
\gc{g3}{-2}{-1}{-2}
\gc{g4}{-1}{-2}{-3}

\begin{scope}[on background layer]
\dlp{({e3}t)--({t4}t)}
\alp{ ({e3}b)--({t6}b)}

\dlp{ ({e5}b)--({t6}t)}
\alp{({e5}t)--({t4}b)}
\end{scope}
\end{tikzpicture}
};

\node[draw,my style] (C1) at (6, 1){\begin{tikzpicture}
\gc{g1}{0}{-2}{-3}
\oc{e3}{\rdual{\evo{\zdual{A}}}}{_3}{-2}{0}{-3}{\epr\lt}
\oc{t4}{{\evo{\zdual{A}}}}{_4}{0}{0}{-1}{\ep\lt}
\gc{g2}{-4}{0}{-3}

\oc{e5}{\rdual{\evo{A}}}{_5}{-1}{-1}{-2}{\epr}
\gc{g3}{-2}{-1}{-2}
\gc{g4}{-1}{-2}{-3}

\begin{scope}[on background layer]
\dlp{({e3}t)--({t4}t)}
\alp{ ({e3}b)\gpb{g1}(1*\d, -2*\h)}

\dlp{ ({e5}b)\gpt{g1}(1*\d, -3*\h)}
\alp{({e5}t)--({t4}b)}

\end{scope}
\end{tikzpicture}
};

\draw[->](A1)--(A1a)node [midway , fill=white] {$I_{\coo{\zdual{A}}}$};

\draw[->](A1a)--(B1)node [midway , fill=white] {$\tridetop$};

\draw[->](B1)--(C1)node [midway , fill=white]{$P_{\coo{A}}$};

\end{tikzpicture}
}
\caption{$f$}\label{petit_f}
\end{subfigure}

\begin{subfigure}[b]{0.8\textwidth}
\resizebox{\textwidth}{!}{
\begin{tikzpicture}[
    mystyle/.style={%
    },
   my style/.style={%
   },
  ]

\node[draw,my style] (C2) at (1, -7){\begin{tikzpicture}
\gc{g1}{3}{0}{-3}

\oc{e3}{{\coo{\zdual{A}}}}{_3}{0}{0}{-1}{\et\lt}
\oc{t4}{{\evo{\zdual{A}}}}{_4}{1}{-1}{-2}{\ep\lt}
\gc{g2}{-4}{0}{-1}

\oc{e5}{\rdual{\evo{A}}}{_5}{0}{-2}{-3}{\epr}
\gc{g3}{-2}{-1}{-2}
\gc{g4}{2}{0}{-3}

\begin{scope}[on background layer]
\alp{
(2.5*\d,-3*\h)\gpb{g1}(4*\d,0*\h)}
\dlp{({e3}b)--({t4}t)}
\alp{ ({e5}t)--({t4}b)}
\dlp{({e5}b)\gpb{g4}(2.5*\d,0*\h)\gpt{g1}(4*\d,-3*\h)}
\alp{ ({e3}t)\gpt{g4}(2.5*\d,-3*\h)
}

\end{scope}
\end{tikzpicture}
};

\node[draw,my style] (C1) at (-5, -7){\begin{tikzpicture}
\gc{g1}{1}{-2}{-3}

\oc{e3}{{\coo{\zdual{A}}}}{_3}{-2}{0}{-3}{\et\lt}
\oc{t4}{{\evo{\zdual{A}}}}{_4}{1}{0}{-1}{\ep\lt}
\gc{g2}{-2}{0}{-3}

\oc{e5}{\rdual{\evo{A}}}{_5}{0}{-1}{-2}{\epr}
\gc{g3}{-2}{-1}{-2}
\gc{g4}{-1}{0}{-3}

\begin{scope}[on background layer]
\alp{
(0.5*\d,-3*\h)\gpb{g1}(2*\d,-2*\h)}
\dlp{({e3}b)\gpb{g4}({t4}t)}
\alp{ ({e5}t)--({t4}b)}
\dlp{({e5}b)\gpt{g1}(2*\d,-3*\h)}
\alp{ ({e3}t)\gpt{g4}(0.5*\d,-3*\h)
}

\end{scope}
\end{tikzpicture}
};

\node[draw,my style] (A2) at (6, -7){\begin{tikzpicture}
\gc{g1}{0}{1}{0}

\oc{e3}{\rdual{\evo{A}}}{_3}{-2}{1}{0}{\epr}
\gc{g2}{-1}{1}{0}
\begin{scope}[on background layer]
\alp{
(-.25,0)\gpb{g1}(1*\d,1*\h)}
\dlp{ ({e3}b)\gpb{g2}(-.25,\h)\gpt{g1}(1*\d,0*\h)}
\alp{ ({e3}t)\gpt{g2}(-.25,0)
}

\end{scope}
\end{tikzpicture}
};

\node[draw,my style] (A3) at (10, -7){\begin{tikzpicture}

\oc{e3}{\rdual{\evo{A}}}{_3}{-2}{1}{0}{\epr}
\begin{scope}[on background layer]
\dlp{({e3}b)--(0*\d,0*\h)}
\alp{ ({e3}t)--(0*\d,1*\h)}
\end{scope}
\end{tikzpicture}
};

\draw[->](C1)--(C2)node [midway , fill=white]{\eqref{ex:pseduonaturality_gamma}};
\draw[->](C2)--(A2)node [midway , fill=white]{$\trictop^{-1}$};
\draw[->](A2)--(A3)node [midway , fill=white]{\ref{fig:dualizable_1_cell_symmetry_2}};
\end{tikzpicture}
}
\caption{$g$}\label{petit_g}
\end{subfigure}
\caption{Maps}
\end{figure}

\begin{proof}It is enough to prove the underived version of this result.  So in this proof, all homs and tensors are underived.   Since it seems simplest, we prove this by chasing elements, but this argument could be replaced by one using adjunctions.

The image of the tuple $(m,\delta,j,k)$ under the left composite is 
\[g(m,f'(\delta,j))(k).\]

For $\alpha \colon N\to M$ and $\beta\colon \rdual{J}\to K$, the image under the right composite is the homomorphism 
\begin{align*}k&\xto{\eta_J}J\otimes_R\rdual{J}
\xto {\eta_N\otimes \id\otimes \id} N\otimes_R \rdual{N}\otimes_k J\otimes_R\rdual{J}
\xto{\id\otimes f'\otimes \id} N\otimes_R L\otimes_R\rdual{J}
\\&
\xto{\alpha \otimes \id\otimes \id}M\otimes_R L\otimes_R\rdual{J}
\xto{g\otimes \id}\rdual{K}\otimes_R\rdual{J}
\xto{\id\otimes \beta}\rdual{K}\otimes_R K
\xto{\epsilon_K}k
\end{align*}
We can take $\eta_J$ to be given by linearly extending $1\mapsto \sum_i (j_i,j_i^*)$ and 
$\eta_N$ to be given by linearly extending $1\mapsto \sum_{\ell}(n_\ell,n_\ell^*)$.
Then the image of $1$ under the composite above is 
\begin{align*}
1&\mapsto \sum_i (j_i,j_i^*)\mapsto \sum_{i,\ell}(n_\ell,n_\ell^*, j_i,j_i^*)
\mapsto  \sum_{i,\ell}(n_\ell,f'(n_\ell^*, j_i),j_i^*)
\\&
\mapsto  \sum_{i,\ell}(\alpha(n_\ell),f'(n_\ell^*, j_i),j_i^*)
\mapsto  \sum_{i,\ell}(g(\alpha(n_\ell),f'(n_\ell^*, j_i)),j_i^*)
\\&
\mapsto  \sum_{i,\ell}(g(\alpha(n_\ell),f'(n_\ell^*, j_i)),\beta(j_i^*))
\mapsto  \sum_{i,\ell}g(\alpha(n_\ell),f'(n_\ell^*, j_i))(\beta(j_i^*))
\end{align*}

The image of $(m,\delta,j,k)$ under the top map is $(m\delta(-), j^*(-)k)$, so the image of $(m,\delta,j,k)$ under the top, right, and bottom composite is 
 \[\sum_{i,\ell}g(m\delta(n_\ell),f'(n_\ell^*, j_i))( j^*(j_i^*)k)=
\sum_{i,\ell}g(m,f'(\delta(n_\ell)n_\ell^*, j_i j^*(j_i^*)))(k)=g(m,f'(\delta,j))(k)\]
\end{proof}

The main theorem of \cite{petit} is the following:

\begin{thm}\label{thm:petit}\cite[p. 4] {petit} Let $A$ be a proper, smooth dg algebra, $M \in D_{\Perf}(A)$, $f \in \hom_A(M,M)$ and $N \in D_{\Perf}(A^{\op})$, $g \in \hom_{A^{\op}}(N,N)$.
Then 
\[hh_k(N{\otimes}_A M,g {\otimes}_A  f)=hh_{A^{\op}}(N,g)\cup hh_A(M,f).\]
\end{thm}

\begin{proof}
First complete the translation of Petit's pairing \eqref{petit:pairing} to a map without internal homs by replacing the arbitrary modules and maps in the left composite in \eqref{eq:simp_petit} with those given by the middle column of \cref{notation_translation} and the maps in \cref{petit_f_prime,petit_g}.  
This composite is the left and bottom composite in \cref{fig:petit:final_diagram}.  

The two regions in \cref{fig:petit:final_diagram}  labeled Nat. commute by naturality.  The remaining region is labeled by Def. and commutes by definition of the map in \eqref{eq:comp_1_duals_shad_trace}.  The top and right composite is the pairing defined in this paper and so this completes the identification of Petit's pairing \eqref{petit:pairing}  and the pairing defined in this paper.

Then \cref{thm:petit} follows from  \cref{thm:pairing_var_2}\ref{item:pairing_var_2}. 
\end{proof}
\begin{figure}
\resizebox{.95\textwidth}{!}{
\begin{tikzpicture}[
    mystyle/.style={%
    },
   my style/.style={%
   },
  ]

\node[draw,my style] (A1) at (10, 2){\begin{tikzpicture}
\oc{t1}{\coo{A}}{_1}{1}{1}{0}{\et}
\oc{e2}{\evo{A}}{_2}{3}{1}{0}{\ep}
\gc{g1}{2}{1}{0}
\begin{scope}[on background layer]
\alp{ ({t1}b)\gpb{g1}({e2}t)}
\dlp{({t1}t)\gpt{g1}({e2}b)}
\end{scope}
\oc{e3}{\rdual{\evo{A}}}{_3}{-2}{1}{0}{\epr}
\oc{t4}{\rdual{\coo{A}}}{_4}{0}{1}{0}{\etr}
\gc{g2}{-1}{1}{0}
\begin{scope}[on background layer]
\dlp{({e3}b)\gpb{g2}({t4}t)}
\alp{ ({e3}t)\gpt{g2}({t4}b)}
\end{scope}
\end{tikzpicture}
};

\node[draw,my style] (A1a) at (-5, 2){\begin{tikzpicture}
\oc{t1}{\coo{A}}{_1}{1}{1}{0}{\et}
\oc{e2}{\evo{A}}{_2}{3}{1}{0}{\ep}
\gc{g1}{2}{1}{0}
\begin{scope}[on background layer]
\alp{ ({t1}b)\gpb{g1}({e2}t)}
\dlp{({t1}t)\gpt{g1}({e2}b)}
\end{scope}
\oc{e3}{{\coo{\zdual{A}}}}{_3}{-2}{1}{0}{\et\lt}
\oc{t4}{{\evo{\zdual{A}}}}{_4}{0}{1}{0}{\ep\lt}
\gc{g2}{-1}{1}{0}
\begin{scope}[on background layer]
\dlp{({e3}b)\gpb{g2}({t4}t)}
\alp{ ({e3}t)\gpt{g2}({t4}b)}
\end{scope}
\end{tikzpicture}
};

\node[draw,my style] (B1) at (-5, -2){\begin{tikzpicture}
\oc{t1}{\coo{A}}{_1}{1}{-2}{-3}{\et}
\oc{e2}{\evo{A}}{_2}{3}{-2}{-3}{\ep}
\gc{g1}{2}{-2}{-3}
\begin{scope}[on background layer]
\alp{ ({t1}b)\gpb{g1}({e2}t)}
\dlp{({t1}t)\gpt{g1}({e2}b)}
\end{scope}
\oc{e3}{{\coo{\zdual{A}}}}{_3}{-3}{0}{-3}{\et\lt}
\oc{t4}{{\evo{\zdual{A}}}}{_4}{0}{0}{-1}{\ep\lt}
\gc{g2}{-2}{0}{-3}

\oc{e5}{\rdual{\evo{A}}}{_5}{-1}{-1}{-2}{\epr}
\oc{t6}{\rdual{\coo{A}}}{_6}{0}{-2}{-3}{\etr}
\gc{g3}{-2}{-1}{-2}
\gc{g4}{-1}{-2}{-3}

\begin{scope}[on background layer]
\dlp{({e3}b)\gpb{g2}({t4}t)}
\alp{ ({e3}t)\gpt{g2}({t6}b)}

\alp{ ({e5}t)--({t4}b)}
\dlp{({e5}b)--({t6}t)}
\end{scope}
\end{tikzpicture}
};

\node[draw,my style] (B2) at (3, -2){\begin{tikzpicture}
\oc{t1}{\coo{A}}{_1}{1}{-2}{-3}{\et}
\oc{e2}{\evo{A}}{_2}{3}{-2}{-3}{\ep}
\gc{g1}{2}{-2}{-3}
\begin{scope}[on background layer]
\alp{ ({t1}b)\gpb{g1}({e2}t)}
\dlp{({t1}t)\gpt{g1}({e2}b)}
\end{scope}
\oc{e3}{{\coo{\zdual{A}}}}{_3}{-3}{0}{-1}{\et\lt}
\oc{t4}{{\evo{\zdual{A}}}}{_4}{-2}{-1}{-2}{\ep\lt}
\gc{g2}{-1}{0}{-3}

\oc{e5}{\rdual{\evo{A}}}{_5}{-3}{-2}{-3}{\epr}
\oc{t6}{\rdual{\coo{A}}}{_6}{0}{-0}{-3}{\etr}
\gc{g3}{-2}{-1}{-2}
\gc{g4}{-1}{-2}{-3}

\begin{scope}[on background layer]
\dlp{({e3}b)--%
({t4}t)}

\alp{ ({e5}t)--({t4}b)}
\dlp{({e5}b)\gpb{g2}({t6}t)}
\alp{ ({e3}t)\gpt{g2}
({t6}b)}
\end{scope}
\end{tikzpicture}
};

\node[draw,my style] (C2) at (3, -7){\begin{tikzpicture}
\oc{e2}{\evo{A}}{_2}{2}{-0}{-3}{\ep}
\gc{g1}{1}{-0}{-3}

\oc{e3}{{\coo{\zdual{A}}}}{_3}{-3}{0}{-1}{\et\lt}
\oc{t4}{{\evo{\zdual{A}}}}{_4}{-2}{-1}{-2}{\ep\lt}
\gc{g2}{-1}{0}{-3}

\oc{e5}{\rdual{\evo{A}}}{_5}{-3}{-2}{-3}{\epr}
\gc{g3}{-2}{-1}{-2}
\gc{g4}{-1}{-2}{-3}

\begin{scope}[on background layer]
\dlp{({e3}b)--%
({t4}t)}

\alp{
(0*\d,-3*\h)\gpb{g1}({e2}t)}

\alp{ ({e5}t)--({t4}b)}
\dlp{({e5}b)\gpb{g2}(0*\d,0*\h)\gpt{g1}({e2}b)}
\alp{ ({e3}t)\gpt{g2}
(0*\d,-3*\h)
}
\end{scope}
\end{tikzpicture}
};

\node[draw,my style] (C1) at (-5, -7){\begin{tikzpicture}
\oc{e2}{\evo{A}}{_2}{1}{-2}{-3}{\ep}
\gc{g1}{0}{-2}{-3}

\oc{e3}{{\coo{\zdual{A}}}}{_3}{-3}{0}{-3}{\et\lt}
\oc{t4}{{\evo{\zdual{A}}}}{_4}{0}{0}{-1}{\ep\lt}
\gc{g2}{-2}{0}{-3}

\oc{e5}{\rdual{\evo{A}}}{_5}{-1}{-1}{-2}{\epr}
\gc{g3}{-2}{-1}{-2}
\gc{g4}{-1}{-2}{-3}

\begin{scope}[on background layer]
\dlp{({e3}b)\gpb{g2}({t4}t)}
\alp{ ({e3}t)\gpt{g2}(-1*\d,-3*\h)\gpb{g1}({e2}t)}

\alp{ ({e5}t)--({t4}b)}
\dlp{({e5}b)\gpt{g1}({e2}b)}
\end{scope}
\end{tikzpicture}
};

\node[draw,my style] (A2) at (10, -7){\begin{tikzpicture}
\oc{e2}{\evo{A}}{_2}{1}{1}{0}{\ep}
\gc{g1}{0}{1}{0}

\oc{e3}{\rdual{\evo{A}}}{_3}{-2}{1}{0}{\epr}
\gc{g2}{-1}{1}{0}
\begin{scope}[on background layer]
\dlp{ ({e3}b)\gpb{g2}(-.25,\h)
}

\alp{ ({e3}t)\gpt{g2}(-.25,0)\gpb{g1}({e2}t)}
\dlp{
(-.25,\h)\gpt{g1}({e2}b)}
\end{scope}
\end{tikzpicture}
};

\node[draw,my style] (A3) at (15, -7){\begin{tikzpicture}
\oc{e2}{\evo{A}}{_2}{0}{1}{0}{\ep}

\oc{e3}{\rdual{\evo{A}}}{_3}{-2}{1}{0}{\epr}
\begin{scope}[on background layer]
\dlp{({e3}b)--({e2}b)}
\alp{ ({e3}t)--({e2}t)}
\end{scope}
\end{tikzpicture}
};

\node[draw,my style] (A4) at (20,-7){\begin{tikzpicture}
		\node (ta) at (2,-1){};
		\node (ta) at (4,-2){};
		\end{tikzpicture}};

\draw[->](A1a)--(B1)node [midway , fill=white] {$\tridetop$};
\draw[->](A1a)--(A1)node [midway , fill=white] {$\text{\eqref{eq:comp_1_duals_shad_trace}}\otimes \id$};

\draw[->](B1)--(C1)node [midway , fill=white]{$P_{\coo{A}}$};
\draw[->](B2)--(C2)node [midway , fill=white]{$P_{\coo{A}}$};

\draw[<-](B1)--(B2)node [midway , fill=white]{\eqref{ex:pseduonaturality_gamma}};

\draw[->](A1)--(B2)node [midway , fill=white]{$\trictop$};
\draw[<-](C1)--(C2)node [midway , fill=white]{\eqref{ex:pseduonaturality_gamma}};
\draw[->](A2)--(C2)node [midway , fill=white]{$\trictop$};

\draw[->](A1)--(A2)node [midway , fill=white] {$P_{\coo{A}}$};
\draw[->](A2)--(A3)node [midway , fill=white] {\ref{fig:dualizable_1_cell_symmetry_2}};

\draw[->](A3)--(A4)node [midway , fill=white] {$L_{\evo{A}}$};

\node at  (barycentric cs:C1=1,C2=1,B1=1,B2=1){Nat.};
\node at  (barycentric cs:A1=1,C2=0,A2=1,B2=1){Nat.};
\node at  (barycentric cs:A1a=1,B2=1){Def.};
\end{tikzpicture}
}
\caption{Comparison of pairings}\label{fig:petit:final_diagram}
\end{figure}
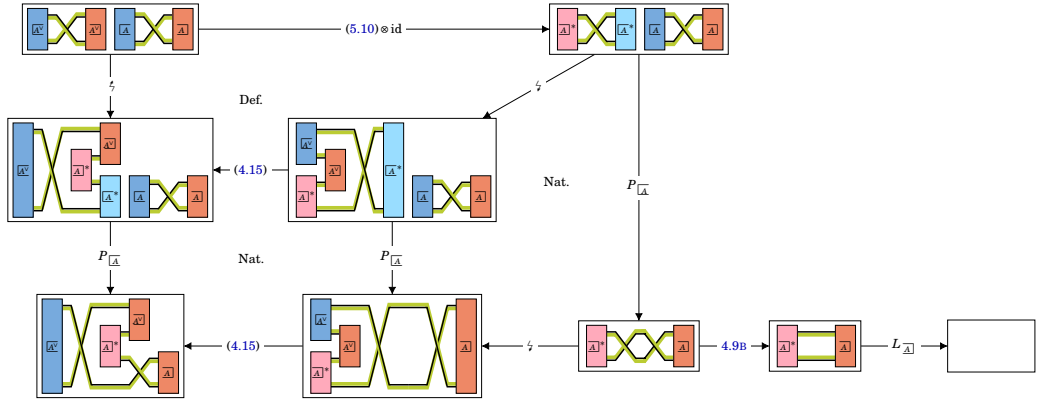

\bibliographystyle{amsalpha2}
\bibliography{fixed_points}
\end{document}